\documentclass[11pt]{amsart}

\usepackage{amsmath,enumerate,amscd,amsfonts,amsthm,amssymb,amscd,mathrsfs}
\usepackage[colorlinks,linkcolor=blue,citecolor=red]{hyperref}
\hypersetup{CJKbookmarks}
\usepackage{appendix}
\usepackage{bm}
\usepackage{cite}
\usepackage{geometry}
\usepackage{graphicx}
\usepackage{setspace}
\usepackage{stmaryrd}
\usepackage{url}
\allowdisplaybreaks[4]

\newtheorem{theorem}{Theorem}[section]
\newtheorem{lemma}[theorem]{Lemma}
\newtheorem{proposition}[theorem]{Proposition}
\newtheorem{corollary}[theorem]{Corollary}

\theoremstyle{definition}
\newtheorem{definition}[theorem]{Definition}
\newtheorem{remark}[theorem]{Remark}

\newtheorem{problem}{Problem}

\DeclareMathOperator{\Op}{\mathbf{Op}}
\DeclareMathOperator{\diff}{d}

\numberwithin{equation}{section}

\newcommand{\Dio[1]}{\mathfrak{D}({#1})}
\newcommand{\T}{\mathsf{T}}
\newcommand{\PM}{\mathrm{PM}}
\newcommand{\PL}{\mathrm{PL}}
\newcommand{\PLR}{\mathrm{PLR}}
\newcommand{\CM}{\mathrm{CM}}

\newcommand{\Con}{\mathrm{Con}}
\newcommand{\R}{\mathcal{R}}

\newcommand{\even}{\mathrm{even}}
\newcommand{\odd}{\mathrm{odd}}

\def\be{\begin{equation}}
\def\ee{\end{equation}}
\def\le{\leqslant}
\def\xZ{\mathbb{Z}}
\def\xC{\mathbb{C}}
\def\xR{\mathbb{R}}
\def\xN{\mathbb{N}}
\def\xT{\mathbb{T}}

\def\D#1{({#1}\cdot\partial)}

\def\Id{\mathrm{Id}}
\def\Avg{\mathrm{Avg}}

\def\dx{\diff \! x}

\def\dt{\diff \! t}
\def\dtau{\diff \! \tau}
\def\be{\begin{equation}}
\def\ee{\end{equation}}
\def\e{\eqref}
\def\defn{\mathrel{:=}}
\def\la{\left\vert}
\def\lA{\left\Vert}
\def\bla{\big\vert}

\def\le{\leq}
\def\les{\lesssim}

\def\ra{\right\vert}
\def\rA{\right\Vert}
\def\bra{\big\vert}

\def\xC{\mathbb{C}}

\def\xT{\mathbb{T}}
\def\xZ{\mathbb{Z}}

\def\gau{\mathrm{left}}
\def\droi{\mathrm{right}}

\title{Paracomposition Operators and Paradifferential Reducibility}

\author{Thomas Alazard}
\address{CNRS - Centre de Mathématiques Laurent Schwartz, {\'E}cole Polytechnique, 
Institut Polytechnique de Paris}
\email{thomas.alazard@polytechnique.edu}

\author{Chengyang Shao}
\address{Department of Mathematics, University of Chicago}
\email{shaoc@uchicago.edu}

\begin{document}

\begin{spacing}{1.1}
\begin{abstract}
Reducibility methods, aiming to simplify systems by conjugating them to those with constant coefficients, are crucial for studying the existence of quasiperiodic solutions. In KAM theory for PDEs, these methods help address the invertibility of linearized operators that arise in a Nash-Moser/KAM type scheme. The goal of this paper is to prove paradifferential reducibility results, enabling the reduction of nonlinear equations themselves, rather than just their linearizations, to constant coefficient form, modulo smoothing terms. As an initial application, we demonstrate the existence of quasiperiodic solutions for certain hyperbolic systems. Despite the small denominator problem, our proof does not rely on traditional Nash-Moser/KAM-type schemes, but instead on the Banach fixed point theorem. To achieve this, we develop two key toolsets. The first focuses on the calculus of paracomposition operators introduced by Alinhac, interpreted as the flow map of a paraproduct vector field. We refine this approach to establish new estimates that precisely capture the dependence on the diffeomorphism in question. The second toolset addresses two classical reducibility problems, one for matrix differential operators and the other for nearly parallel vector fields on torus. We resolve these problems by paralinearizing the conjugacy equation and exploiting, at the paradifferential level, the specific algebraic structure of conjugacy problems, akin to Zehnder's approximate Nash-Moser approach.
\end{abstract}

\maketitle

\section{Introduction}
\subsection{Main Results and Organization of the Paper}
In this paper, we aim to demonstrate how paradifferential calculus can be used to construct quasiperiodic solutions to partial differential equations (``KAM for PDEs") without relying on Nash-Moser/KAM-type iterative schemes, extending the ``KAM via standard fixed point theorems" argument in \cite{AS2023}. The theorem that we prove to realize this general idea is the following:
\begin{theorem}\label{1NLPHyp}
Fix natural numbers $n,N$. Let $f:\xT^n\to\xR^N$ be a smooth odd function. Let $X:\xT_x^n\times\xR_z^N\to\xR^n$ be a smooth function that is even in $x\in\xT^n$. Let $F:\xT_x^n\times\xR_z^N\to\xR^N$ be a smooth function, odd in $x\in\xT^n$ and vanishing when $z=0$. Then there exists an index $s>n/2+1$ and a constant $\varepsilon_*=\varepsilon_*(f,X,F)>0$ with the following properties. For $\varepsilon<\varepsilon_*$, there is a Cantor type subset $\mathfrak{O}_\varepsilon\subset\xR^n$, so that given $\omega\in\mathfrak{O}_\varepsilon$, the quasilinear hyperbolic system
\begin{equation}\label{1NLPEQ}
(\omega\cdot\nabla_x+\varepsilon X(x,u)\cdot\nabla_x)u
+\varepsilon F(x,u)=\varepsilon f
\end{equation}
has an $H^s$ even solution $u=u(\omega):\xT^n\to\xR^N$. Moreover, the set $\mathfrak{O}_\varepsilon$ is asymptotically ample in measure when $\varepsilon\to0$: fixing any $R>0$ and $0<a<1/3$, there holds
$$
|\bar{B}(0,R)\setminus\mathfrak{O}_\varepsilon|\leq C_{n,a}R^n\varepsilon^a.
$$
\end{theorem}

Theorem \ref{1NLPHyp} immediately yields the following corollary:


\begin{corollary}\label{NLPHypCoro}
Fix natural numbers $\nu,d,N$. Let $f:\xT_\varphi^\nu\times\xT_x^d\to\xR$ be a smooth odd function. Let $X:\xT_\varphi^\nu\times\xT_x^d\times\xR_z^N\to\xR^n$ be a smooth function that is even in $(\varphi,x)\in\xT_\varphi^\nu\times\xT_x^d$. Let $F:\xT_\varphi^\nu\times\xT_x^d\times\xR_z^N\to\xR^N$ be a smooth function, odd in $(\varphi,x)\in\xT_\varphi^\nu\times\xT_x^d$, vanishing when $z=0$. 

Then there is an index $s>(\nu+d)/2+1$ and a constant $\varepsilon_*=\varepsilon_*(f,X,F)>0$ with the following property. For $\varepsilon<\varepsilon_*$, there is a Cantor type subset $\mathfrak{O}_\varepsilon\subset\xR^{\nu+d}$ which is asymptotically ample when $\varepsilon\to0$, so that given $(\omega',\omega'')\in\mathfrak{O}_\varepsilon$, the forced quasilinear hyperbolic system
\begin{equation}\label{NLPEQ'}
\partial_tu+\big(\omega''\cdot\nabla_x+\varepsilon X(\omega' t,x,u)\cdot\nabla_x\big)u
+\varepsilon F(\omega' t,x,u)=\varepsilon f(\omega' t,x)
\end{equation}
has an $H^s$ quasiperiodic solution $u(t,x)$ with time frequency $\omega'$.
\end{corollary}
\begin{proof}
In fact, denoting $\omega=(\omega',\omega'')$, the problem is equivalent to finding a Sobolev solution $v=v(\varphi,x)$ on $\xT^\nu\times\xT^d$ of
$$
\big(\omega\cdot\nabla_{\varphi,x}+\varepsilon X(\varphi,x,v)\cdot\nabla_x\big)v
+\varepsilon F(\varphi,x,v)=f(\varphi,x),
$$
a special case of the general form (\ref{1NLPEQ}). The existence of such $v$ is guaranteed by Theorem \ref{1NLPHyp}. The solution $u(t,x)$ of (\ref{NLPEQ'}) is then simply taken as $u(t,x)=v(\omega't,x)$. 
\end{proof}
\begin{remark}
A special case of (\ref{NLPEQ'}) with $d=N=1$ is the \emph{forced Burger's equation}
$$
\partial_tu+u\partial_xu=\varepsilon^2 f(t,x).
$$
Replacing $u$ by $1+\varepsilon u$, this reduces to the form indicated in (\ref{NLPEQ'}). Existence and dynamical implications of time-periodic \emph{viscous} solutions with time-periodic forcing $f$ was addressed by Jauslin, Kreiss and Moser~\cite{JKM1998} as well as 
Weinan E \cite{Weinan1999} in the context of Hamilton-Jacobi equation. In general, being a quasilinear perturbation of the linear transport system
$$
\partial_tu+(\omega''\cdot\nabla_x)u=0,
$$
the equation (\ref{NLPEQ'}) covers several cases of physical interest. Another example that resembles (\ref{NLPEQ'}) is the forced incompressible Euler equation. Baldi and Montalto recently proved in \cite{BalMon2021} that quasiperiodically forced incompressible Euler equation admits quasiperiodic solutions. However, our result differs from \cite{BalMon2021} since the characteristic velocity $X(x,u)$ in (\ref{NLPEQ'}) is not required to be divergence free. Consequently, quasiperiodic solutions of (\ref{NLPEQ'}) has different frequency-amplitude relation compared to that of incompressible Euler equation, as will be shown in Section \ref{sec:7}. In fact, for (\ref{NLPEQ'}), the eigenfrequencies are of the form $(\omega+h)\cdot\xi$, with $\xi\in\xZ^n$ and some first-order amendment $h\cdot\xi$ depending on the solution, while for the incompressible Euler equation there is no amendment $h\cdot\xi$ but only $O(1)$ ones. It will be clearly shown in the proof that the set $\mathfrak{O}_\varepsilon$ is indeed a Lipschitz distortion of the set of Diophantine vectors; see (\ref{diopc}) and (\ref{O_Set}).
\end{remark}

The paper is divided into two parts. Part \ref{Part1}, including Sections \ref{sec:2} through \ref{sec:4}, studies 
the \emph{paracomposition operators} introduced by Alinhac \cite{Alinhac1986}, along with a symbolic calculus to handle them. These paracomposition operators will be used to 
quantify the diffeomorphism that straightens the vector field $\omega+\varepsilon X(x,u)$ on $\xT^n$. We revisit the Egorov's approach, which 
consists in viewing a paracomposition operator as the flow map corresponding to a paraproduct vector field. We prove new estimates that clarify the dependence on the diffeomorphism in question, which is an essential ingredient to prove contraction estimates and hence to implement a Banach fixed-point argument. Furthermore, this alternative definition, compared to the one originally proposed by \cite{Alinhac1986} (and further develop in \cite{Taylor2000,Said2023}), allows us to show that paradifferential and paracomposition operators share the same \emph{algebraic structure} as standard pseudo-differential operators ($\Psi\mathrm{DOs}$) and Fourier integral operators (FIOs), while also clearly tracking potential loss of regularity. The first part is self-contained and could thus serve as a toolbox for future works. We outline the basic idea of this part in Subsection \ref{sec:1.2}.

Part \ref{Part2} illustrates how the concepts from Part \ref{Part1} can be applied to study partial differential equations involving small denominators, known as ``KAM for PDEs". Starting from Section \ref{sec:5}, it contributes the major novelties of this paper. Sections \ref{sec:5} and \ref{sec:6} provide two reducibility theorems for certain linear operators. 
We will study reducibility results for $(i)$ matrix operators of the 
form $\D{\omega} + A$ and $(ii)$ nearly parallel vector fields on the torus, in order to find a diffeomorphism that straightens the vector field $\omega+ X$ on $\xT^n$. Such reducibility results 
are typically achieved through a similar Nash-Moser/KAM scheme. The first novelty here is that we proceed with the Banach fixed-point theorem only, providing ``KAM via standard fixed point" style proof. 

Theorem \ref{1NLPHyp} is then proved in Section \ref{sec:7} as an initial application of Part \ref{Part1} and Sections \ref{sec:5}-\ref{sec:6}, again through a Banach fixed point scheme. The second and most notable novelty of the paper is that the methodology reduces the \emph{nonlinear equation itself} to constant coefficient form modulo smoothing terms, rather than just its linearization. We summarize the general idea behind these two novelties as \emph{paradifferential reducibility}. A brief description of the core idea is outlined in Subsection \ref{sec:1.3}-\ref{sec:1.5}.

\subsection{Paracomposition and ``Parachange of Variable"}\label{sec:1.2}
Here we briefly outline the intuition behind the toolbox of paracompositions. 

Paradifferential operators were first introduced by Bony \cite{Bony1981} as an extension of \emph{paraproduct decomposition}, a widely used harmonic analysis technique in the study for nonlinear dispersive or hyperbolic 
PDEs. Given two functions $a,u$, with the aid of Littlewood-Paley decomposition $\Id=\sum_{j\geq0}\Delta_j$ (see Section \ref{sec:2} for details), one defines the paraproduct operator
$$
T_au=\sum_{k\geq0}S_{k-3}a\cdot\Delta_ku,\quad S_{k-3}=\sum_{0\le j\le k-3}\Delta_j,
$$
and obtains the decomposition
$$
au=T_au+T_ua+R_\PM(a,u).
$$
The remainder $R_\PM(a,u)$ collects ``high-high frequency interactions" and has better regularity, while the regularity of $T_au$, being the same as $u$, relies on very mild regularity assumption of $a$. For application of this decomposition to nonlinear dispersive or hyperbolic equations, see for example \cite{SautTemam1976,GR1987,KPV1993,Kato1995,Staffilani1995,Taylor2000,Tao,MePise}.

Bony \cite{Bony1981} noticed the feature of ``capturing the lowest regularity" in paraproduct decomposition and generalized it to the decomposition $F(u)=T_{F'(u)}+\R_F(u)$ of a nonlinear expression $F(u)$. It turns out that the \emph{paralinearization} $T_{F'(u)}u$ captures most of the irregularity of $F(u)$, and the remainder $\R_F(u)$ gains more regularity. On the other hand, Bony considered paraproduct operators as special \emph{$\Psi\mathrm{DOs}$}. These operators are of ``forbidden type (1,1)", a terminology used in \cite{SteMur1993} and \cite{Hormander1988} to refer to their ``bad" properties compared to usual type (1,0) \emph{$\Psi\mathrm{DOs}$}. However, Bony discovered that they still enjoy a symbolic calculus as classical $\Psi\text{DOs}$ of type (1,0)  do. They are thus named as \emph{paradifferential operators}. This is a key advantage of paradifferential operators: algebraically they can be manipulated by the usual symbolic calculus, while analytically they keep clear track of any possible regularity loss. These results are summarized in Section \ref{sec:2}, specifically Theorem \ref{2theo:sc0}-\ref{2theo:sc2} and Theorem \ref{2ParaLin}.

Composition with appropriate diffeomorphisms sometimes could drastically simplify differential operators under consideration. However, composition with 
a diffeomorphism $\chi$ is \emph{not} a pseudo-differential 
operator but rather a \emph{Fourier integral operator}:
$$
(f\circ\chi)(x)=\sum_{\xi\in\xZ^n}\hat f(\xi)e^{i\chi(x)\cdot\xi}.
$$
For $\Psi\text{DOs}$ of type (1,0), there is an explicit change-of-variable formula resembling the usual chain rule; see for example, Theorem 2.1.2 of H\"{o}rmander's classical work \cite{Hormander1971}, or Proposition 7.1 of Chapter I of \cite{AG} (Theorem \ref{Change(1,0)} in the present paper). It turns out there is also a counterpart of $f\mapsto f\circ\chi$ for 
paradifferential operators, known as \emph{paracomposition} corresponding to $\chi$. Applying paracomposition operator may be legitimately referred as \emph{parachange of variable}.

The idea of introducing paracomposition operator was first due to Alinhac \cite{Alinhac1986}. In studying the nonlinear composition $f\circ\chi$ where both $f$ and $\chi$ are of limited regularity, Alinhac made the further decomposition from $f\circ\chi-T_{f'\circ\chi}\chi$ to capture the part that is ``the most irregular with respect to $f$". In the discussion for ``good unknown" (\emph{bonne inconnue} introduced in \cite{Alinhac1988,Alinhac1989}) to remove the superficial loss of regularity due to coordinate change in free boundary problems, this decomposition naturally appears. This leads to the original definition of paracomposition operator $\chi^\star$ corresponding to $\chi$ (see Definition 2.2.1 of \cite{Alinhac1986}):
\begin{equation}\label{FIO}
(\chi^\star f)(x)=\sum_{j\geq0}(S_{j+N}-S_{j-N})\big((\Delta_jf)\circ\chi(x)\big),
\end{equation}
where the index $N$ depends only on the Lipschitz constant of $\chi$. This expression can be considered as a ``Fourier integral operator (FIO) of forbidden type".

An important feature is that the definition of $\chi^\star$ only requires $\chi$ to be a bi-Lipschitz homeomorphism, but there always holds $\|\chi^\star f\|_{H^s}\lesssim_{s,\chi}\|f\|_{H^s}$ for any $s\in\xR$. Alinhac then proved a refinement of the paralinearization formula (see Lemma 3.1 of \cite{Alinhac1986}):
\begin{equation}\label{AlinhacParalin}
f\circ\chi=T_{f'\circ\chi}\chi+\chi^\star f+\text{smoother remainder},
\end{equation}
together with a parachange of variable formula when conjugating a paradifferential operator $T_a$ with $\chi^\star$. Consequently, $\chi^\star$ plays the role of Fourier integral operators in the realm of paradifferential calculus. See the Appendix of Chapter 2 of \cite{Taylor2000} for detailed discussion on this construction, and the paper \cite{Said2023} for discussion on the parachange-of-variable formula. 

The definition (\ref{FIO}) suffers from being too complicated for some practical manipulation. For example, the dependence of $\chi^\star$ on $\chi$ is not easily read off from it. An alternative definition is thus more feasible. It is well known that composition with a diffeomorphism can be considered as the time 1 flow map corresponding to a vector field, which is the essence of the \emph{Egorov theorem} in pseudo-differential calculus (see for example, Section 7.8 of the standard textbook \cite{Taylor2013} or \cite{MR1218524} for applications to paradifferential operators). Consequently, it is natural to consider the time 1 flow corresponding to a ``\emph{para}differential vector field". 

This is exactly how the alternative definition for paracomposition operator could be formulated. Corresponding to a diffeomorphism $\chi$, one may determine a deformation vector field $X=X(\tau,x)$ such that $f\to f\circ\chi$ is the time 1 flow map for the linear transport equation
$$
\partial_\tau u+X\cdot\nabla_xu=0,\quad u\big|_{\tau=0}=f
\,\Longrightarrow\,
u\big|_{\tau=1}=f\circ\chi.
$$
The natural generalization is, of course, the \emph{paratransport equation}
$$
\partial_\tau w+T_X\cdot\nabla_xw=0,\quad w\big|_{\tau=0}=f,
$$
and it is feasible to \emph{define} the action $\chi^\star f$ exactly as $w\big|_{\tau=1}$. An immediate advantage is that boundedness of the operator $\chi^\star$ within Sobolev spaces follows almost immediately. Moreover, compared to the original Fourier integral definition (\ref{FIO}), it is much easier to decide the conjugation of a paradifferential operator $T_p$ using $\chi^\star$: 
\begin{equation}\label{PCConjRough}
\chi^\star T_p=T_q\chi^\star+\text{lower order operators},
\end{equation}
where the principal symbol of $q$ is identical to that in the usual change-of-variable formula:
$$
q(x,\xi)=p\left( \chi(x),\chi'(x)^{-\T}\xi\right).
$$
Being exactly as predicted in the classical Egorov theorem, this once again illustrates that paradifferential operators share the same algebraic structures as usual (1,0) $\Psi\text{DOs}$.

The flow map of paratransport equation was already implemented in several scenarios, for example in \cite{BD2018,BFF2021,BFP2023,BMM2024} to obtain extended lifespan estimate for water waves. The flow map operator $G(U;t)$ (where $U$ determines $\chi$) defined in these references relied on \emph{Bony-Weyl quantization}, which bypassed manipulation with lower order symbols and sufficed to simplify the principal symbols therein. However, these are insufficient when lower-order symbols do play crucial role. Besides, the arguments in these references provide no explicit analysis on how the operator depend on the diffeomorphism $\chi$.

To quantitatively analyze the dependence of the paracomposition operator on the diffeomorphism, we study the paratransport equation more carefully in Part \ref{Part1}. We directly work within the usual Bony-Kohn-Nirenberg quantization instead of the Bony-Weyl quantization, choosing a deformation vector field $X(\tau,x)$ that gains almost 1 derivative; see Definition \ref{Deformation} and \ref{DefParacomposition}. By analyzing the paratransport equation, we obtain continuous dependence of $\chi^\star$ on $\chi$, together with \emph{Lipschitz} dependence of $\chi^\star$ on $\chi$ at the price of losing 1 derivative. Such dependence is stated in Proposition \ref{PCNormCont}. Furthermore, we prove a refined version of Bony's paralinearization formula, Theorem \ref{ParaLinRefined}, which resembles Alinhac's observation (\ref{AlinhacParalin}). This shows that the time 1 flow map alternative definition is essentially the same as the original one (\ref{FIO}). Finally, we prove the parachange-of-variable formula in Theorem \ref{4PCConj}. The statement not only clearly indicates that the lower order symbols are \emph{algebraically} the same as their counterparts in the usual change-of-variable for (1,0) $\Psi\text{DOs}$, but also gives quantitative Lipschitz dependence of the regularizing remainder on the diffeomorphism.

\subsection{``KAM for PDE" and Reducibility}\label{sec:1.3}
In this subsection and the next two, we outline the core idea 
of the proof of Theorem \ref{1NLPEQ} and complement it with some historical context.

Traditionally, equations of the form (\ref{1NLPEQ}) are treated by Nash-Moser type iterations, as the loss of regularity due to small denominators led 
to the belief that transforming them into classical fixed point forms was impossible. The general idea is to investigate the \emph{linearized} equation $L(\omega,u)v=g$ of the nonlinear problem $\mathscr{F}(\omega,u)=0$. In case of equation (\ref{1NLPEQ}), we have
\begin{equation}\label{F(u)}
\mathscr{F}(\omega,u)=(\omega+\varepsilon X(x,u))\cdot \nabla_xu- \varepsilon F(x,u)-\varepsilon f,
\end{equation}
\begin{equation}\label{L(u)}
\begin{aligned}
L(\omega,u)v
&=(\omega+\varepsilon X(x,u))\cdot\nabla_xv
+\varepsilon X'_z(x,u)v- \varepsilon F'_z(x,u)v\\
&=:(\omega+\varepsilon X)\cdot\nabla_xv+\varepsilon Av.
\end{aligned}
\end{equation}
Should the linearized equation $L(\omega,u)v = g$ be solvable and loses only a fixed number of derivatives, one could then formulate a \emph{modified Newtonian scheme} of the form
$$
u_{k+1}=u_k-L(\omega,u_k)^{-1}S_k \mathscr{F}(u_k),
\quad \text{ where }S_k\text{ is a suitable smoothing operator.}
$$
The loss of regularity caused by $L(\omega,u_k)^{-1}$ is compensated by the quadratic convergence of Newtonian schemes, resulting in a sequence converging to a solution $u=u(\omega)$ of $\mathscr{F}(\omega,u) = 0$. Realization of this general idea can be found in several classical literature on ``KAM for PDEs", for example \cite{Wayne1994,Poschel1996,Kuksin1998}. 
Concerning Nash-Moser type schemes, let us also mention \cite{BH2017} where the authors implement some basic orthogonality idea used in paradifferetial calculus (see M\'etivier's Lecture Notes \cite{MePise}), into Nash-Moser-H\"ormander scheme.

For nonlinear PDEs, especially quasilinear ones, inverting the linear operator $L(\omega,u)$ often involves technical challenges of individual interest. Usually $L(\omega,u)$ is a linear differential operator with variable coefficients. One way of showing its global invertibility on torus is to conjugate it with suitable FIOs (in this case, change-of-variables operators) and $\Psi\text{DOs}$ to reduce the coefficients to constants. This idea dates back to the pioneering works of Iooss-Plotnikov-Toland on time-periodic or travelling water waves, for example \cite{PT2001,IPT2005,IP2009}. They introduced the straightening of principal symbols via changes of variables, and the so-called ``descent method" which consist in elliminating the lower order terms one by one, by conjugating with $\Psi\text{DOs}$. Later this was employed to construct time-periodic standing gravity-capillary waves in \cite{AB2015}. 

However, in trying to apply the ideas of Iooss-Plotnikov-Toland to \emph{quasiperiodic} solutions, such as those defined by (\ref{F(u)})-(\ref{L(u)}), equations defining the diffeomorphisms and symbols to realize the ``descent method" already involve small denominators in themselves. This raises the question of the \emph{reducibility} of $L(\omega,u)$. More concretely, for (\ref{L(u)}), conjugating the vector field $\omega+\varepsilon X$ to a parallel (that is  constant) one leads to a nonlinear differential equation of the form
\begin{equation}\label{Eq(theta)}
(\omega\cdot\partial)(\eta-\Id)=\varepsilon X\circ\eta+\text{constant counter-term},
\end{equation}
where $\eta$ is the diffeomorphism to be determined. 
Similarly, the ``descent method" used to conjugate $\omega\cdot\nabla_x+\varepsilon A$ to $\omega\cdot\nabla_x$ results in 
a linear differential equation of the form
\begin{equation}\label{Eq(U)}
(\omega\cdot\partial)U=\varepsilon A\cdot(I_N+U),
\end{equation}
where $A$ is the $N\times N$ matrix-valued function from (\ref{L(u)}), and $U$ is the $N\times N$ unknown matrix. 

Both equations (\ref{Eq(theta)}) and (\ref{Eq(U)}) are \emph{dynamical conjugacy equations} involving loss of regularity caused by small denominators. Traditionally, such reducibility is achieved through another Nash-Moser/KAM scheme, precisely the focus of traditional KAM theory. It is therefore not surprising that KAM reducibility results concerning variants of (\ref{Eq(theta)}) and (\ref{Eq(U)}) are extensively discussed in the literature. In particular, (\ref{Eq(theta)}) is often treated as a model problem in studies of Nash-Moser techniques and classical KAM theory, as seen in, for example, \cite{Moser19662, Hormander1977, AG, Poschel2011}. For quantitative reducibility results, see, for example, \cite{FGMP2019} for (\ref{Eq(theta)}) and \cite{BLM2019} for (\ref{Eq(theta)}) with unbounded perturbations. There have also been several reducibility results for quasiperiodic matrix linear differential operators, such as \cite{Eliasson1992,JS1992,JRV1997,Krikorian1999}, as well as several results related to infinite-dimensional systems, including ``KAM for PDEs'' \cite{BG2001,BGMR2017,Bambusi2018,FGP2019}.

To summarize, except in the case of finding time-periodic solutions, the construction of quasiperiodic solutions to the nonlinear problem $\mathscr{F}(\omega,u)=0$ has relied on a ``Nash-Moser within Nash-Moser'' scheme. This is a key feature found in much of the literature on ``KAM for PDEs''; see for example, \cite{FP2015,BBHM-2018,BM2020,BalMon2021,HHM}, where additional ``Melnikov conditions" are also required to ensure the KAM-type reducibility of the linearized operators. 

\subsection{Paradifferential Reducibility for Conjugacy Problems}\label{sec:1.4}
Following the general idea in \cite{AS2023}, we provide ``fixed point style" proofs to the following reducibility theorems for conjugacy problems (\ref{Eq(theta)})(\ref{Eq(U)}) in Section \ref{sec:5}-\ref{sec:6}. Precise statements are given as Theorem \ref{MatRed} and \ref{6VectRed}. 

\begin{theorem}[Reducibility of (\ref{Eq(U)}), informal version]\label{MatRedInformal}
Let $\omega\in\xR^n$ be a Diophantine frequency. Suppose $A:\xT^n\to\mathbf{M}_{N}(\xR)$ is an odd matrix-valued function, with small magnitude in some fixed Sobolev space $H^{s_*}$, where $s_*$ depends on the dimension $n$ and the Diophantine index of $\omega$. Then (\ref{Eq(U)}) has a solution $U=U(\omega,A)$. If in addition $A\in H^{s}$ with $s\geq s_*$, then $U(\omega,A)$ is tame Lipschitz in $(\omega,A)$.
\end{theorem}

\begin{theorem}[Reducibility of (\ref{Eq(theta)}), informal version]\label{VectRedInformal}
Given a Diophantine index $\tau$, if $X$ is a vector field close to zero in some fixed Sobolev space $H^{s_*}$, where $s_*$ depends on the dimension $n$ and $\tau$, then there exists a set $\mathfrak{O}$ of frequencies, an amendment $h(\omega,X)\in\xR^n$ and a diffeomorphism $\eta=\eta(\omega,X)$, such that the vector field $\omega+X$ is conjugated to the parallel vector field $\omega+h(\omega,X)$. The amendment is Lipschitz with respect to $(\omega,X)\in\mathfrak{O}\times H^{s_*}$. Furthermore, if in addition $X\in H^{s}$ with $s\geq s_*$, then $\eta(\omega,X)$ is tame Lipschitz in $(\omega,X)$.
\end{theorem}

Theorem \ref{MatRedInformal} can be regarded as a ``matrix Kuksin lemma". See \cite{Kuksin1998} or Chapter VI of \cite{Kappeler-Poeschel} for detailed discussion. When the size of matrix $N=1$ the solution is algebraically straightforward to obtain. But for $N>1$, the non-commutativity of matrices becomes a major difficulty. This result appeared as Proposition 5.2 in \cite{BalMon2021} but was proved by a KAM type iterative scheme. On the other hand, Theorem \ref{VectRedInformal} is essentially the same as the main result in \cite{FGMP2019}, although we do not attempt to optimize the regularity of functions involved. As pointed out in \cite{FGMP2019}, the essence is to obtain \emph{tame estimates} of the straightening diffeomorphism $\eta$ in terms of the frequency $\omega$ and perturbation $X$. 

Our strategy to prove Theorem \ref{MatRedInformal} and \ref{VectRedInformal} differs from most of the existing literature on KAM reducibility. The general idea of \cite{AS2023} is employed to transform the conjugacy equations (\ref{Eq(U)})(\ref{Eq(theta)}) into fixed point form that balance out the regularity loss. We first paralinearize the conjugacy equations (\ref{Eq(U)})(\ref{Eq(theta)}) and then exploit the algebraic structure specific to conjugacy problems, in the spirit of Zehnder’s \emph{approximate right inverse} argument \cite{Zehnder1975}. This leads to the \emph{parahomological equation}. Being a quadratic perturbation of the usual linear homological equation, it is then transformed into a convenient Banach fixed point form. The equivalence between the parahomological equation and the original conjugacy problem is established using a straightforward Neumann series argument. The advantage of Banach fixed point form is that parameter dependence can be easily recognized. This could be considered as another realization of ``KAM via standard fixed point theorems", or using the terminology of the present paper, \emph{paradifferential reducibility}. 

The proof of both Theorem \ref{MatRedInformal} and \ref{VectRedInformal} involves the Fourier multiplier $(\omega\cdot\partial)^{-1}$, well-defined only for Diophantine $\omega$, which is a meager set (in the sense of Baire). It will be convenient to extend it to a Fourier multiplier, well-defined for \emph{all} $\omega$, coinciding with $(\omega\cdot\partial)^{-1}$ for Diophantine $\omega$; see Lemma \ref{Dio_Ext}. This enables us to solve the parahomological equations via Banach fixed point argument \emph{depending on parameter}, obtaining a diffeomorphism $\eta$ and a matrix-valued function $U$, well-defined for \emph{all} $\omega$, that solve (\ref{Eq(theta)})(\ref{Eq(U)}) for Diophantine $\omega$.

\subsection{Paradifferential Reducibility for PDEs}\label{sec:1.5}
We turn to discuss the role of paradifferential calculus as an alternative to the Nash-Moser/KAM scheme in the construction of quasi-periodic solutions and reducibility.

It has been observed that Nash-Moser type iteration schemes can be replaced by paradifferential calculus in problems involving loss of regularity. This approach was first noted in H\"{o}rmander's paper \cite{Hormander1990}. To the authors' knowledge, the first practical application of this general idea was carried out by Delort \cite{Delort2012}, who sought for \emph{periodic} solutions to \emph{semilinear} PDEs. No change of variables via diffeomorphism was involved there, and the issue with reducibility does not occur\footnote{In fact, Delort~(\cite{Delort2012}, page 642) pointed out that the issue with reducibility in case of multi-dimensional frequency makes the search for quasi-periodic solutions significantly more difficult, raising the question of the validity of paradifferential method. This paper provides a positive answer to this question at least for hyperbolic systems.}. Other attempts of overcoming regularity loss by paralinearization can be found in case of Landau damping \cite{BMM2016}. Usage of paracomposition to balance the regularity loss appeared in \cite{AM2009}, but the present paper seems to be the first time it is employed for an alternative to KAM proof. 

The general idea is to \emph{paralinearize} the nonlinear equation (\ref{F(u)}), namely $\mathscr{F}(\omega,u)=0$, as
$$
T_{\mathscr{F}'(\omega,u)}u+\text{smoothing remainder of }u=0.
$$
Inverting the \emph{paradifferential} operator $T_{\mathscr{F}'(\omega,u)}$ is algebraically the same as inverting $L(\omega,u)$ as guaranteed by the algebraic structure of paradifferential calculus. If inverting $L(\omega,u)=\mathscr{F}'(\omega,u)$ calls for conjugation with a diffeomorphism $\eta$ and $\Psi\text{DO}$ with symbol $B=I_N+U$ (which are given by (\ref{Eq(theta)})(\ref{Eq(U)}) in our case), the paracomposition and paradifferential operators $\eta^\star$, $T_B$ will replace their roles in inverting $T_{\mathscr{F}'(\omega,u)}$, thanks to the symbolic calculus coinciding with that for the usual $\Psi\text{DOs}$ (here we write $\eta_*f:=f\circ\eta$):
$$
\begin{aligned}
B\eta_* L(\omega,u)&\eta_*^{-1}B^{-1}=L_0(\omega),\text{  which has constant coefficients}\\
\Longleftrightarrow\,&
T_B\eta^\star T_{\mathscr{F}'(\omega,u)}(\eta^\star)^{-1}T_{B}^{-1}=L_0(\omega)+\text{smoothing operators}.
\end{aligned}
$$
This then gives the \emph{paradifferential reducibility} of $\mathscr{F}(\omega,u)=0$: with the new unknown $y=(\eta^\star)^{-1}T_{B}^{-1}u$, there holds
$$
L_0(\omega)y+(\eta^\star)^{-1}T_B^{-1}(\text{smoothing remainder of }y)
=0.
$$
This is a constant coefficient equation modulo smoothing terms. Even though $L_0(\omega)^{-1}$ loses some regularity, the loss is \emph{balanced} by the smoothing remainder, yielding a ``fixed point form" within a single Banach space of functions.

The procedure is presented in detail in Section \ref{sec:7}, see in particular Proposition \ref{ParaRedu}. The key equations in this procedure include the paralinearized equation (\ref{7ParalinEQ}), the reduced equation (\ref{7Paray}), and finally the ``fixed point form" (\ref{7ParalinEQ2}). See the beginning of Section \ref{sec:7} for a more precise outline. As discussed in the end of Subsection \ref{sec:1.4}, we employ the technique of extending the Fourier multiplier $L_0(\omega)^{-1}$ to \emph{any} $\omega\in\xR^n$, and the ``fixed point form" equation is equivalent to the original nonlinear equation for $\omega$ in a feasible set. This enables us to solve the ``fixed point form" equation again via Banach fixed point argument depending on external parameter. The detail can be found in Subsection \ref{sec:7.3}. Finally, the measure estimate for the set of feasible $\omega$ is fulfilled in the end of Section \ref{sec:7}.

In summary, the present paper provides useful tools for further inquiry involving analysis of change-of-variables. It also offers simplified proof for ``KAM for PDEs", together with new perspectives for KAM theory in the context of hyperbolic PDEs.

\newpage
\part{Toolbox of Paracomposition Operators}\label{Part1}

\section{Review of Paradifferential Calculus}\label{sec:2}

In this section, we shall review notations and results about Bony's paradifferential calculus. We refer to Bony's paper \cite{Bony1981}, Meyer's paper \cite{Meyer}, H\"{o}rmander's paper \cite{Hormander1988}, Chapter X of H\"{o}rmander's book \cite{Hormander1997}, or Chapter 4 of M\'{e}tivier's textbook \cite{MePise}, for the general theory. Here we follow the presentation by M\'etivier in \cite{MePise}. The advantage of applying paradifferential calculus to nonlinear problems is twofold: on the one hand, it preserves all the algebraic structures enjoyed by usual (pseudo)differential operators, while on the other hand it explicitly manages the regularity loss. 

Throughout the paper, the manifold that we are concerned with is the $n$-dimensional torus $\xT^n$. It is identified with $\xR^n$ modulo the discrete subgroup $(2\pi\xZ)^n$. There is no essential difficulty in generalizing the theory in this paper to $\xR^n$, but focusing on $\xT^n$ avoids the technicalities at infinity.

For~$k\in\xN$, we denote by $C^k$ the usual space of functions on $\xT^n$ whose derivatives of order $\le k$ are continuous. For non-integer $r>0$, we denote by $C^r$ the space of functions whose derivatives up to order $[r]$ are bounded and uniformly H\"older continuous with exponent $r-[r]$. 

\subsection{Littlewood-Paley Decomposition}\label{s1}
We introduce the \emph{Littlewood-Paley decomposition} of a distribution on $\xT^n$ as a standard harmonic analysis construction. 

We represent a distribution $u$ on $\xT^n$ as a Fourier series:
$$
u=\sum_{\xi\in \xZ^n} \hat{u}(\xi) e^{i\xi\cdot x},
$$
where the Fourier coefficient
$$
\hat{u}(\xi)=\int_{\xT^n}u(x)e^{-i\xi\cdot x}\dx.
$$
We denote $\Avg u=\hat u(0)$ for the mean value of $u$ if $u\in L^1$. If $u\in L^2$, then the series converges in $L^2$.

The Littlewood-Paley decomposition is fixed as follows. Let $\varphi\in C^\infty_0(\xR^n)$ be an \emph{even} function, 
with support in an annulus $\{1/2\le \la \xi\ra\le 2\}$, so that
$$
\sum_{j=1}^\infty\varphi(2^{-j}\xi)=1-\psi(\xi),
\quad\mathrm{supp}\psi\subset\{|\xi|\leq1\}.
$$
We can then decompose any distribution $u$ on $\xT^n$ as
$$
u=\Delta_{0}u+\sum_{j\ge1}\Delta_ju,
\quad\text{where}\quad
\Delta_j u=\sum_{\xi\in\mathbb{Z}^n}\varphi(2^{-j}\xi)\hat{u}(k)e^{i\xi\cdot x},
\quad j\geq1,
$$
while obviously $\Delta_0u=\hat u(0)$. We then define the \emph{partial sum operator} $S_j$ to be
$$
S_j=\sum_{l\leq j}\Delta_l,\quad j\geq0,
$$
while for $j\leq0$ we just fix $S_j=\Delta_0$. 

\begin{remark}
The reason of fixing an even frequency cut-off function $\varphi$ is that the operators $S_j$ and $\Delta_j$ shall preserve parity of the functions.
\end{remark}

For an index $s\in \xR$, the \emph{Sobolev space} $H^s(\xT^n)$ consists of those distributions $u$ on $\xT^n$ such that
$$
\lA u\rA_{H^s}:=\left( \sum_{\xi\in \xZ^n} \big(1+\la \xi\ra^2\big)^s \bla \hat{u}(\xi) \bra^2\right)^{1/2} <+\infty.
$$
The space $(H^s,\lA \cdot\rA_{H^s})$ is a Hilbert space. If $s\in \xN$, then
\[
H^s(\xT^n)=\bigl\{u\in L^2(\xT^n) : 
\forall \alpha\in \xN^n,~|\alpha|\leqslant s,~\partial_x^\alpha u\in L^2(\xT^n)\bigr\},
\]
where $\partial_x^\alpha$ is the derivative in the sense of distribution of $u$. Moreover, when $s$ is not an integer, the Sobolev spaces coincide with those obtained by interpolation. If we define
$$
\lA u\rA_{s}^2 \defn \sum_{j=0}^{\infty}2^{2js}\lA \Delta_j u\rA_{L^2}^2,
$$
then it follows from Plancherel theorem that $\lA \cdot\rA_{H^s}$ and $\lA \cdot\rA_{s}$ are equivalent.

We will be using \emph{H\"{o}lder spaces} (also known as \emph{Lipschitz spaces} in the literature) throughout the paper. For an non-integer index $r>0$, the H\"{o}lder norm with index $r$ is defined as
$$
|u|_{C^r}:=|u|_{C^{[r]}}
+\sum_{|\alpha|=[r]}\sup_{x,y\in\mathbb{T}^n}\frac{\big|\partial^{\alpha}u(x)-\partial^{\alpha}u(y)\big|}{|x-y|^{r-[r]}}.
$$
Here $[r]$ is the integer part of $r$. On the other hand, when $r$ is an integer, we simply define $|u|_{C^{[r]}}$ as the classical norm of functions that are $r$ times continuously differentiable.

The Littlewood-Paley characterization of $C^r$ is as follows: for non-integer $r>0$, there holds
$$
|u|_{C^r}\simeq_r
\sup_{j\geq0} 2^{jr}|\Delta_ju|_{L^\infty}.
$$
Direct manipulation with series implies $H^s\subset C^{s-n/2}$ for $s>n/2$ with $s-n/2\notin\xN$. Notice that when $r$ is a natural number, the Littlewood-Paley characterization yields the so-called \emph{Zygmund space} $C^r_*$, but we shall not use it within the scope of this paper.

\subsection{Paradifferential Operators}\label{s2}

\begin{definition}
Given~$r\in [0, +\infty)$ and~$m\in\xR$, the symbol class $\Gamma_{r}^{m}$ denotes the space of locally bounded functions~$a(x,\xi)$ on~$\xT^n\times\xR^n$, $C^\infty$ with respect to $\xi$, such that for all~$\alpha\in\xN^n$, the partial derivative $x\mapsto \partial_\xi^\alpha a(x,\xi)$ is $C^{r}$ with respect to $x$ and, moreover, there exists a constant $C_\alpha$ such that,
\begin{equation*}
\big|\partial_\xi^\alpha a(\cdot,\xi)\big|_{C^{r}}\le C_\alpha
(1+\la\xi\ra)^{m-\la\alpha\ra}.
\end{equation*}
\end{definition}

\begin{remark}\label{Discr}
Notice that the intersection $\cap_{r>0}\Gamma^m_r$ coincides with the usual symbol class $\mathscr{S}^m_{1,0}$ studied in standard pseudo-differential calculus.
\end{remark}

Given a symbol~$a\in\Gamma^m_r$, we define
the \emph{paradifferential operator} $T_a$ by
\begin{equation}\label{eq.para}
\widehat{T_a u}(\xi)=(2\pi)^{-n}\sum_{\eta\in\mathbb{Z}^n}
\chi(\xi-\eta,\eta)\widehat{a}(\xi-\eta,\eta)\widehat{u}(\eta).
\end{equation}
Here
$$
\widehat{a}(\xi,\eta)=\int_{\xT^n} e^{-ix\cdot\xi}a(x,\eta)\dx
$$
is the Fourier transform of~$a$ with respect to the spatial variable, and the cut-off function $\chi$ is fixed as follows:
$$
\chi(\xi,\eta)=\sum_{j\geq0}\left(\sum_{l\leq j-3}\varphi(2^{-l}\xi)\right)\varphi(2^{-j}\eta),
$$
with $\varphi$ being the smooth truncation used to define Littlewood-Paley decomposition. It satisfies
$$
\chi(\xi,\eta)=1 \quad \text{if}\quad \la\xi\ra\le 0.125\la \eta\ra,\qquad
\chi(\xi,\eta)=0 \quad \text{if}\quad \la\xi\ra\geq 0.5\la\eta\ra,
$$
and
$$
\la \partial_\xi^\alpha \partial_\eta^\beta \chi(\xi,\eta)\ra\le 
C_{\alpha,\beta}(1+\la \eta\ra)^{-\la \alpha\ra-\la \beta\ra}.
$$

When $a(x,\xi)$ is independent of $\xi$, the paradifferential operator $T_a$ becomes the widely used \emph{paraproduct} operator:
\begin{equation}\label{T_au}
T_au=\sum_{j\geq0}(S_{j-3}a)\Delta_ju.
\end{equation}
If $a$ is a constant, then $S_{j-3}a\equiv a$, and we have
$$
T_au=\sum_{j\geq0}a\Delta_ju=au.
$$
More generally, if the symbol $a(x,\xi)=\sum_{\alpha}a_\alpha(x)(i\xi)^\alpha$ is a polynomial with respect to the frequency variable $\xi$, then 
$$
T_au=\sum_{\alpha}T_{a_\alpha}\partial^\alpha u.
$$

In case of matrix valued functions $A,U$, the order of multiplication matters, so we define the corresponding \emph{left paraproduct} and \emph{right paraproduct} as
\begin{equation}\label{LRPM}
T_A^{\gau}U=\sum_{j\geq0}S_{j-3}A\cdot\Delta_jU,
\quad
T_U^{\droi}A=\sum_{j\geq0}\Delta_jA\cdot S_{j-3}U.
\end{equation}
However, in case $A$ is $N\times N$-matrix valued and $u$ is $\xC^N$-column valued, the only legitimate order is $Au$, so there is no risk of confusion for simply writing $T_Au$.

In case the expression for $a$ is lengthy, we shall also use the notation 
\begin{equation}\label{DefOpPM}
\Op^\PM(a)=T_a,
\quad\Op^\PM_{\gau}(A)=T_A^{\gau},\,\text{etc.},
\end{equation}
where PM is the abbreviation for \emph{paramultiplication}, the terminology used in \cite{Bony1981}. 

\subsection{Symbolic Calculus}\label{sec.2.2}
We shall use quantitative results from \cite{MePise} about operator norms estimates in symbolic calculus. 
Introduce the following semi-norms.
\begin{definition}\label{defiGmrho}
For~$m\in\xR$,~$r\in [0,+\infty)$ and~$a\in \Gamma^{m}_{r} $, we set
\begin{equation}\label{defi:norms}
\mathcal{M}_{r}^{m}(a)= 
\sup_{\la\alpha\ra\le 2(n+2) +r}\sup_{\la\xi\ra \ge 1/2~}
\big| (1+\la\xi\ra)^{\la\alpha\ra-m}\partial_\xi^\alpha a(\cdot,\xi)\big|_{C^r_x}.
\end{equation}
\end{definition}

\begin{definition}\label{defi:order}
Let~$m\in\xR$.
An operator is said to be of  order~$\le m$ if, for all~$s\in\xR$,
it is bounded from~$H^{s} $ to~$H^{s-m} $. 
\end{definition}

The main features of symbolic calculus for paradifferential operators are given by the following theorems.

\begin{theorem}\label{2theo:sc0}
Let~$m\in\xR$. 
If~$a \in \Gamma^m_0 $, then~$T_a$ is of order~$\le m$. 
Moreover, for all~$s\in\xR$,
\begin{equation}\label{esti:quant1}
\lA T_a \rA_{\mathcal{L}(H^{s},H^{s-m})}\lesssim_s \mathcal{M}_{r}^{m}(a).
\end{equation}
In particular, for a function $a=a(x)$, the paraproduct operator is bounded in all Sobolev spaces under a very mild assumption:
$$
\lA T_a \rA_{\mathcal{L}(H^{s},H^{s})}\lesssim_s |a|_{L^\infty}.
$$
\end{theorem}

\begin{theorem}[Composition]\label{2theo:sc}
Let $m\in\xR$ and $r>0$. 
Suppose $a\in \Gamma^{m}_{r} , b\in \Gamma^{m'}_{r} $, and define
$$
a\#_r b=
\sum_{\la \alpha\ra \leq r} \frac{1}{i^{\la\alpha\ra} \alpha !} \partial_\xi^{\alpha} a \partial_{x}^\alpha b.
$$
Then $T_a T_b -T_{a\#_r b}$ is of order $\le m+m'-r$, and its operator norm satisfies a tame estimate: 
\begin{equation}\label{esti:quant2}
\big\| T_a T_b  - T_{a\#_r b}; 
\mathcal{L}(H^{s},H^{s-(m+m'-r)})\big\|
\lesssim_{s,r}
\mathcal{M}_{r}^{m}(a)\mathcal{M}_{0}^{m'}(b)+
\mathcal{M}_{0}^{m}(a)\mathcal{M}_{r}^{m'}(b).
\end{equation}
In particular, for functions $a,b\in C^r$, the remainder $\R_\CM(a,b):=T_aT_b-T_{ab}$, where CM is the abbreviation for \emph{composition de paramultiplication}, satisfies
$$
\lA \R_\CM(a,b);\mathcal{L}(H^{s},H^{s+r}) \rA
\lesssim_{s,r}
|a|_{L^\infty}|b|_{C^r}+|a|_{C^r}|b|_{L^\infty}.
$$
\end{theorem}

\begin{theorem}[Adjoint]\label{2theo:sc2}
Let $m\in\xR$, $r>0$ and $a\in \Gamma^{m}_{r}$. Denote by 
$(T_a)^*$ the adjoint operator of $T_a$ and by $\bar{a}$ the complex-conjugated of $a$.  Define 
$$
a^{\mathbf{c};r} =
\sum_{\la \alpha\ra \leq r} \frac{1}{i^{\la\alpha\ra} \alpha !} \partial_\xi^\alpha \partial_x^{\alpha} \bar{a} .
$$
Then 
$(T_a)^* -T_{a^{\mathbf{c};r}}$ is of order $\le m-r$. Moreover, for all $s\in\mathbb{R}$, 
\begin{equation}\label{esti:quant3}
\big\|(T_a)^*-T_{a^{\mathbf{c};r}};\mathcal{L}(H^{s}, H^{s-m+r})\big\|
\lesssim_{s,r}
\mathcal{M}_{r}^{m}(a).
\end{equation}
\end{theorem}
\begin{remark}\label{2MatSymbol}
In case that all symbols in the above theorems are matrix valued, the above symbolic calculus formulas are still valid, but the order and type of paraproducts must be specified. For example, the result for matrix valued paraproduct remainder $\R_\CM^{\gau}(A,B):=T_A^{\gau}T_B^{\gau}-T_{AB}^\gau$ reads
$$
\big\|\R_\CM^{\gau}(A,B)\big\|_{\mathcal{L}(H^s,H^{s+r})}
\lesssim_{s,r}|A|_{L^\infty}|B|_{C^r}+|A|_{C^r}|B|_{L^\infty}.
$$
The proof does not differ from the case of scalar valued symbols.
\end{remark}

We observe that the $\#_r$ and $\mathbf{c};r$ operations are simply cut-offs of the composition and adjoint operations in the symbolic calculus for (1,0) pseudo-differential operators. Hence, we assert that \emph{paradifferential calculus retains the same algebraic structure as standard pseudo-differential calculus, modulo smoothing operators}. 

A further observation is that the remainder estimates in Theorem \ref{2theo:sc} are \emph{tame} with respect to the symbols. This feature becomes particularly useful for delicate study of nonlinear problems.

\begin{remark}
We have the following corollary: if 
$$
a=\sum_{0\le j <r}a^{(m-j)}\in \sum_{0\le j <r} \Gamma^{m-j}_{r-j},\quad 
b=\sum_{0\le k <r}b^{(m-k)}\in \sum_{0\le k <r} \Gamma^{m'-k}_{r-k},
$$
with $m,m'\in\xR$ and $r>0$, then with
$$
c=
\sum_{ \la \alpha\ra +j+k < r} \frac{1}{i^{\la\alpha\ra} \alpha !} \partial_\xi^{\alpha} a^{(m-j)} \partial_{x}^\alpha b^{(m'-k)},
$$
the operator $T_a T_b -T_{c}$ is of order $\le m+m'-r$.
\end{remark}
The symbols considered in the previous remark naturally belong to the (poly)homogeneous symbol class introduced by Bony, defined as follows:
\begin{definition}\label{D:3.5}
The space $\dot\Gamma_{r}^{m}$ of homogeneous symbols of degree $m$ consists of $\Gamma^m_r$ symbols $a(x,\xi)$ of the form
$$
a(x,\xi)=\psi(|\xi|)|\xi|^ma'\left(x,\frac{\xi}{|\xi|}\right),
$$
where $a'(x,\omega)$ is $C^r$ in $x\in \xT^n$ and smooth in $\omega\in S^{n-1}\subset\xR^n$, while $\psi(t)$ is a smooth cut off function that vanishes for $t\leq1/2$ and equals to 1 for $t\geq1$. 

The space $\Sigma^m_r$ of polyhomogeneous symbols of degree $m$ consistes of $\Gamma^m_r$ symbols $a(x,\xi)$ of the form
$$
a=\sum_{j:0\le j <r}a^{(m-j)}\in \sum_{j:0\le j <r}\dot\Gamma^{m-j}_{r-j}.
$$
We say that 
$a^{(m)}$ is the principal symbol of $a$.
\end{definition}
\begin{remark}
Our definition of (poly)homogeneous symbols avoids the singularity at zero frequency. Since we focus on periodic functions, this adjustment does not impact our theory at all. For simplicity, we will omit the cut-off function $\psi$ in our expressions from now on. For instance, the symbol $\xi/|\xi|^2$ is understood as having zero average. With a slight abuse of notation, we also treat functions as homogeneous symbols of degree zero, though we \emph{do not} average out in this case.
\end{remark}

\subsection{Paraproduct and Paralinearization}\label{sec.2.3}
Recall that if $a=a(x)$ is a function of $x$ only, the paradifferential operator $T_a$ is called a paraproduct. A key feature of paraproducts is that one can replace 
bilinear expressions by paraproducts up to smoothing operators. 
Also, one can define paraproducts $T_a$ for rough functions $a$ which only belongs to $L^\infty$. 

\begin{definition}\label{R_PM}
Given two functions~$a,b$ on $\xT$, we define the paraproduct remainder 
$$
\begin{aligned}
R_\PM(a,u)&=au-T_a u-T_u a\\
&=\sum_{j,k:|j-k|<3}\Delta_j a\Delta_ku.
\end{aligned}
$$
Here PM is the abbreviation for \emph{paramultiplication}. For matrix valued functions $A,U$, we define similarly
$$
\begin{aligned}
R_\PM(A,U)&=AU-T_A^{\gau}U-T_U^{\droi}A\\
&=\sum_{j,k:|j-k|<3}\Delta_jA\cdot\Delta_kU.
\end{aligned}
$$
\end{definition}

We record here some estimates about paraproducts; see Chapter 2 of \cite{BCD}, Chapter 2 of \cite{Chemin} or Chapter X of \cite{Hormander1997}. The proof for matrix valued version does not differ from scalar valued version.

\begin{theorem}\label{PMReg}
Let $\alpha,\beta\in \xR$. If $\alpha+\beta>0$, then
\begin{align}
&\lA R_\PM(a,u) \rA _{H^{\alpha + \beta-n/2} }
\lesssim_{\alpha,\beta} \lA a \rA _{H^{\alpha} }\lA u\rA _{H^{\beta} },\label{Bony} \\ 
&\lA R_\PM(a,u) \rA _{H^{\alpha + \beta} } 
\lesssim_{\alpha,\beta} |a|_{C^{\alpha}}\lA u\rA _{H^{\beta} }.\label{Bony3}
\end{align}
In case of matrix valued functions the result is the same.
\end{theorem}

The theorem states that in the bilinear expression $B(a,u)=au$, the sum of the paraproducts $T_au+T_ua=T_{B'(a,u)}\cdot(a,u)$ captures the most irregularity. In the same vein, we 
have the following celebrated \emph{paralinearization theorem} due to Bony \cite{Bony1981} (a beautiful proof can be found in \cite{Meyer} or Chapter X of \cite{Hormander1997}, and a quantitative version can be found in \cite{AS2023}). 

\begin{theorem}\label{2ParaLin}
Let $0\leq r<s$. Suppose the function $F=F(x,z)\in C^{\infty}$, and $u\in H^s\cap C^r$. Then there holds the paralinearization formula:
$$
F(x,u)-F(x,0)
=T_{F_z'(x,u)}u+\R_\PL(F,u)\in H^s+H^{s+r}.
$$
The remainder $\R_{\PL}(F,u)$ is a smooth nonlinear mapping in $u$, and is tame in the following sense:
$$
\|\R_\PL(F,u)\|_{H^{s+r}}\lesssim_{r,s,F}\big(1+|u|_{C^r}\big)\|u\|_{H^s}.
$$
In particular, if $F\in C^{\infty}$ and $u\in H^s$ with $s>n/2$, then $F(u)\in H^s$.

In parallel, if both $r$ and $2r$ are not integer, and $F\in C^{\infty}$, then
$$
F(x,u)-F(x,0)
=T_{F_z'(x,u)}u+\R_\PL(F,u)\in C^r+C^{2r}.
$$

Note that neither $F$ nor $u$ are restricted to scalar-valued functions.
\end{theorem}

Theorem \ref{2ParaLin}, along with the symbolic calculus theorems \ref{2theo:sc}-\ref{2theo:sc2}, allows us to handle nonlinear expressions $F(u)$ \emph{as if} they were linear. We first \emph{paralinearize} the expression $F(u)$, then analyze the operator $T_{F'(u)}$, which behaves algebraically like the differential $F'(u)$. Since any potential loss of regularity can be recovered by the remainder $\R_\PL(u)$, there is no concern about losing regularity.

We conclude this subsection by a useful tame estimate for product and composition of functions in H\"{o}lder spaces (see Section III.C.3.2 of \cite{AG}): 
\begin{proposition}\label{2esti:C^scomp}
Suppose $r\geq1$. 

(1) If $f,g\in C^r(\xT^n,\xR)$, then
$$
|fg|_{C^r}\lesssim_r |f|_{C^0}|g|_{C^r}+|f|_{C^r}|g|_{C^0}.
$$

(2) If $f\in C^r(\xR^N,\xR^L)$, $g\in C^r(\xT^n,\xR^N)$, then 
$$
|f\circ g|_{C^r}\lesssim_r
|f|_{C^r}|g|_{C^1}^r+|f|_{C^1}|g|_{C^r}+|f|_{C^0}.
$$

(3) Let $\chi$ be a $C^r$ diffeomorphism of $\xT^n$. Then its inverse $\chi^\iota$ satisfies
$$
|\chi^\iota|_{C^r}\leq C\big(|\chi|_{C^1},|\chi^\iota|_{C^1}\big)|\chi|_{C^r}.
$$
\end{proposition}

\section{Paracomposition as Solution of Paratransport Equation}\label{sc:3}
In this section, we aim to present an alternative definition of the paracomposition operator as a time 1 flow map. The idea here resembles the classical Egorov theorem in the usual pseudo-differential calculus as discussed in Subsection \ref{sec:1.2}. Examples of realization of this general idea include \cite{ABHK} (Appendix C),  \cite{BD2018,BFF2021,BFP2023,BMM2024}, and also \cite{BFPT2024} (Section 6) in the \emph{smooth} category. Comparing to the definition (\ref{FIO}) as Fourier integral operator, the advantage of this alternative approach is clear: it presents more of the hidden geometric information, and is much easier to manipulate with.

\subsection{Auxiliary Results on Linear (Para)transport Equations}
We state and prove some necessary auxiliary results on the Cauchy problem of classical linear transport equations and its \emph{paradifferential} counterpart. Hereafter, given $r>0$, we use  
$L^p_tC^r_x$ and $C^0_tC^r_x$ as a compact notation for $L^p(\xR;C^r(\xT^n))$ and $C^0(\xR;C^r(\xT^n))$ respectively, where the spaces $C^r(\xT^n)$ are as defined at the beginning of Section~\ref{s2}.

\begin{proposition}\label{3RegTrans}
Suppose $r\geq1$. Let $X\colon \xR\times \xR^N\to \xR^N$ be a time-dependent vector field of class $L^1_tC^r_x$, and let $u_0\in C^r(\xR^N)$ and $f\in L^1_tC^r_x$. Then the Cauchy problem
\begin{equation}\label{LinTranspIVP}
\partial_tu+X\cdot\nabla_x u=f,
\quad
u(0,x)=u_0(x)
\end{equation}
has 
a unique solution $u\in C^0_tC^r_x$, and there is a function $K_r\colon\xR\times\xR_+\to\xR_+$, increasing in its two arguments, such that
\be\label{n200}
|u(t,\cdot)|_{C^r_x}
\leq K_r\big(t,\|X\|_{L^1_tC^r_x}\big)\left(|u_0|_{C^r}+\int_0^t |f(\tau,\cdot)|_{C^r_x}\dtau\right).
\ee
\end{proposition}
\begin{proof}
The proof follows from the standard characteristic method and thus we omit most of the details. Write $\Phi(\tau,t;x)$ for the flow of the vector field $X$; namely $\Phi(\tau,t;x)$ is the unique solution of the initial value problem
\begin{equation}\label{TranspFlow}
\frac{\partial}{\partial \tau}\Phi(\tau,t;x)=X\big(\tau,\Phi(\tau,t;x)\big),
\quad 
\Phi(t,t;x)=x.
\end{equation}
It follows from standard Carathéodory theory of ODEs (see the appendix of \cite{BV2017}) that $x\mapsto\Phi(\tau,t;x)$ is a diffeomorphism depending absolutely continuously in $(\tau,t)$, satisfying
$$
\Phi\big(\tau_2,t;\Phi(t,\tau_1;x)\big)=\Phi(\tau_2,\tau_1;x).
$$
The hypothetical solution $u$ of (\ref{LinTranspIVP}) should satisfy
$$
\begin{aligned}
\frac{\partial}{\partial \tau}u\big(\tau,\Phi(\tau,0;x)\big)
&=(\partial_tu)\big(\tau,\Phi(\tau,0;x)\big)+X\big(\tau,\Phi(\tau,0;x)\big)\cdot\nabla_xu\big(\tau,\Phi(\tau,0;x)\big)\\
&=f\big(\tau,\Phi(\tau,0;x)\big),
\end{aligned}
$$
implying
$$
u\big(t,\Phi(t,0;x)\big)=u_0(x)+
\int_0^t f\big(\tau,\Phi(\tau,0;x)\big)\dtau.
$$
As a result, noticing that $\Phi(0,t,x)$ is the inverse of the diffeomorphism $\Phi(t,0,x)$, we have
\begin{equation}\label{SolLinTranspIVP}
u(t,x)=u_0\big(\Phi(0,t;x)\big)+
\int_0^t f\big(\tau,\Phi(\tau,t;x)\big)\dtau.
\end{equation}
This gives the explicit formula of solving (\ref{LinTranspIVP}).

We can then proceed to prove the regularity result. In view of the composition estimate (Proposition \ref{2esti:C^scomp}), it suffices to show that the flow map $x\mapsto \Phi(\tau,t,x)$ is of $C^r$ regularity. For clarity, let us sketch how this could be fulfilled for $r\in(1,2)$. In fact, we know from standard ODE theory that $\Phi$ is $C^1$ in $x$, and the differential $\partial_x\Phi$ solves the linear equation
$$
\frac{\partial}{\partial \tau}\partial_x\Phi(\tau,t;x)
=(\partial_x X)\big(\tau,\Phi(\tau,t;x)\big)\cdot\partial_x\Phi(\tau,t;x),
\quad 
\partial_x\Phi(t,t;x)=I_N.
$$
It remains to notice that for two given points $x_1,x_2\in\mathbb{R}^N$, the difference quotient
$$
\frac{\nabla_x\Phi(\tau,t;x_1)-\nabla_x\Phi(\tau,t;x_2)}{|x_1-x_2|^{r-1}}
$$
again solves a linear ODE with respect to $\tau$. Standard a priori estimates for linear ODEs then ensures that this difference quotient is uniformly bounded in terms of the H\"{o}lder norms of $X$:
\begin{equation}\label{Phi_Esti}
|\Phi(\tau,t;x)|_{C^r_x}\leq
C\exp\left(C\int_t^\tau |X(\tau_1,x)|_{C^r_x}\dtau_1\right).
\end{equation}
Then the desired estimate~\e{n200} follows from~\e{SolLinTranspIVP} and Proposition \ref{2esti:C^scomp}.
\end{proof}

We next introduce the notion of \emph{paratransport equation}. Formally this is nothing but the paradifferential version of linear transport equation. However, the calculus of paradifferential operators enable one to directly obtain several useful a priori estimates for the solution.

When~$a$ and~$u$ are
symbols and functions depending on time $t$, we still denote by~$T_a u$ the spatial
paradifferential operator (or paraproduct), namely for all time $t$, $(T_a u)(t):=T_{a(t)}u(t)$. 

\begin{proposition}\label{ParaTranspEq}
Suppose $X$ is a time-dependent real vector field of class $L^1_tC^1_x$. Suppose $f\in L^1_tH^s_x$, $h\in H^s$, where $s\in\mathbb{R}$ and $B\in \cap_{\sigma\in\xR}L^1_t\mathcal{L}(H^\sigma_x, H^\sigma_x)$. Then there exists a unique function
$$
w\in C^0_tH^s_x\cap C^1_tH^{s-1}_x
$$
solving the Cauchy problem
\be\label{n20}
\partial_t w+T_X\cdot\nabla_x w+Bw=f,\quad w\big|_{t=0}=h.
\ee
Moreover, the following estimate holds:
$$
\begin{aligned}
\|w(t)\|_{H^s_x}
&\leq C_s\exp\left[C_s\int_0^t\left(|X(\tau)|_{C^1_x}
+\|B(\tau)\|_{\mathcal{L}(H^s_x, H^s_x)}\right)\dtau\right]\\
&\quad\times\left(\|h\|_{H^s}+\int_0^t \|f(\tau)\|_{H^s_x}\dtau\right).
\end{aligned}
$$
\end{proposition}
\begin{proof}
Since the proof of this proposition is quite standard, we still only provide a sketch of it. We first assume $s=0$ and further assume that 
$w\in C^1_tL^2_x$. Multiplying (\ref{n20}) by $\bar w$, integrating by parts with respect to $x$, implementing the adjoint formula (using Proposition \ref{2theo:sc2} with $r=1$) for $T_X\cdot\nabla_x$, we obtain the differential inequality
$$
\frac{\diff}{\dt}\|w(t)\|_{L^2_x}^2
\lesssim \left(|X(t)|_{C^1_x}+\|B(t)\|_{\mathcal{L}(L^2_x,L^2_x)}\right)
\|w(t)\|_{L^2_x}^2+\|f(t)\|_{L^2_x}\|w(t)\|_{L^2_x}.
$$
Direct application of Grönwall's inequality then implies the wanted estimate for $s=0$. We then infer that the inequality holds 
for all $w\in C^0_tL^2_x\cap C^1_tH^{-1}_x$ by using the Friedrichs mollifiers.

For general $s\in\mathbb{R}$, we commute equation (\ref{n20}) with $\langle D_x \rangle^s$ to obtain
\begin{equation}\label{n20s}
\begin{aligned}
&\partial_t \langle D_x \rangle^sw+T_X\cdot\nabla_x  \langle D_x \rangle^sw
+\big[\langle D_x \rangle^s,T_X\cdot\nabla_x\big]w
+\langle D_x \rangle^s Bw
=\langle D_x \rangle^sf,\\
& \langle D_x \rangle^sw\Big|_{t=0}=\langle D_x \rangle^sh.
\end{aligned}
\end{equation}
Therefore the unknown $u=\langle D_x \rangle^sw$ satisfies an equation of the form~\eqref{n20} with $B$ replaced by $B_s\defn\big[\langle D_x \rangle^s,T_X\cdot\nabla_x\big]\langle D_x \rangle^{-s}
+\langle D_x \rangle^s B\langle D_x \rangle^{-s}$. 
According to Proposition \ref{2theo:sc}, the commutator $\big[\langle D_x \rangle^s,T_X\cdot\nabla_x\big]$ is of order $s$, with operator norm in $\mathcal{L}(H^s_x,L^2_x)$ controlled by $|X|_{C^1_x}$. This implies that $B_s$ belongs to $L^1_t\mathcal{L}(L^2_x, L^2_x)$. 
We are then in position to apply the inequality for $s=0$ to the unknown $u$, which then implies the a priori estimate for general~$s$. 

Once the a priori estimate has been established, a standard duality argument may be implemented to show existence and uniqueness of solution to the Cauchy problem; see for example, Subsection 6.3 in Chapter VI of \cite{Hormander1997}. 
\end{proof}

The following corollary is a direct implication of the energy estimates in Proposition \ref{ParaTranspEq}. For simplicity we present the result for $B\equiv 0$, since this is the only one we need. 

\begin{corollary}\label{ParaTranspEqCor1}
The solution $w$ has Lipschitz dependence on $X$ at the price of losing one derivative in the following sense. Fix $X_1,X_2\in L^1_tC^1_x$, $f_1,f_2\in L^1_tH^s_x$, such that
$$
|X_i|_{L^1_tC^1_x}\leq L,\quad
\|f_i\|_{L^1_tH^s_x}\leq M.
$$
For $i=1,2$, let $w_i\in C^0_tH^s_x\cap C^1_tH^{s-1}_x$ be the solution of the Cauchy problem 
$$
\partial_tw_i+T_{X_i}\cdot\nabla_x w_i=f_i,
\quad w_i\big|_{t=0}=h.
$$
Then with $\Delta_{1,2}A:=A_1-A_2$,
    $$
    \begin{aligned}
    \|\Delta_{1,2}&w(t)\|_{H^{s-1}_x}\\
    &\leq C_{s,L}(t)
    \left[\int_0^t\big(M+\|h\|_{H^s}\big)|\Delta_{1,2}X(\tau)|_{C^1_x}
    +\|\Delta_{1,2}f(\tau)\|_{H^{s-1}_x}
    \dtau\right].
    \end{aligned}
    $$
\end{corollary}
\begin{proof}
In fact, $\Delta_{1,2}w$ satisfies
$$
\partial_t\Delta_{1,2}w+T_{X_1}\cdot\nabla_x\Delta_{1,2}w
=\Delta_{1,2}f-T_{\Delta_{1,2}X}\cdot\nabla_x w_2,
\quad \Delta_{1,2}w\Big|_{t=0}=0.
$$
Since we already know the estimate of $\|w_2(\tau)\|_{H^{s}_x}$ in terms of $\|f_2\|_{L^1_tH^s_x}$ and $\|h\|_{H^s}$, so the right-hand-side satisfies
$$
\begin{aligned}
\big\|\Delta_{1,2}f(\tau)&-T_{\Delta_{1,2}X(\tau)}\cdot\nabla_x w_2(\tau)\big\|_{H^{s-1}_x}\\
&\lesssim_{s,L}\|\Delta_{1,2}f(\tau)\|_{H^{s-1}_x}
+|\Delta_{1,2}X(\tau)|_{C^1_x}\big(\|f_2\|_{L^1H^s}+\|h\|_{H^s}\big).
\end{aligned}
$$
We conclude by directly applying Proposition \ref{ParaTranspEq}.
\end{proof}

\subsection{Strongly Continuous Dependence}
With a bit more effort, we are even able to prove strongly continuous dependence of the solution to paratransport equation on the vector field in a certain sense. However, no explicit estimate on the modulus of such continuity can be obtained, and we should not anticipate such estimate.
\begin{proposition}\label{ParaTranspCont}
Fix a time interval $[0,T]$, an index $s\in\xR$ and $h\in H^s(\xT^n)$. Suppose $X\in L^1_TC^1_x$. The solution $w\in L^\infty_TH^s_x$ of the initial value problem
$$
\partial_t w+T_{X}\cdot\nabla_x w=0,\quad w\Big|_{t=0}=h
$$ 
has strongly continuous dependence on $X$: if $\|X_k- X;L^1_TC^1_x\|\to0$ as $k$ goes to $+\infty$, then the corresponding sequence of 
solutions $w_k$ converges to $w$ for the $L^\infty_TH^s_x$-norm.
\end{proposition}
\begin{proof}
Let us first study the evolution of a Littlewood-Paley building block $\Delta_jw$, defined in Subsection \ref{s1}. Commuting the equation with $\Delta_j$, we obtain, with $P_j(X)w:=[T_X\cdot\nabla_x,\Delta_j]w$,
\begin{equation}\label{D_jw}
\partial_t\Delta_jw
+T_{X}\cdot\nabla_x \Delta_jw
=P_j(X)w,
\quad 
\Delta_jw\Big|_{t=0}=\Delta_jh.
\end{equation}
Recalling from (\ref{T_au}) the definition of paraproducts, we rewrite the paratransport term as
$$
\begin{aligned}
T_{X}\cdot\nabla_x \Delta_jw
=\sum_{l:|l-j|\leq1}(S_{l-3}X)\cdot\nabla_x\Delta_l\Delta_jw.
\end{aligned}
$$
Now, we have
$$
(S_{j-3}X)\cdot\nabla_x\Delta_jw=\sum_{l:|l-j|\leq1}(S_{j-3}X)\cdot\nabla_x\Delta_l\Delta_jw,
$$
and hence we obtain that
\begin{equation}\label{D_jwbis}
\partial_t\Delta_jw
+(S_{j-3}X)\cdot\nabla_x\Delta_jw
=P_j(X)w+Q_j(X)w,
\quad 
\Delta_jw\Big|_{t=0}=\Delta_jh.
\end{equation}
where 
$$
Q_j(X)w=\sum_{l:|l-j|\leq1}(S_{j-3}X-S_{l-3}X)\cdot\nabla_x\Delta_l\Delta_jw.
$$
Since for $|l-j|\leq 1$ there holds,
$$
|S_{l-3}X-S_{j-3}X|_{L^\infty_x}\lesssim 2^{-j}|X|_{C^1_x},
$$
we get a factor $2^{-j}$ that balances the derivative on $w$, so
\begin{equation}\label{PjQj}
\|Q_j(X)w\|_{L^2_x}
\lesssim |X|_{C^1_x}\|\Delta_jw\|_{L^2_x}.
\end{equation}
On the other hand, the commutator $P_j(X)$ is uniformly of order zero, and extracts Fourier coefficients $\hat w(\xi)$ with $|\xi|\simeq 2^j$. We may then insert a frequency localization $\zeta_j(D_x)$, satisfying $\zeta_j\equiv1$ on $\text{supp}(\varphi_{j-1}+\varphi_j+\varphi_{j+1})$, and find $P_j(X)w=P_j(X)\zeta_j(D_x)w$ together with the estimate
\begin{equation}\label{PjQj-b}
\|P_j(X)w\|_{L^2_x}
\lesssim |X|_{C^1_x}\|\zeta_j(D_x)w\|_{L^2_x}.
\end{equation}

Denoting by $\Phi_j(\tau,t;x)$ the flow of the regularized vector field $S_{j-3}X$, we may apply (\ref{SolLinTranspIVP}) to (\ref{D_jwbis}). Inserting the frequency localization $\zeta_j(D_x)$ in front of each $\Delta_jw$, summing over $j$, we get the following integral equation for $w$:
\begin{equation}\label{W_Int}
\begin{aligned}
w(t,x)
&=\sum_{j\geq0}\zeta_j(D_x)(\Delta_jh)\big(\Phi_j(0,t;x)\big)\\
&\quad+\sum_{j\geq0}\zeta_j(D_x)\int_0^t(P_j(X)w)\big(\tau,\Phi_j(\tau,t,x)\big)+(Q_j(X)w)\big(\tau,\Phi_j(\tau,t,x)\big)\dtau\\
&=:A(X)h+B(X)w.
\end{aligned}
\end{equation}

Let us show that the linear operators $A(X),B(X)$ are bounded. Noting that the regularized vector fields $S_{j-3}X$ satisfy $\|S_{j-3}X\|_{L^1_TC^1_x}\lesssim \|X\|_{L^1_TC^1_x}$, we find that the flows $\Phi_j(\tau,t;x)$ form a family of $C^1$ diffeomorphisms of $\xT^n$ absolutely continuous in $\tau,t$, again by the Carathéodory theory of ODEs (see the appendix of \cite{BV2017}). Therefore, given any $f\in L^2(\xT^n)$, we have the following estimate independent from $j$:
\begin{equation}\label{L2_flow}
\big\|f\big(\Phi_j(\tau,t,x)\big)\big\|_{L^2_x}
\leq K\big(T,\|X\|_{L^1_TC^1_x}\big)\|f\|_{L^2}.
\end{equation}
This can be obtained by a simple change of variable, with the aid of the estimate (\ref{Phi_Esti}). Therefore we can use (\ref{PjQj})-(\ref{PjQj-b}) to get
\begin{equation}\label{L2_flow1}
\begin{aligned}
\big\|(\Delta_jh)\big(\Phi_j(0,t;x)\big)\big\|_{L^2_x}
&\leq K\big(T,\|X\|_{L^1_TC^1_x}\big)\|\Delta_jh\|_{L^2},\\
\big\|(P_j(X)w)\big(\tau,\Phi_j(\tau,t,x)\big)\big\|_{L^2_x}
&\leq K\big(T,\|X\|_{L^1_TC^1_x}\big)|X(\tau)|_{C^1_x}\|\zeta_j(D_x)w\|_{L^2_x},\\
\big\|(Q_j(X)w)\big(\tau,\Phi_j(\tau,t,x)\big)\big\|_{L^2_x}
&\leq K\big(T,\|X\|_{L^1_TC^1_x}\big)|X(\tau)|_{C^1_x}\|\Delta_jw\|_{L^2_x}.
\end{aligned}
\end{equation}
(\ref{L2_flow1}) implies that as long as $\|X\|_{L^1_TC^1_x}$ stays bounded, the sequence $j\to 2^{js}\big\|(\Delta_jh)\big(\Phi_j(0,t;x)\big)\big\|_{L^2_x}$ will be uniformly summable in $\ell^2$ since $2^{js}\|\Delta_jh\|_{L^2}$ serves as a dominating sequence.

Given the presence of the frequency localization $\zeta_j(D_x)$ in (\ref{W_Int}), we find that the $H^s_x$ norm of $A(X)h,B(X)w$ are just weighted sums of the $L^2$ norms in (\ref{L2_flow1}); for example,
\begin{equation}\label{A(X)h}
\begin{aligned}
\|(A(X)h)(t)\|_{H^s_x}^2
&\lesssim_s \sum_{j\geq0} 2^{2js}\big\|(\Delta_jh)\big(\Phi_j(0,t;x)\big)\big\|_{L^2_x}^2\\
&\leq K_s\big(T,\|X\|_{L^1_TC^1_x}\big)\sum_{j\geq0} 2^{2js}\|\Delta_jh\|_{L^2}^2
\simeq K_s\big(T,\|X\|_{L^1_TC^1_x}\big)\|h\|_{H^s}^2.
\end{aligned}
\end{equation}
We thus obtain
\begin{equation}\label{AB(X)}
\begin{aligned}
\|A(X)h\|_{L^\infty_TH^s_x}
&\leq K_s\big(T,\|X\|_{L^1_TC^1_x}\big)\|h\|_{H^s},\\
\|B(X)w\|_{L^\infty_TH^s_x}
&\leq \left(\int_0^T|X(\tau,x)|_{C^1_x}\dtau\right)K_s\big(T,\|X\|_{L^1_TC^1_x}\big)\|w\|_{L^\infty_TH^s_x}.
\end{aligned}
\end{equation}

We next verify that the linear operators $A(X),B(X)$ are continuous in $X\in L^1_TC^1_x$ with respect to the strong operator topology of $L^\infty_TH^s_x$; by definition of the latter topology, this means that, if $X_k\to X$ in $X\in L^1_TC^1_x$, then for all $h\in H^s(\xT^n)$ and all 
$w\in L^\infty_TH^s_x$, there holds as $k$ goes to $+\infty$,
$$
\|A(X_k)h-A(X)h\|_{L^\infty_TH^s_x}\to0,
\quad
\|B(X_k)w-B(X)w\|_{L^\infty_TH^s_x}\to0.
$$
In fact, the flow map $\Phi_j^k(\tau,t;x)$ of $S_{j-3}X_k$ converges in $C^1_x$ to $\Phi_j(\tau,t;x)$, uniformly in $\tau,t$, so we immediately obtain
$$
\big\|(\Delta_jh)\big(\Phi_j^k(0,t;x)\big)-(\Delta_jh)\big(\Phi_j(0,t;x)\big)\big\|_{L^2_x}\to0.
$$
Applying the dominated convergence theorem to (\ref{A(X)h}), we obtain the convergence $\|A(X_k)h-A(X)h\|_{L^\infty_TH^s_x}\to0$. The convergence $\|B(X_k)w-B(X)w\|_{L^\infty_TH^s_x}\to0$ is similar.

The last step is quite straightforward. We can take $T$ so small that the constant 
$$
\left(\int_0^T|X_k(\tau,x)|_{C^1_x}\dtau\right)K_s(\|X_k\|_{L^1_TC^1_x})
\leq \frac{1}{2},
$$
which is possible since $X_k\to X$ in $L^1_TC^1_x$. Therefore, the integral equation (\ref{W_Int}) is directly solved by $w=\big(\Id-B(X)\big)^{-1}A(X)h$. It then suffices to justify that the sequence
$$
w_k=\big(\Id-B(X_k)\big)^{-1}A(X_k)h
\to
\big(\Id-B(X)\big)^{-1}A(X)h=w,
\quad \text{in the norm of }L^\infty_TH^s_x.
$$
This follows from the uniform bound $\|B(X_k);\mathcal{L}(L^\infty_TH^s_x)\|\leq1/2$ and the convergence $A(X_k)\to A(X)$, $B(X_k)\to B(X)$ in the strong operator topology. Finally, a continuous induction on time interval justifies this convergence for any time $T$.
\end{proof}

\begin{remark}
The proof just presented may be viewed as the realization of the abstract interpolation argument developed in \cite{ABITZ2024}.
\end{remark}

\subsection{Isotopy and Deformation Vector Field}
Throughout this subsection, we shall denote $x \in \xT^n$ as a \emph{column} vector. Unlike the convention in symplectic geometry, the frequency $\xi$ will also be treated as a \emph{column} vector. The notation $\cdot$ will denote the inner product of column vectors. A vector field on $\xT^n$ will be considered as taking values in column vectors. The gradient $\nabla_x a(x, \xi)$ or $\nabla_\xi a(x, \xi)$ of a scalar function will likewise be treated as a column vector. 

We will be considering a diffeomorphism of the form
$$
\xT^n\ni x\mapsto \chi(x)=x+\theta(x) \in \xT^n,
$$
where the vector function $\theta=(\theta_1,\ldots,\theta_n)^\T$ satisfies\footnote{The upper bound 0.99 is not essential. We fix it simply to omit dependence of constants on upper bounds of $\sup_x|\partial_x\theta(x)|$.}
\be\label{theta}
\theta\in C^r\quad\text{for some }r\in (1,+\infty)\setminus\xN,
\quad\text{and}
\quad \sup_{x\in\xT ^n}\la \partial_x\theta(x)\ra\leq0.99.
\ee
Here $\partial_x\theta$ is the Jacobian matrix of $\theta$: 
$$
\partial_x\theta(x)=\left(\begin{matrix}
\partial_x\theta_1 \\ \vdots \\ \partial_x\theta_n
\end{matrix}\right)=
\left(\partial_{x_j}\theta_i
\right)_{1\le i,j\le n},
$$
and $\la \partial_x\theta(x)\ra$ denotes the Hilbert-Schmidt matrix norm
$$
\la \partial_x\theta(x)\ra=\left(\sum_{1\le i,j\le n}(\partial_{x_j}\theta_i)^2\right)^{1/2}.
$$

Our next step is to introduce deformations associated to this diffeomorphism $\chi$.

\begin{definition}\label{Deformation}
Given $\theta$ satisfying (\ref{theta}), define a homotopy
\begin{equation}\label{homotopy}
\Theta(\tau,x)=\tau e^{-(1-\tau)\langle D_x\rangle}\theta(x)
:=\sum_{\xi\in\xZ^n}\tau e^{-(1-\tau)\langle\xi\rangle}\hat{\theta}(\xi)e^{i\xi\cdot x},
\quad \tau\in[0,1].
\end{equation}
Given a scalar function $f\in C^1(\xT^n)$, the deformation $f^{\sharp;\theta}$ in the space $C^1(\xT^n)$ is defined by
$$
f^{\sharp;\theta}(\tau,x)=f(x+ \Theta(\tau,x)),
$$
so that $f^{\sharp;\theta}(1,x)=f\circ \chi (x)=f(x+\theta(x))$. Finally, the deformation vector field $X(\tau,x)$ associated to the diffeomorphism $\chi(x)=x+\theta(x)$ is defined as
\be\label{N10}
X(\tau,x)=-\big(I_n+\partial_x\Theta(\tau,x)\big)^{-1}\partial_\tau\Theta(\tau,x).
\ee
\end{definition}

We shall verify below that $x+\Theta(\tau,x)$ is in fact an isotopy between the identity map and $\chi$. Of course, the simplest isotopy between identity and $\chi$ is $x+\tau\theta(x)$, the one used in \cite{BD2018,BFF2021,BFP2023,BMM2024}, and also \cite{BFPT2024} in the \emph{smooth} category. However, the $\Theta(\tau,x)$ in Definition \ref{Deformation} is more regular than $\theta$ when averaged with respect to the parameter $\tau$, reminiscent of the smoothing diffeomorphism considered by Lannes~\cite{Lannes2005}. This makes it more suitable for symbolic calculus compared to the simplest one; see for example, the Remark after Theorem \ref{ParacompoPrd}.

\begin{lemma}\label{L:rho}
Suppose $\theta$ is as in \e{theta} and define $\Theta$ as in Definition  \ref{Deformation}.

(1) There holds
\begin{align}
&\sup_{\tau\in[0,1]}|\Theta|_{C^r_x}
+\sup_{\tau\in[0,1]}|\partial_\tau\Theta|_{C^{r-1}_x}
\lesssim |\theta|_{C^r},\label{N1f}\\
&
\int_0^1|\Theta(\tau,x)|_{C^{r+\lambda}_x}\dtau\lesssim_{\lambda,r}
|\theta|_{C^r},
\quad0<\lambda<1.\label{N1}
\end{align}

(2) For all $\tau\in [0,1]$, there holds
$$
\sup_{x\in\xT ^n}\la \partial_x\Theta(\tau,x)\ra
\leq\sup_x|\partial_x\theta(x)|\leq 0.99,
$$
so for any $\tau\in [0,1)$ and any $r'<r$, the mapping $x\mapsto x+\Theta(\tau,x)$ is a $C^{r'+1}$ diffeomorphism of the torus $\xT^n$.

(3) For any $r'<r$, the deformation vector field $X$ defined by \e{N10} satisfies
\be\label{XReg}
|X(\tau,x)|_{C^0_\tau C^{r-1}_x}
\lesssim_{r}|\theta|_{C^r},
\quad
|X(\tau,x)|_{L^1_\tau C^{r'}_x}
\lesssim_{r,r'}|\theta|_{C^r}.
\ee
Furthermore, the mapping $\theta\to X$ from $C^r$ to $C^0_\tau C^{r-1}_x\cap L^1_\tau C^{r'}_x$ is smooth.
\end{lemma}

\begin{proof}
The estimate \e{N1f} follows from the fact that, for each fixed $\tau\in [0,1]$, 
$e^{-(1-\tau)\langle D_x\rangle}$ is a standard pseudo-differential operator of order $0$, hence bounded 
on the H\"older space $C^r$, provided that $r$ is not an integer. To prove~\e{N1}, notice that, for any 
$0<\lambda<\lambda+\varepsilon<1$,   
$|1-\tau|^{\lambda+\varepsilon} \langle D_x\rangle^\lambda e^{-(1-\tau)\langle D_x\rangle}$ is bounded from $L^\infty$ to $L^\infty$, uniformly in $\tau \in [0,1]$. Therefore, for every Littlewood-Paley building block operator $\Delta_j$, we have
\begin{align*}
\big| \Delta_j   e^{-(1-\tau)\langle D_x\rangle}\theta\big|_{L^\infty} 
&= \vert 1-\tau \vert^{-\lambda-\varepsilon} \big| (|1-\tau|^{\lambda+\varepsilon} \langle D_x\rangle^\lambda) e^{-(1-\tau)\langle D_x\rangle}\langle D_x\rangle^{-\lambda} \Delta_j \theta\big|_{L^\infty}\\
&\leq C_\varepsilon\vert 1-\tau \vert^{-\lambda-\varepsilon} \big| \langle D_x\rangle^{-\lambda}  \Delta_j \theta\big|_{L^\infty}.
\end{align*}
Taking into account the Littlewood-Paley characterization of H\"{o}lder functions, this implies 
\begin{equation}\label{Theta,(1-tau)}
|\Theta(\tau,x)|_{C^{r+\lambda}_x}\leq \frac{C_{\varepsilon,r}}{|1-\tau|^{\lambda+\varepsilon}}|\theta|_{C^r},
\end{equation}
hence the desired result. This proves the first statement.

As for the second statement, we note that the Fourier multiplier $e^{-y\langle D_x\rangle}$ is nothing but the Green operator for Helmholtz equation on the periodic half space $\mathbb{T}_x^n\times[0,+\infty)_y$: the vetor function $u(y,x):=e^{-y\langle D_x\rangle}\theta(x)$ satisfies
$$
\big(-\Delta_{y,x}+1\big)u(y,x)=0,
\quad\text{where }\Delta_{y,x}=\partial_y^2+\sum_{i=1}^n\partial_{x^i}^2
$$
inside $\mathbb{T}_x^n\times(0,+\infty)_y$, tending to 0 uniformly as $y\to+\infty$, with boundary value $u(0,x)=\theta(x)$. We then compute, with $|\partial_xu|^2=\sum_{i,j=1}^n|\partial_{x^i}u_j|^2$,
$$
\Delta_{y,x}|\partial_{x}u(y,x)|^2
=2|\partial_{x}u(y,x)|^2+2\sum_{i,j=1}^n|\nabla_x\partial_{x^i}u_j(y,x)|^2\geq0.
$$
By the maximum principle for Laplacian, we must have
$$
\sup_{x,y}|\partial_xu(y,x)|^2\leq\sup_x|\partial_xu(0,x)|^2
=\sup_x|\partial_x\theta(x)|^2.
$$
This proves the derivative bound for $\Theta$ by noticing $\Theta(\tau,x)=\tau u(1-\tau,x)$. Thus $\mathrm{Id}+\Theta(\tau,\cdot)$ is a $C^1$ diffeomorphism by the standard inverse function theorem, and $C^{r'+1}$ regularity of its inversion follows from inversion of $C^{r'+1}$ mapping (see Proposition \ref{2esti:C^scomp}).

The third statement is clearly a corollary of the first two.
\end{proof}

\subsection{Deformation of Symbol}
We now state a lemma that justifies the terminology \emph{deformation vector field}\footnote{The vector field $X$ is required to be non-autonomous, since \emph{not} all diffeomorphisms close to the identity can be embedded into an autonomous flow. Equivalently, the exponential map from autonomous vector fields (Lie algebra of the diffeomorphism group) to the diffeomorphism group is \emph{not} surjective near zero. See for example, Section 43 of \cite{KM1997}.} associated to $\chi(x)=x+\theta(x)$:

\begin{lemma}\label{L:X}
Let $\chi=\Id+\theta$, $\Theta$ and $X$ be as in in Definition \ref{Deformation}.

(1) For any function $f\in C^1(\xT^n)$, 
the deformation $f^{\sharp;\theta}(\tau,x)=f\big(x+\Theta(\tau,x)\big)\in C^1([0,1]\times\xT^n)$ solves the linear transport equation
\be\label{FcnTransp}
\partial_\tau f^{\sharp;\theta}+X\cdot\nabla_x f^{\sharp;\theta}=0.
\ee

(2) For any function $a(x,\xi)$ that is $C^1$ in $x\in \xT^n$ and smooth in $\xi\in\xR^n\setminus\{0\}$, 
the deformation $a^{\sharp;\theta}$ defined by 
$$
a^{\sharp;\theta}(\tau,x,\xi)=a\Big( x+\Theta(\tau,x),
\big(I_n+\partial_x\Theta(\tau,x)\big)^{-\T}\xi\Big)
$$
solves the linear transport equation
\be\label{SymTransp}
\partial_\tau a^{\sharp;\theta}+X\cdot\nabla_x a^{\sharp;\theta}-\nabla_\xi a^{\sharp;\theta}\cdot 
 \nabla_x(\xi \cdot X)=0.
\ee
Here $-\T$ stands for taking the inverse of transpose. 
\end{lemma}
\begin{remark}\label{HamiltonxiX}
The last identity (\ref{SymTransp}) could also be written as a Hamiltonian equation
$$
\partial_\tau a^{\sharp;\theta}+\left\{ H,a^{\sharp;\theta}\right\}=0.
$$
Here the time-dependent Hamiltonian $H$ on the symplectic manifold $\mathbb{T}^n\times\mathbb{R}^n$ is defined 
by
$$
H(\tau;x,\xi)=\xi \cdot  X(\tau,x),
$$ 
and $\{\cdot,\cdot\}$ denotes the Poisson bracket defined by
$$
\left\{a,b\right\}=\nabla_\xi a\cdot \nabla_x b-\nabla_x a\cdot
\nabla_\xi b.
$$
Thus (2) of Lemma \ref{L:X} states nothing but the following: the flow $\Phi$ of the Hamiltonian vector field of $\xi \cdot  X(\tau,x)$ on $\mathbb{T}^n\times\mathbb{R}^n$, defined by (see Proposition \ref{3RegTrans})
$$
\frac{\partial}{\partial\tau}\left(\begin{matrix}
\Phi^x(\tau,t;x,\xi) \\
\Phi^\xi(\tau,t;x,\xi)\
\end{matrix}\right)
=\left(\begin{matrix}
X\big(\tau,\Phi^x(\tau,t;x,\xi)\big) \\
-\partial_xX\big(\tau,\Phi^x(\tau,t;x,\xi)\big)\Phi^\xi(\tau,t;x,\xi)\
\end{matrix}\right),
\,
\Phi(t,t;x,\xi)=\left(\begin{matrix}
x \\
\xi
\end{matrix}\right),
$$
in fact takes an explicit form of a type I canonical transformation
$$
\Phi(0,\tau;x,\xi)=\left(\begin{matrix}
x+\Theta(\tau,x) \\
\big(I_n+\partial_x\Theta(\tau,x)\big)^{-\T}\xi
\end{matrix}\right).
$$
In particular, the flow acts as linear transformation with respect to $\xi$, and thus the zero section $\{\xi=0\}$ is an invariant submanifold of the flow $\Phi$.
\end{remark}
\begin{proof}
The first statement is clearly a particular case of the second, so it suffices to prove the latter. We introduce the matrix $A(\tau,x)=\big(I_n+\partial_x\Theta(\tau,x)\big)^{-\T}$. The chain rule reads 
$$
\begin{aligned}
\partial_\tau a^{\sharp;\theta}+X\cdot\nabla_x a^{\sharp;\theta}
&=(\mathrm{I})+(\mathrm{II})\\
&:=(\nabla_x a)^{\sharp;\theta} (x,\xi)
\cdot (\partial_\tau+X\cdot\nabla_x)\big(x+\Theta(\tau,x)\big)\\
&\quad +(\nabla_\xi a)^{\sharp;\theta} (x,\xi)
\cdot (\partial_\tau+X\cdot\nabla_x)\big(A(\tau,x)\xi\big).
\end{aligned}
$$
Now we have $(\mathrm{I})=0$, since directly from the definition of the vector field $X$ (see~\e{N10}),
\be\label{N20}
(\partial_\tau+X\cdot\nabla_x)\big(x+\Theta(\tau,x)\big)=0.
\ee

To handle the term $(\mathrm{II})$, we begin by using an elemetary result: for any invertible constant matrix $A=(a_{ij})_{1\le i,j\le n}$ and any function $v=v(\xi)$, there holds
$$
\nabla_\xi \big(v(A\xi)\big)=A^\T(\nabla_\xi v)(A\xi) 
\quad\Longleftrightarrow\quad
(\nabla_\xi v)(A\xi)=
A^{-\T}\nabla_\xi \big(v(A\xi)\big).
$$
We deduce that
$$
(\nabla_\xi a)^{\sharp;\theta} (x,\xi)
=A(\tau,x)^{-\T}\nabla_\xi (a^{\sharp;\theta}(x,\xi))
=\big(I_n+\partial_x\Theta(\tau,x)\big)\nabla_\xi (a^{\sharp;\theta}(x,\xi)).
$$
Therefore the term $(\mathrm{II})$ simplifies to
\begin{align*}
(\mathrm{II})
&=\left[\big(I_n+\partial_x\Theta(\tau,x)\big)\nabla_\xi (a^{\sharp;\theta}(x,\xi))\right]
\cdot \left[(\partial_\tau+X\cdot\nabla_x)\big(A(\tau,x)\xi\big)\right]\\
&=
\nabla_\xi \big(a^{\sharp;\theta}(x,\xi)\big)
\cdot \left[\big(I_n+\partial_x\Theta(\tau,x)\big)^{\T}
\cdot(\partial_\tau+X\cdot\nabla_x)\big(A(\tau,x)\xi\big)\right].
\end{align*}
Note that $\cdot$ stands for inner product between vectors in $\xR^n$. This is why the transpose $\big(I_n+\partial_x\Theta(\tau,x)\big)^{\T}$ appears. 

Next, we commute the matrix $\big(I_n+\partial_x\Theta(\tau,x)\big)^{\T}$ with the differential operator $(\partial_\tau+X\cdot\nabla_x)$. From the identity
$$
\big(I_n+\partial_x\Theta(\tau,x)\big)^{\T}A(\tau,x)=I_n,
$$
we obtain that the bracket in the right-hand-side of (II) equals
$$
\begin{aligned}
\big(I_n+\partial_x\Theta(\tau,x)\big)^{\T}&
\cdot(\partial_\tau+X\cdot\nabla_x)(A(\tau,x)\xi)\\
&=
-\big[(\partial_\tau+X\cdot\nabla_x)\big(I_n+\partial_x\Theta(\tau,x)\big)^{\T}\big]\cdot A(\tau,x)\xi.
\end{aligned}
$$
By differentiating the identity~\e{N20} with respect to $x_j$, we get for any $1\le j,k\le n$, 
$$
\begin{aligned}
(\partial_\tau+X\cdot\nabla_x)(\delta_j^k+\partial_{j}\Theta_k)
&=-\partial_{x_j}X\cdot \nabla_x (x_k+\Theta_k)\\
&=-\sum_l\partial_{x_j}X_l 
(\delta_k^l+\partial_l \Theta_k),
\end{aligned}
$$
so
$$
\begin{aligned}
(\partial_\tau+X\cdot\nabla_x)(I_n+\partial_x\Theta)^\T=-(\partial_xX)(I_n+\partial_x\Theta)^\T.
\end{aligned}
$$
We conclude that
$$
(\mathrm{II})=\nabla_\xi a^{\sharp;\theta} 
\cdot \big(
(\partial_xX)\xi\big)=
\nabla_\xi a^{\sharp;\theta} 
\cdot\nabla_x (\xi\cdot X).
$$
This completes the proof.
\end{proof}

\subsection{Paracomposition}
We are now prepared to provide an alternative definition of paracomposition operator in this subsection. For concreteness, we confine ourselves in presenting the theory on $\xT^n$ to avoid issues at infinity. 

In the last subsection, given a diffeomorphism $\chi(x)=x+\theta(x)$, we defined an isotopy $x+\Theta(\tau,x)$ and the corresponding deformation vector field $X(\tau,x)$ (\ref{N10}) as in Definition \ref{Deformation}. Lemma \ref{L:X} shows that the composition $f\circ \chi$ is in fact the value at time $\tau=1$ of the solution to the transport equation~\e{FcnTransp} with initial value $f$. 

To define the paracomposition operator $\chi^\star$ associated to $\chi$, we consider a paradifferential version of~\e{FcnTransp} and define $\chi^\star f$ as the value at time $1$ of the solution to the latter equation. 

\begin{definition}\label{DefParacomposition}
Suppose the diffeomorphism $\chi=\mathrm{Id}+\theta$ is such that
$$
\theta\in C^r\quad\text{for some }r\in (1,+\infty)\setminus\xN,
\quad \sup_{x\in\xT ^n}\la \partial_x\theta(x)\ra\leq0.99.
$$
Set the homotopy $\Theta$ and the deformation vector field $X$ as in Definition \ref{Deformation}:
$$
\Theta(\tau,x)=\tau e^{-(1-\tau)\langle D_x\rangle}\theta(x),
\quad
X(\tau,x)=-\big(I_n+\partial_x\Theta(\tau,x)\big)^{-1}\partial_\tau\Theta(\tau,x).
$$
For fixed $s\in\xR$, the action of paracomposition operator $\chi^\star$ on $f\in H^s$ is defined to be the time 1 flow map of the Cauchy problem
\be\label{n20-bis}
\partial_\tau w+ T_{X}\cdot \nabla_x w=0,\quad w\Big|_{\tau=0}=f.
\ee
In other words,
$$
\chi^\star f=w(1,\cdot).
$$
\end{definition}

An important feature of $\chi^\star$ is that it is a well-defined bounded invertible linear operator in \emph{all} Sobolev spaces, and the bound requires only ``minimal" regularity of $\chi$:

\begin{proposition}[Boundedness of paracomposition operator]\label{PCNorm}
Assume that $\chi$ is as in Definition \ref{DefParacomposition}. Then $\chi^\star$ 
is a bounded invertible linear operator from $H^s$ to $H^s$ for all $s\in\mathbb{R}$:
$$
\|\chi^\star\|_{\mathcal{L}(H^s,H^s)}\lesssim_{s} |\chi|_{C^r},
\quad
\|(\chi^\star)^{-1}\|_{\mathcal{L}(H^s,H^s)}\lesssim_{s} |\chi|_{C^r}.
$$
\end{proposition}
\begin{proof}
Recalling (\ref{n20-bis}), Proposition \ref{ParaTranspEq} implies the first inequality $\|\chi^\star\|_{\mathcal{L}(H^s,H^s)}\lesssim_{s} |\chi|_{C^r}$. As for invertibility, we may simply follow standard argument for ``fundamental solutions". For $\tau,\tau_0\in [0,1]$, we denote by $\Xi(\tau,\tau_0)f$ the solution $w(\tau,\cdot)$ of (\ref{n20-bis}) such that $w(\tau_0,\cdot)=f(\cdot)$. Proposition \ref{ParaTranspEq} shows that this evolution problem is always uniquely solvable, whether $\tau\geq\tau_0$ or not, so $\Xi(\tau,\tau_0)$ is a well-defined bounded linear operator satisfying energy estimates. Furthermore, uniqueness of solution implies the composition property $\Xi(\tau,\tau_0)\Xi(\tau_0,\tau')=\Xi(\tau,\tau')$. Taking $\tau=\tau'=1,\tau_0=0$, we obtain $\chi^\star\Xi(0,1)=\mathrm{Id}$, which shows that $\Xi(0,1)$ inverts $\chi^\star$ and still satisfies energy estimates.
\end{proof}

We can also determine the continuity of $\chi^\star$ with respect to the diffeomorphism $\chi$:

\begin{proposition}\label{PCNormCont}
The assignment $\chi\mapsto\chi^\star$, viewed as a mapping starting from $C^r$ diffeomorphisms to the space of operators, is Lipschitz continuous at the price of losing one derivative. More precisely, assume that $\chi_1=\mathrm{Id}+\theta_1$, $\chi_2=\mathrm{Id}+\theta_2$ are as in Definition \ref{DefParacomposition}. Then 
$$
\|(\chi_1^\star-\chi_2^\star)f\|_{H^{s-1}}
\lesssim_{r,s}|\theta_1-\theta_2|_{C^r} \|f\|_{H^s},
\quad \forall s\in\xR.
$$
Consequently, the assignment $\chi\mapsto\chi^\star$ is
\begin{itemize}
    \item continuous, viewed as a mapping from $C^r$ to $\mathcal{L}(H^s,H^s)$, with $\mathcal{L}(H^s,H^s)$ topologized by the strong operator topology\footnote{Recall that a sequence of bounded operators $A_k$ in a Banach space converges to some $A$ in the strong operator topology, if $\|(A_k-A)x\|\to0$ for any element $x$. This is weaker than convergence in operator norm.};
    \item $C^\alpha$-H\"{o}lder continuous (here $0<\alpha\leq1$), viewed as a mapping from $C^r$ to $\mathcal{L}(H^s,H^{s-\alpha})$, with $\mathcal{L}(H^s,H^{s-\alpha})$ topologized by the operator norm.
\end{itemize}
\end{proposition}
\begin{remark}
It is natural to expect that Lipschitz dependence loses one derivative from Taylor's formula $f\circ\chi_1-f\circ\chi_2\simeq f'\circ\chi_2\cdot(\theta_1-\theta_2)$.
\end{remark}
\begin{proof}
Note that $\chi_i^\star f$ is the solution of paratransport equation (\ref{n20-bis}). By Lemma \ref{L:rho}, we can estimate the $L^1_\tau C^1_x$ norm of $X_1-X_2$ in terms of $|\theta_1-\theta_2|_{C^r}$. Corollary \ref{ParaTranspEqCor1} then implies the desired Lipschitz estimate 
$$
\|(\chi_1^\star-\chi_2^\star)f\|_{H^{s-1}}
\lesssim_s|X_1-X_2|_{L^1_\tau C^1_x}\|f\|_{H^s}
\lesssim_{s,r}|\theta_1-\theta_2|_{C^r}\|f\|_{H^s},
$$
which gives the Lipschitz continuity from $C^r$ to $\mathcal{L}(H^s,H^{s-1})$. The H\"{o}lder continuity for $0<\alpha<1$ follow from interpolation inequalities in $H^s$. 

The proof of continuity of the assignment $\chi\mapsto\chi^\star$ under the strong operator topology, namely\footnote{Note that the strong operator topology in $\mathcal{L}(H^s,H^s)$ restricted to a bounded set is metrizable.} $\|(\chi_k^\star-\chi^\star)f\|_{H^s}\to0$ when $\chi_k\to\chi$ in $C^r$, is a direct consequence of Proposition \ref{ParaTranspCont}. In fact, if $\chi_k\to\chi$ in $C^r$, then the corresponding deformation vector fields $X_k\to X$ in $L^1_\tau C^{r'}_x$ for any $r'\in[1,r)$. Proposition \ref{ParaTranspCont} is then directly applicable.
\end{proof}

We note that this alternative definition of paracomposition operator is slightly less general compared to the original one (\ref{FIO}), as the diffeomorphism is required to be of class $C^r$ instead of Lipschitz. However, in practice, this alternative definition is not undermined since a calculus for paracomposition operators inherently requires regularity beyond Lipschitz continuity. On the other hand, in the integral equation (\ref{W_Int}), it can be proved with a little bit more effort that the operator $B(X)$ in fact gains back $r-1$ derivatives. When evaluated at time 1, the expression 
$$
A(X)f=\sum_{j\geq0}\zeta_j(D_x)(\Delta_jf)\circ\chi_j,
\quad \chi_j\text{ a smooth approximating sequence to }\chi
$$
also differs from (\ref{FIO}) by a $(r-1)$-smoothing term. Therefore, the alternative definition of $\chi^\star$ via flow map is essentially the same as (\ref{FIO}).

\section{Calculus of Paracomposition Operators}\label{sec:4}

In this section, we will develop symbolic calculus of paracomposition operators.

\subsection{Refinement of Paralinearization} We start by presenting a refinement of Bony's paralinearization theorem \ref{2ParaLin}, which, in the original context of Alinhac \cite{Alinhac1986}, leads to the very definition of paracomposition, as discussed in Subsection \ref{sec:1.2}.

\begin{theorem}[Refinement of Paralinearization]\label{ParaLinRefined}
Let $s>n/2+1$, $r>1$. Let $\chi=\mathrm{Id}+\theta$ be a $C^r$ diffeomorphism on $\xT^n$ as in Definition \ref{DefParacomposition}. Suppose $f\in H^s$. Then there holds the following refinement of Bony's paralinearization formula:
$$
f\circ\chi
=T_{f'\circ\chi}\theta+\chi^\star f+\R_{\PLR}(f,\theta).
$$
Here PLR is the abbreviation for \emph{para-linéarisation raffin\'ee}, and the remainder $\R_{\PLR}(f,\theta)$ is smoothing in the following sense: set $\rho=\min(s-n/2-1,r-1)$, then for any $r'\in(1,r)$,
$$
\|\R_{\PLR}(f,\theta)\|_{H^{r'+\rho}}
\lesssim_{r',r,s} |\theta|_{C^r}\|f\|_{H^s}.
$$
Furthermore, if in addition $s>n/2+2$, then the remainder $\R_\PLR(f,\theta)$ is Lipschitz with respect to $\theta$ at the price of losing one derivative: with $\rho'=\min(s-n/2-2,r-1)$,
$$
\|\R_\PLR(f,\theta_1)-\R_\PLR(f,\theta_2)\|_{H^{r'+\rho'}}
\lesssim_{r',r,s} \|f\|_{H^s}|\theta_1-\theta_2|_{C^r}.
$$
\end{theorem}
This refined paralinearization formula states that, if $f$ is only of limited differentiability, then $\chi^\star f$ captures the most irregular part ``with respect to $f$" in the remainder $f\circ\chi-T_{f'\circ\chi}\theta$. In view of Ainhac's decomposition (\ref{AlinhacParalin}), the alternative definition for $\chi^\star$ essentially gives the same object as the original Fourier integral one (\ref{FIO}) does, at least on sufficiently regular Sobolev functions. 

Before we give the proof, let us state an auxiliary result:
\begin{lemma}\label{fTheta}
Let $r\in(1,+\infty)\setminus\xN$, $\chi_i=\mathrm{Id}+\theta_i$, $i=1,2$ be $C^r$ diffeomorphisms on $\xT^n$ as in Definition \ref{DefParacomposition}. Suppose $F$ is a $C^\mu$ function on $\xT^n$ with $\mu>1$. Then 
$$
\big|F\circ\chi_1-F\circ\chi_2\big|_{C^{\min(\mu-1,r)}}
\lesssim |F|_{C^\mu}|\theta_1-\theta_2|_{C^r}.
$$
\end{lemma}
\begin{proof}
We write
$$
F\circ\chi_1(x)-F\circ\chi_2(x)
=\left(\int_0^1 F'\big(x+t\theta_2(x)+(1-t)\theta_1(x)\big)\dt\right)\big(\theta_2(x)-\theta_1(x)\big).
$$
Then due to the H\"{o}lder estimate (Proposition \ref{2esti:C^scomp}) for composite functions, the $C^{\min(\mu-1,r)}$ norm of the integral is controlled by $|F|_{C^\mu}$. The result then follows from the usual product H\"{o}lder estimate.
\end{proof}

\begin{proof}[Proof of Theorem \ref{ParaLinRefined}]
We want to compare the linear transport equation
$$
\partial_\tau f^{\sharp;\theta}+X\cdot\nabla_x f^{\sharp;\theta}=0
$$
with the linear paratransport equation
$$
\partial_\tau w+T_X\cdot\nabla_x w=0,
$$
of the same initial value $f$ at $\tau=0$. Here as in (\ref{homotopy}) and (\ref{N10}),
$$
\Theta(\tau,x)=\tau e^{-(1-\tau)\langle D_x\rangle}\theta,
\quad
X(\tau,x)=-\big(I_n+\partial_x\Theta(\tau,x)\big)^{-1}\partial_\tau\Theta(\tau,x).
$$
Subtracting one equation from another, we obtain the equation for the difference $f^{\sharp;\theta}-w$:
\begin{equation}\label{PLRemTransp}
\partial_\tau(f^{\sharp;\theta}-w)
+T_{X}\cdot \nabla_x (f^{\sharp;\theta}-w)
+X\cdot \nabla_x f^{\sharp;\theta}-T_{X}\cdot \nabla_x f^{\sharp;\theta}=0,
\quad (f^{\sharp;\theta}-w)\Big|_{\tau=0}=0.
\end{equation}
It then suffices to study the paratransport equation (\ref{PLRemTransp}) in view of Proposition \ref{ParaTranspEq}.

To control the regularity of $f^{\sharp;\theta}-w$, we paralinearize the product $X\cdot\nabla_xf^{\sharp;\theta}$ to get 
\be\label{T_XDu}
X\cdot \nabla_x f^{\sharp;\theta}-T_{X}\cdot \nabla_x f^{\sharp;\theta}=
T_{\nabla_x f^{\sharp;\theta}}\cdot X+R_\PM(X,\nabla_x f^{\sharp;\theta}).
\ee
By Lemma \ref{L:rho} we have $X\in L^1_tC^{r'}_x$ for any $r'<r$. On the other hand, since $\rho=\min(s-n/2-1,r-1)$, we may use the H\"{o}lder composition estimate (Proposition \ref{2esti:C^scomp}) to estimate
$$
|\nabla_xf^{\sharp;\theta}|_{C^0_\tau C^\rho_x}
\lesssim_{r}|f|_{C^{\rho+1}_x}
+|f|_{C^1}\big(1+|\Theta(\tau,x)|_{C^0_\tau C^r_x}\big)
\lesssim \|f\|_{H^s}+|f|_{C^1}|\theta|_{C^r}.
$$
Therefore by Theorem \ref{PMReg},
$$
\big\| R_\PM(X,\nabla_x f^{\sharp;\theta})\big\|_{H^{r'+\rho}_x}
\les 
|X|_{C^{r'}_x}\|f\|_{H^s}.
$$
Hence
$$
\big\| R_\PM(X,\nabla_x f^{\sharp;\theta})\big\|_{L^1_\tau H^{r'+\rho}_x}
\les |\theta|_{C^r_x}\|f\|_{H^s}.
$$
It then remains to study the paraproduct term $T_{\nabla_x f^{\sharp;\theta}}\cdot X$ in (\ref{T_XDu}).

To this end we use the already well-established paralinearization formula, namely Theorem \ref{2ParaLin}, to obtain\footnote{Recall from (\ref{DefOpPM}) that $\Op^\PM(a)$ is simply another way of writing $T_a$. Strictly speaking we should add subscripts left/right, but we abbreviate them here since the order of action is clear in the context.}
$$\begin{aligned}
X
&=-(I_n+\partial_x\Theta)^{-1}\partial_\tau\Theta\\
&= 
-\Op^\PM\left((I_n+\partial_x\Theta)^{-1}\right)\partial_\tau\Theta\\
&\quad-\Op^\PM\left((I_n+\partial_x\Theta)^{-1}\bullet X\right)\partial_x\Theta
+\R_1(\partial_\tau\Theta,\partial_x\Theta),
\end{aligned}
$$
where the matrix operation $(A\bullet B)C:=ACB$, the remainder $\R_1(\partial_\tau\Theta,\partial_x\Theta)$ is $C^1$ in $\partial_\tau\Theta,\partial_x\Theta$, satisfying
$$
\big\|\R_1(\partial_\tau\Theta,\partial_x\Theta)\big\|_{L^1_\tau H^{r+r'-1}_x}
\lesssim
\big\|\R_1(\partial_\tau\Theta,\partial_x\Theta)\big\|_{L^1_\tau C^{r+r'-1}_x}
\lesssim|\theta|_{C^r}.
$$
Observing
$$
(\nabla_x f^{\sharp;\theta})(\tau,x)= f'(x+\Theta(\tau,x))\cdot\big(I_n+\partial_x\Theta(\tau,x)\big),
$$
we find
\be\label{T_DuX}
\begin{aligned}
T_{\nabla_x f^{\sharp;\theta}}\cdot X
&= -\Op^\PM\big(f'(x+\Theta)\cdot\big(I_n+\partial_x\Theta\big)\big)
\cdot\Op^\PM\big((I_n+\partial_x\Theta)^{-1}\big)\partial_\tau\Theta\\
&\quad 
-\Op^\PM\big(f'(x+\Theta)\cdot(I_n+\partial_x\Theta)\big)
\cdot\Op^\PM\left((I_n+\partial_x\Theta)^{-1}\bullet X\right)\partial_x\Theta\\
&\quad+T_{\nabla_xf^{\sharp;\theta}}\R_1(\partial_\tau\Theta,\partial_x\Theta)\\
&=:-T_{f'(x+\Theta)}
\big(\partial_\tau\Theta+(T_{X}\cdot\nabla_x)\Theta\big)
+\R_2(f;\partial_\tau\Theta,\partial_x\Theta).
\end{aligned}
\ee
Here the remainder 
\begin{equation}\label{PLAR2}
\begin{aligned}
\R_2(f;\partial_\tau\Theta,\partial_x\Theta)
&:=T_{\nabla_xf^{\sharp;\theta}}\R_1(\partial_\tau\Theta,\partial_x\Theta)\\
&\quad-\R_\CM\big(f'(x+\Theta)(I_n+\partial_x\Theta),(I_n+\partial_x\Theta)^{-1}\big)\partial_\tau\Theta\\
&\quad-\R_\CM\big(f'(x+\Theta)(I_n+\partial_x\Theta),(I_n+\partial_x\Theta)^{-1}\bullet X\big)\partial_x\Theta.
\end{aligned}
\end{equation}
Recalling $\rho=\min(s-n/2-1,r-1)$, we find that $f'(x+\Theta)\in C^0_\tau C^{\rho}_x$, with norm controlled by $\|f\|_{H^s}$, so by the symbolic calculus theorem \ref{2theo:sc0}, $\R_2$ enjoys the following regularity:
$$
\|\R_2(f;\partial_\tau\Theta,\partial_x\Theta)\|_{L^1_\tau H^{r'+\rho}_x}
\lesssim\|\R_2(f;\partial_\tau\Theta,\partial_x\Theta)\|_{L^1_\tau C^{r'+\rho}_x}
\lesssim \|f\|_{H^s}|\theta|_{C^r}.
$$

Now the usual Leibniz rule for paraproducts implies
$$
\begin{aligned}
\big(\partial_\tau+T_X\cdot\nabla_x\big)T_{f'(x+\Theta)}
&=\Op^\PM\big((\partial_\tau+X\cdot\nabla_x) f'(x+\Theta)\big)
+T_{f'(x+\Theta)}\big(\partial_\tau+T_X\cdot\nabla_x\big)\\
&=T_{f'(x+\Theta)}\big(\partial_\tau+T_X\cdot\nabla_x\big),
\end{aligned}
$$
since 
$$
\big(\partial_\tau+X(\tau,x)\cdot\nabla_x\big)f'\big(x+\Theta(\tau,x)\big)=0.
$$
Consequently, (\ref{T_DuX}) becomes
$$
T_{\nabla_x f^{\sharp;\theta}}\cdot X
=-\big(\partial_\tau+T_X\cdot\nabla_x\big)T_{f'(x+\Theta)}\Theta
+\R_2(f;\partial_\tau\Theta,\partial_x\Theta).
$$
Substituting this into (\ref{T_XDu}), then back to (\ref{PLRemTransp}), we find that the function $v=(f^{\sharp;\theta}-w)-T_{f'(x+\Theta)}\Theta$ satisfies
\begin{equation}\label{ParaTransp(v)}
\partial_\tau v+T_X\cdot\nabla_xv=R_\PM(X,\nabla_xf^{\sharp;\theta})
+\R_2(f;\partial_\tau\Theta,\partial_x\Theta),
\quad v\Big|_{\tau=0}=0,
\end{equation}
where the right-hand-side is in $L^1_\tau C^{r'+\rho}_x\subset L^1_\tau H^{r'+\rho}_x$. The estimate for $\R_\PLR(f,\theta)$ then follows from Proposition \ref{ParaTranspEq} by evaluating $v$ at $\tau=1$.

As for the Lipschitz dependence of $\R_\PLR(f,\theta)$ under the stronger condition $f\in H^s$, $s>n/2+2$, we simply notice that the right-hand-side of (\ref{ParaTransp(v)}) is Lipschitz continuous with respect to $\theta$ at the price of losing one derivative. This is because the mapping $\theta\to X$ is smooth by Lemma \ref{L:rho}, while we have by Lemma \ref{fTheta}, recalling $\rho'=\min(s-n/2-2,r-1)$,
$$
\begin{aligned}
|f'(x+\Theta_1(\tau,x))-f'(x+\Theta_2(\tau,x))|_{C_x^{\rho'}}
&\lesssim \|f\|_{H^s}|\theta_1-\theta_2|_{C^r}.
\end{aligned}
$$
Concerning the expression (\ref{PLAR2}) of $\R_2$, Theorem \ref{2theo:sc0} and \ref{PMReg} implies
$$
\begin{aligned}
\left\|R_\PM\big(X_1,\nabla_x[f(x+\Theta_1)]\big)-R_\PM\big(X_2,\nabla_x[f(x+\Theta_2)]\big)\right\|_{L^1_\tau H^{r'+\rho'}_x}
&\lesssim \|f\|_{H^s}|\theta_1-\theta_2|_{C^r}\\
\left\|\R_2(f;\partial_\tau\Theta_1,\partial_x\Theta_1)-\R_2(f;\partial_\tau\Theta_1,\partial_x\Theta_2)\right\|_{L^1_\tau H^{r'+\rho'}_x}
&\lesssim \|f\|_{H^s}|\theta_1-\theta_2|_{C^r}.
\end{aligned}
$$
Then Corollary \ref{ParaTranspEqCor1} implies the desired result.
\end{proof}
\subsection{Parachange of Variable for Paraproduct}
We now state and prove the formula for the conjugation of a paraproduct operator with a paracomposition. This result is a special case of the more general Theorem \ref{4PCConj} proved later. However, since the proof for this special case is technically much simpler, it deserves separate treatment.

\begin{theorem}[Conjugation with Paraproduct]\label{ParacompoPrd}
Suppose $r>1$ is not an integer and let $\sigma>1$. Let $\chi=\mathrm{Id}+\theta$ be a $C^r$ diffeomorphism on $\xT^n$, with $\theta$ as in Definition \ref{DefParacomposition}. Let $a\in C^\sigma$ be a function on $\xT^n$. Then there holds the following conjugation formula:
$$
\chi^\star T_a=T_{a\circ\chi}\chi^\star+\R_{\Con}(a,\chi),
$$
where the operator $\R_\Con(a,\chi)$ is smoothing: given $r'<r$ and $\mu=\min(\sigma,r'+1)$, there holds,
$$
\big\|\R_\Con(a,\chi);\mathcal{L}(H^s,H^{s+\mu-1})\big\|
\lesssim_{s,r,r',\sigma}|a|_{C^\sigma}|\theta|_{C^r},
\quad \forall s\in\mathbb{R}.
$$
Furthermore, if $\chi_i=\Id+\theta_i$, $i=1,2$ satisfy the requirement of Definition \ref{DefParacomposition}, we have the following Lipschitz dependence of $\R_\Con(a,\chi)$: with $r'<r$, $\mu_1=\min(\sigma-1,r'+1)$, there holds
$$
\big\|\R_\Con(a,\chi_1)-\R_\Con(a,\chi_2);\mathcal{L}(H^s,H^{s+\mu_1-1})\big\|
\lesssim_{s,r,r',\sigma}|a|_{C^\sigma}|\theta_1-\theta_2|_{C^r}.
$$
\end{theorem}
\begin{remark}
If $\sigma\geq r+1$, the regularity gain due to $\R_\Con(a,\chi)$ is $r'$, only slightly less than the regularity of $\chi$. This results from the almost full derivative gained through $\Theta(\tau,x)$, as indicated by Proposition \ref{L:rho}: the deformation vector field $X$ is only slightly less smooth than $\chi$ itself. In contrast, the simplest homotopy $\tau\theta(x)$ yields less gain in regularity.
\end{remark}
\begin{proof}
We set $\Theta(\tau,x)$ and the deformation vector field $X(\tau,x)$ are as in Definition \ref{Deformation}:
$$
\Theta(\tau,x)=\tau e^{-(1-\tau)\langle D_x\rangle}\theta(x),
\quad
X(\tau,x)=-\big(I_n+\partial_x\Theta(\tau,x)\big)^{-1}\partial_\tau\Theta(\tau,x).
$$
Furthermore, given $s\in\mathbb{R}$ and $f\in H^s$, we still denote by $w\in C^0([0,1];H^s)\cap C^1([0,1];H^{s-1})$ the unique solution to
$$
\partial_\tau w+T_{X}\cdot\nabla_x w=0,\quad w\Big|_{\tau=0}=f,
$$
so that $w(1,\cdot)=\chi^\star f$, as in Definition \ref{DefParacomposition}.

As in Lemma \ref{L:X}, we set $a^{\sharp;\theta}(\tau,x)=a\big(x+\Theta(\tau,x)\big)$. Then a direct computation yields 
\begin{equation}\label{TawEq}
(\partial_\tau+T_X\cdot\nabla_x) T_{a^{\sharp;\theta}}w
=\R_{\CM}(X,\nabla_x a^{\sharp;\theta})w+[T_X,T_{a^{\sharp;\theta}}]\cdot\nabla_xw.
\end{equation}
Here $\R_\CM$ is defined as in Theorem \ref {2theo:sc}. Since $r'<r$, we know that $\Theta(\tau,x)\in L^1_\tau C^{r'+1}_x$ by Proposition \ref{L:rho}. Taking into account the definition $\mu=\min(\sigma,r'+1)$, and using the product and composition estimates in H\"{o}lder spaces, namely Proposition \ref{2esti:C^scomp}, we obtain
$$
|a^{\sharp;\theta}|_{C^0_\tau C^{1}_x}
\lesssim_{r,\sigma,\mu}|a|_{C^\sigma},
\quad
|a^{\sharp;\theta}|_{L^1_\tau C^\mu_x}
\lesssim_{r,\sigma,\mu}|a|_{C^\sigma}|\theta|_{C^r_x}.
$$
We then notice that $\R_\CM$ satisfies \emph{tame} estimates as indicated in Theorem \ref {2theo:sc}. For example, using the $L^1_\tau C^{r'}_x$ estimates for $X$ as in Proposition \ref{L:rho}, we have
\begin{equation}\label{TameXa}
\begin{aligned}
\int_0^1\|\R_\CM & (X,\nabla_x a^{\sharp;\theta})w\|_{H^{s+\mu-1}_x}\dtau \\
&\overset{(\ref{esti:quant2})}{\lesssim_{s,r,r',\sigma}}
\int_0^1\left(|X|_{C^{\mu-1}_x}|\nabla_xa^{\sharp;\theta}|_{C^0}+|X|_{C^0_x}|\nabla_xa^{\sharp;\theta}|_{C^{\mu-1}_x}\right)\|w\|_{H^s_x}\dtau\\
&\lesssim_{s,r,r',\sigma}\left(|a^{\sharp;\theta}|_{C^0_\tau C^{1}_x}|X|_{L^1_\tau C^{\mu-1}_x}+|X|_{C^0_\tau C^0_x}|a^{\sharp;\theta}|_{L^1_\tau C^\mu_x}\right)\|w\|_{C^0_\tau H^s_x}\\
&\lesssim_{s,r,r',\sigma} |a|_{C^\sigma}|\theta|_{C^r_x}\|f\|_{H^s}.
\end{aligned}
\end{equation}
A similar estimate holds for $[T_X,T_{a^{\sharp;\theta}}]\cdot\nabla_x w$.

The right-hand-side of (\ref{TawEq}) thus enjoys the estimate
$$
\int_0^1\big\|\R_{\CM}(X,\nabla_x a^{\sharp;\theta})w+[T_X,T_{a^{\sharp;\theta}}]\cdot\nabla_xw\big\|_{H^{s+\mu-1}_x}\dtau
\lesssim_{s,r,r',\sigma}
|a|_{C^\sigma}|\theta|_{C^r}\|f\|_{H^s}.
$$
Noting that $T_{a^{\sharp;\theta}}w\big|_{\tau=0}=T_af$, we conclude from Proposition \ref{ParaTranspEq} that 
$$
\big\|T_{a^{\sharp;\theta}}w\big|_{\tau=1}-\chi^\star T_af\big\|_{H^{s+\mu-1}}
\lesssim_{s,r,r',\sigma}
|a|_{C^\sigma}|\theta|_{C^r}\|f\|_{H^s}.
$$
Noticing that $T_{a^{\sharp;\theta}}w\big|_{\tau=1}=T_{a\circ\chi}\chi^\star f$, 
we obtain the wanted result.

In order to prove Lipschitz dependence, it suffices to determine the Lipshictz dependence of the right-hand-side of (\ref{TawEq}). We make use of Lemma \ref{fTheta} again. Since $\mu_1=\min(\sigma-1,r'+1)$, we obtain
$$
\begin{aligned}
\big\|a^{\sharp;\theta_1}-a^{\sharp;\theta_2}\big\|_{C^{\mu_1}_x}
&\lesssim_{r,r',\sigma}|a|_{C^\sigma}|\Theta_1(\tau,x)-\Theta_2(\tau,x)|_{C^{r'+1}_x}.
\end{aligned}
$$
Repeating the proof for (\ref{TameXa}), we obtain the Lipschitz dependence of the right-hand-side of (\ref{TawEq}): for example,
$$
\begin{aligned}
\int_0^1&\big\|\R_\CM  (X_1,\nabla_x a^{\sharp;\theta_1})-\R_\CM (X_2,\nabla_x a^{\sharp;\theta_2});\mathcal{L}(H_x^s,H_x^{s+\mu_1-1})\big\|\dtau\\
&\overset{(\ref{esti:quant2})}{\lesssim_{s,r,r',\sigma}}
\int_0^1\left(|X_1-X_2|_{C^{\mu_1-1}_x}|\nabla_xa^{\sharp;\theta_1}|_{C^0}+|X_2|_{C^0_x}|\nabla_xa^{\sharp;\theta_1}-\nabla_xa^{\sharp;\theta_2}|_{C^{\mu_1-1}_x}\right)\dtau\\
&\lesssim_{s,r,r',\sigma}|a^{\sharp;\theta_1}|_{C^0_\tau C^{1}_x}|X_1-X_2|_{L^1_\tau C^{\mu_1-1}_x}+|X|_{C^0_\tau C^0_x}|a^{\sharp;\theta_1}-a^{\sharp;\theta_2}|_{L^1_\tau C^{\mu_1}_x}\\
&\lesssim_{s,r,r',\sigma} |a|_{C^\sigma}|\theta_1-\theta_2|_{C^r_x}.
\end{aligned}
$$
Therefore, the right-hand-side of (\ref{TawEq}), considered as a function taking value in $L^1_\tau H^{s+\mu_1-1}_x$, has Lipschitz dependence on the diffeomorphism $\chi\in C^r$.
\end{proof}

\subsection{General Conjugation Formula} 
We are finally ready to develop the conjugation property of paracomposition operators on the symbol class $\Sigma^m_r$ introduced in Definition~\ref{D:3.5}. Recall that $\Sigma^m_r$ consists of symbols $p$ of the form
$$
p(x,\xi)=p_m(x,\xi)+\cdots +p_{m-[r]}(x,\xi),\quad m\in \xR,~r\in [0,+\infty),
$$
where $[r]$ is integer part of $r$, and $p_{m-k}(x,\xi)$ is homogeneous of degree $m-k$ in $\xi$ with regularity $C^{r-k}$ in $x$. Here ``homogeneous" is understood in the sense of \ref{D:3.5}.

Before presenting the conjugation formula, we first cite the change-of-variable formula for standard pseudo-differential operators of type (1,0) for comparison; see for example, Theorem 2.1.2 of \cite{Hormander1971}, or Proposition 7.1 of Chapter I of \cite{AG}:

\begin{theorem}\label{Change(1,0)}
Suppose $p(x,\xi)\in\mathscr{S}^m_{1,0}$ is a polyhomogeneous symbol smooth in $(x,\xi)\in\xT^n\times\xR^n\setminus\{0\}$; that is, $p(x,\xi)$ admits the asymptotic expansion
$$
p(x,\xi)\sim p_m(x,\xi)+p_{m-1}(x,\xi)+\cdots,
$$
where $p_{m-k}$ is homogeneous in $\xi$ of degree $m-k$ and smooth in $(x,\xi)\in\xT^n\times\xR^n\setminus\{0\}$. 

Let $\chi:\xT^n\to\xT^n$ be a smooth diffeomorphism, such that $\chi(x)=x+\theta(x)$ with $\sup_x|\partial_x\theta(x)|<1$. Then there is a symbol $q(x,\xi)\in\mathscr{S}^m_{1,0}$ such that 
$$
(\Op(p)f)\circ\chi=\Op(q)(f\circ\chi),
\quad f\in \bigcup_{s\in\xR}H^s.
$$
Moreover $q(x,\xi)$ admits the following asymptotic expansion: 
\begin{equation}\label{Smoothq(x,xi)}
q(x,\xi)
\sim\sum_{|\alpha|=0}^{\infty}
\frac{1}{i^{|\alpha|}|\alpha|!}\partial_y^\alpha\partial_\xi^\alpha
\left(p\big(\chi(x),J(x,y)^{-1}\xi\big)
\left|\frac{\det\chi'(y)}{\det J(x,y)}\right|\right)
\Bigg|_{y=x},
\end{equation}
where 
$$
J(x,y)=\int_0^1\chi'\big(tx+(1-t)y\big)^{\T}dt.
$$
\end{theorem}

The \emph{para}change-of-variable formula for \emph{para}differential operators is stated below. It is clear from the statement that this formula resembles Theorem \ref{Change(1,0)} \emph{algebraically}. Naturally, 
it suffices to state the theorem for $p\in\dot{\Gamma}^m_\sigma$, a single homogeneous symbol.

\begin{theorem}[Conjugation with paracomposition]\label{4PCConj}
Let $m\in\xR$, let $r>1$ and $\sigma>1$ be non-integer. Assume that $\theta\in C^r$ is as in Definition \ref{DefParacomposition}, and set again $\chi(x)=x+\theta(x)$. Let $p(x,\xi)\in\dot{\Gamma}^m_\sigma$ be a symbol of $C^\sigma$ regularity in $x$, smooth and homogeneous of degree $m$ for $\xi\in\mathbb{R}^n\setminus\{0\}$. 
Fix any $r'<r$, and define $\mu=\min(\sigma,r'-1)$.
 
Then there is a symbol $q\in \Sigma^m_{\min(\sigma,r-1)}$, whose principal part is given by 
$$
q_m(x,\xi)=p\left( \chi(x),
\chi'(x)^{-\T}\xi\right),
$$
where recall that $-\T$ stands for taking the inverse of transpose, 
such that
$$
\chi^\star T_p=T_{q}\chi^\star+\R_{\Con}(p,\chi).
$$
The symbol $q(x,\xi)$ is explicitly computed as
\begin{equation}\label{q(x,xi)}
q(x,\xi)
=\sum_{|\alpha|=0}^{[\mu]}
\frac{1}{i^{|\alpha|}|\alpha|!}\partial_y^\alpha\partial_\xi^\alpha
\left(p\big(\chi(x),J(x,y)^{-1}\xi\big)
\left|\frac{\det\chi'(y)}{\det J(x,y)}\right|\right)
\Bigg|_{y=x},
\end{equation}
with 
$$
J(x,y)=\int_0^1\chi'\big(tx+(1-t)y\big)^{\T}\dt.
$$ 
The remainder is regularizing in $H^s$ for any $s\in\mathbb{R}$:
$$
\left\|\R_{\Con}(p,\chi);\mathcal{L}(H^s,H^{s+(\mu-m-1)})\right\|
\lesssim_{s,r,r',\sigma}
|\theta|_{C^r}\mathcal{M}_\sigma^m(p).
$$

Moreover, if $\chi_i=\Id+\theta_i$, $i=1,2$ satisfy the requirement of Definition \ref{DefParacomposition}, we have the following Lipschitz dependence of $\R_\Con(p,\chi)$: with $r'<r$, $\mu_1=\min(\sigma-1,r'-1)$, there holds
$$
\big\|\R_\Con(p,\chi_1)-\R_\Con(p,\chi_2);\mathcal{L}(H^s,H^{s+(\mu_1-m-1)})\big\|
\lesssim_{s,r,r',\sigma}\mathcal{M}^m_\sigma(p)|\theta_1-\theta_2|_{C^r}.
$$
\end{theorem}

The proof of this theorem will take up the remainder of this subsection. The underlying idea is similar to that of Theorem \ref{ParacompoPrd}, though the technical details are more intricate.

We continue to define the homotopy $\Theta(\tau,x)$ and the deformation vector field $X(\tau,x)$ as in (\ref{homotopy}) and (\ref{N10}):
$$
\Theta(\tau,x)=\tau e^{-(1-\tau)\langle D_x\rangle}\theta(x),
\quad
X(\tau,x)=-\big(I_n+\partial_x\Theta(\tau,x)\big)^{-1}\partial_\tau\Theta(\tau,x).
$$
Furthermore, given $s\in\mathbb{R}$ and $f\in H^s$, we still set $w\in C^0([0,1];H^s)\cap C^1([0,1];H^{s-1})$ as the unique solution to
$$
\partial_\tau w+T_{X}\cdot\nabla_x w=0,\quad w\Big|_{\tau=0}=f,
$$
so that $w(1,\cdot)=\chi^\star f$, as in Definition \ref{DefParacomposition}. We will determine a deformation $\wp(\tau;x,\xi)$ of $p(x,\xi)$, such that 
$$
\wp(\tau;x,\xi)=\wp_m(\tau;x,\xi)+\cdots+\wp_{m-[\mu]}(\tau;x,\xi)\in C^0_\tau\Sigma^m_{\mu},
$$
where each $\wp_{m-k}$ is homogeneous in $\xi$ of degree $m-k$, and $T_{\wp} w\big|_{\tau=1}-\chi^\star T_pf$ is an acceptable error. 
Then we will get the desired symbol $q$ by setting $q=\wp\arrowvert_{\tau=1}$. The symbols in formula (\ref{q(x,xi)}) are in fact of slightly higher regularity than $\mu$, but this will follow from the last step of the proof.

\noindent
\textbf{Step 1: Set up.} 

Let us form an evolution equation for $T_{\wp}w$. 
Remembering that $\partial_\tau w+T_{X}\cdot\nabla_x w=0$ by definition of $w$, and using $\Op^\PM(a)$ as another way of writing $T_a$, we compute that
\begin{equation}\label{Tqmw}
\begin{aligned}
(\partial_\tau&+T_X\cdot\nabla_x)(T_{\wp}w)\\
&=T_{\wp}(\partial_\tau +T_X\cdot\nabla_x)w
+T_{\partial_\tau \wp}w+T_X\cdot T_{\nabla_x \wp} w
+[T_X,T_{\wp}]\cdot\nabla_x w\\
&=\Op^\PM(\partial_\tau \wp +X\cdot\nabla_x \wp)w
+(T_X\cdot T_{\nabla_x \wp}-T_{X\cdot\nabla_x \wp}) w
+[T_X,T_{\wp}]\cdot\nabla_x w\\
&=\Op^\PM\left[\partial_\tau \wp +X\cdot\nabla_x \wp-(\nabla_\xi \wp)\cdot\nabla_x \big(\xi \cdot X\big)\right]w\\
&\quad
+(T_X\cdot T_{\nabla_x \wp}-T_{X\cdot\nabla_x \wp}) w
+\left([T_X,T_{\wp}]-\Op^\PM\frac{1}{i}\left\{
X,\wp\right\}\right)\cdot\nabla_x w,
\end{aligned}
\end{equation}
where the third equality follows from the straightforward computation
$$
\begin{aligned}
\left\{
X,\wp\right\}(\tau;x,\xi) \cdot \xi
=-(\nabla_\xi \wp)(\tau;x,\xi) \cdot 
\nabla_x \big(\xi \cdot X(\tau,x)\big).
\end{aligned}
$$
By the symbolic calculus of paradifferential operators (see Theorem~\ref{2theo:sc}), it follows that
the operator norm of $\R_1(p,\chi;\tau):=T_X\cdot T_{\nabla_x \wp}-T_{X\cdot\nabla_x \wp}$ in the right-hand-side of (\ref{Tqmw}) satisfies a \emph{tame} estimate
\be\label{NR0}
\begin{aligned}
\big\|\R_1(p,\chi;\tau);&\,\mathcal{L}(H^s_x, H^{s-m+(\mu-1)}_x)\big\|\\
&\overset{(\ref{esti:quant2})}{\lesssim_{s,r,r',\sigma}}
|X|_{L^\infty_x}
\sum_{k=0}^{[\mu]}\mathcal{M}_{\mu-(k+1)}^{m-k}(\wp_{m-k})
+|X|_{C^{\mu-1}_x}\sum_{k=0}^{[\mu]}\mathcal{M}_{1}^{m-k}(\wp_{m-k})
\end{aligned}
\ee
Thus, as long as $\wp$ has enough regularity with respect to $x$, this operator $\R_1(p,\chi;\tau)$ is regularizing in $x$.

Using again the symbolic calculus of paradifferential operators, we compute
\be\label{N31}
\left([T_X,T_{\wp}]-\Op^\PM\frac{1}{i}\left\{
X,\wp\right\} \right)\cdot\nabla_x
=T_{\gamma_{m-1}}+\cdots+T_{\gamma_{m-([\mu]-1)}}+\R_{2}(p,\chi;\tau),
\ee
where for $k\geq1$,
\begin{equation}\label{gamma_m-k}
\gamma_{m-k}=-\sum_{l=0}^{k-1}\sum_{\substack{\alpha\in\xN^n\\ |\alpha|=k-l+1}}\frac{1}{i^{\la \alpha\ra}\alpha!}\partial_x^\alpha (i\xi \cdot X)\partial_\xi^\alpha \wp_{m-l}
\in L^1_\tau C^{\mu-k}_x,
\end{equation}
and the remainder $\R_2(p,\chi;\tau)$ still satisfies a \emph{tame} estimate as indicated by Theorem \ref{2theo:sc}:
\begin{equation}\label{Rmu}
\begin{aligned}
\big\|\mathcal{R}_{2}(p,\chi;\tau);&\,{\mathcal{L}(H^s_x, H^{s-m+(\mu-1)}_x)}\big\|\\
&\overset{(\ref{esti:quant2})}{\lesssim_s} |X|_{C^{\mu-1}_x}
\sum_{k=0}^{[\mu]}\mathcal{M}_0^{m-k}(\wp_{m-k})
+|X|_{L^\infty_x}
\sum_{k=0}^{[\mu]}\mathcal{M}_{\mu-k}^{m-k}(\wp_{m-k})
\end{aligned}
\end{equation}

We can thus rearrange (\ref{Tqmw}) as
\begin{equation}\label{Tqmw'}
\begin{aligned}
(\partial_\tau&+T_X\cdot\nabla_x)(T_{\wp}w)\\
&=\Op^\PM\left[\partial_\tau \wp_m +X\cdot\nabla_x \wp_m-(\nabla_\xi \wp_m)\cdot\nabla_x \big(\xi \cdot X\big)\right]w\\
&\quad+\sum_{k=1}^{[\mu]}\Op^\PM\left[\partial_\tau \wp_{m-k}+X\cdot\nabla_x \wp_{m-k}-(\nabla_\xi \wp_{m-k})\cdot\nabla_x \big(\xi \cdot X\big)+\gamma_{m-k}\right]w\\
&\quad+\R_1(p,\chi;\tau)w+\R_{2}(p,\chi;\tau)w,
\end{aligned}
\end{equation}
with $\R_1(p,\chi;\tau),\R_{2}(p,\chi;\tau)$ satisfying (\ref{NR0}) and (\ref{Rmu}). This form enables us to solve the symbols $\wp_{m-k}$ inductively.

\noindent
\textbf{Step 2: Cascade of transport equations.} Set $\wp_m=p^{\sharp;\theta}$ as in Lemma \ref{L:X}:
\begin{equation}\label{wp(m)}
\wp_m(\tau;x,\xi)=p\Big( x+\Theta(\tau,x),
\big(I_n+\partial_x\Theta(\tau,x)\big)^{-\T}\xi\Big).
\end{equation}
Obviously $\wp_m\in C^0_\tau C^{\min(\sigma,r-1)}_x$. Moreover, it follows from Proposition~\ref{L:rho} that for any $\tau\in[0,1)$, the symbol has additional regularity $\wp_m(\tau;\cdot,\cdot)\in \dot{\Gamma}^m_{\min(\sigma,r')}$, and satisfies the estimate
\be\label{Reg(m)}
\int_0^1 \mathcal{M}^m_{\min(\sigma,r')}(\wp_m)\dtau
\lesssim
\mathcal{M}^m_\sigma(p).
\ee
Recall that Lemma~\ref{L:X} gives
$$
\partial_\tau \wp_m +X\cdot\nabla_x \wp_m-(\nabla_\xi \wp_m)\cdot\nabla_x \big(\xi \cdot X\big)=0.
$$
Consequently, the first term in the right-hand-side of (\ref{Tqmw'}) is cancelled.

We then determine the sub-principal symbols $\wp_{m-k}$ by solving a cascade of transport equations inductively. Given $k\ge 1$, if $\wp_{m-(k-1)}$ is already defined, then define $\wp_{m-k}$ to be the solution of the classical linear transport equation
\begin{equation}\label{wp(m-k)}
\left\{
\begin{aligned}
&\partial_\tau \wp_{m-k}
+X\cdot\nabla_x \wp_{m-k}
-\nabla_\xi \wp_{m-k}\cdot 
 \nabla_x(\xi \cdot X)+\gamma_{m-k}=0,\\
&\wp_{m-k}\Big|_{\tau=0}=0.
\end{aligned}
\right.
\end{equation}
This cascade of equations is truly solvable inductively, since by (\ref{gamma_m-k}), the definition of $\gamma_{m-k}$ only involves $\wp_m,\ldots,\wp_{m-(k-1)}$. In particular, 
$$
\gamma_{m-1}=-\frac{1}{2! i^2}
\sum_{\alpha:|\alpha|=2}\partial_x^\alpha (i\xi \cdot X) \partial_\xi^\alpha \wp_m.
$$
This choice of $\wp_{m-k}$'s exactly cancels all the homogeneous symbols in the sum on the right-hand-side of (\ref{Tqmw'}). What remains is to investigate the regularity of the $\wp_{m-k}$'s inductively. 

We then aim to show that 
\be\label{Reg(m-k)}
\sup_{\tau\in [0,1]}
\mathcal{M}^{m-k}_{\mu-k}(\wp_{m-k})
\lesssim_{\mu}
|\theta|_{C^r}\mathcal{M}^m_\sigma(p)
\ee
for $0\leq k\leq [\mu]$. This then justifies the choice of all the symbols $\wp_{m-k}$.

Suppose (\ref{Reg(m-k)}) is true for $k\leq k'-1$, where $k'\geq1$ is given. We aim to prove (\ref{Reg(m-k)}) for $k=k'$. By (\ref{gamma_m-k}), the expression of $\gamma_{m-k'}$  involves $k'+1$'th order $x$-derivative for $X$ and symbols $\wp_m,\cdots,\wp_{m-(k'-1)}$, so by combining the classical estimates for products and and compositions in H\"older spaces (see Proposition \ref{2esti:C^scomp}) with Lemma \ref{L:rho} and the fact that $\mu=\min(\sigma,r'-1)$, we obtain 
\begin{equation}\label{Int(gamma(m-k))}
\begin{aligned}
\int_0^1\mathcal{M}^{m-k'}_{\mu-k'}(\gamma_{m-k'})\dtau
&\lesssim \sum_{l=0}^{k'-1}\sup_{\tau\in[0,1]}\mathcal{M}^{m-l}_{\mu-l}(\wp_{m-l})
\int_0^1|X|_{C^{r'}_x}\dtau \\
&\lesssim |\theta|_{C^r}\mathcal{M}^m_\sigma(p).
\end{aligned}
\end{equation}
From the proof of Proposition~\ref{3RegTrans} , along with Remark \ref{HamiltonxiX}, it is possible to explicitly write down the expression for $\wp_{m-k'}$: with the flow maps
$$
\Phi(0,\tau;x,\xi)=\left(\begin{matrix}
x+\Theta(\tau,x) \\
\big(I_n+\partial_x\Theta(\tau,x)\big)^{-\T}\xi
\end{matrix}\right),
\quad
\Phi(\tau,t;x,\xi):=\Phi^\iota\big(0,t;\Phi(0,\tau;x,\xi)\big),
$$
there holds
\begin{equation}\label{wp(m-k)explicit}
\wp_{m-k'}(\tau;x,\xi)
=\int_0^\tau\gamma_{m-k'}\big(\tau,\Phi(\tau',\tau;x,\xi)\big)\dtau'.
\end{equation}
Noting that the flow map acts as linear transformation for $\xi$, we conclude inductively that the symbol $\wp_{m-k'}$ is still homogeneous of degree $m-k'$ and smooth with respect to $\xi\neq0$. Furthermore, Lemma \ref{L:rho} shows that the flow map $\Phi$ is of $C^{r'}$ regularity in $x$ for $r'<r$, so (\ref{Int(gamma(m-k))}) (together with H\"{o}lder composition estimate) implies the desired result (\ref{Reg(m-k)}).

Therefore, (\ref{Tqmw'}) is converted to
$$
(\partial_\tau +T_X\cdot\nabla_x)(T_{\wp}w)
=\R_1(p,\chi;\tau)w+\R_2(p,\chi;\tau)w,
$$
where, by (\ref{NR0}) and (\ref{Rmu}) together with (\ref{Reg(m-k)}), $\R_1(p,\chi;\tau)+\R_{2}(p,\chi;\tau)$ is of order $m-(\mu-1)$ uniformly in $\tau$:
$$
\sup_{\tau\in[0,1]}\big\|\R_1(p,\chi;\tau)+\R_{2}(p,\chi;\tau);\,{\mathcal{L}(H^s_x, H^{s-m+(\mu-1)}_x)}\big\|
\les_{s,\mu} |\theta|_{C^r}\mathcal{M}^m_r(p).
$$

\noindent
\textbf{Step 3: Smoothing Remainder.} We can finally define the symbol $q$ as
$$
q(x,\xi)
:=\wp(\tau;x,\xi)\Big|_{\tau=1}
=\sum_{k=0}^{[\mu]}\wp_{m-k}(\tau;x,\xi)\Big|_{\tau=1}.
$$
To see that it is the symbol in Theorem \ref{4PCConj}, set $v\in C^0([0,1];H^{s-m} )\cap C^1([0,1];H^{s-m-1} )$ as the unique solution to
\be\label{n10bis}
\partial_\tau v+T_X\cdot\nabla_x v=0,\quad v\Big|_{\tau=0}=T_pf,
\ee
so that $v(1,\cdot)=\chi^\star T_pf$. Then the difference $T_\wp w-v$ satisfies
\begin{equation}\label{Diff}
\partial_\tau (T_\wp w-v)+T_X\cdot\nabla_x (T_\wp w-v)
=\R_1(p,\chi;\tau)w+\R_{2}(p,\chi;\tau)w,
\quad T_\wp w-v\Big|_{\tau=0}=0.
\end{equation}
By Proposition \ref{ParaTranspEq}, recalling that $w$ solves $(\partial_\tau-T_X\cdot\nabla_x)w=0$ with initial value $f$, we estimate
$$
\begin{aligned}
\sup_{\tau\in[0,1]}\|T_\wp w-v;H^{s-m+(\mu-1)}_x\|
&\lesssim_{s,\mu} |\theta|_{C^r}\mathcal{M}^m_\sigma(p)\|w\|_{C^0_\tau H^{s}_x}\\
&\lesssim_{s,\mu}
|\theta|_{C^r}\mathcal{M}^m_\sigma(p)\|f\|_{H^s}
\end{aligned}
$$
But the value of $T_\wp w-v$ at $\tau=1$ is nothing but
$$
T_q\chi^\star f-\chi^\star T_pf=\R_\Con(p,\chi)f.
$$
This then justifies that $\R_\Con(p,\chi)=T_q\chi^\star-\chi^\star T_p$ gains back $\mu-m-1$ derivatives. 

To obtain Lipschitz continuity of $\R_\Con(p,\chi)$ with respect to $\chi$, we just consider the difference of (\ref{Diff}) for two different $\chi_1,\chi_2$. Remembering Corollary \ref{ParaTranspEqCor1}, we find that the difference loses one derivative in $x$. Noticing that $\R_1(p,\chi;\tau),\R_2(p,\chi;\tau)$ are both linear in the deformation vector field $X$, we obtain the desired estimate for $\R_\Con(p,\chi_1)-\R_\Con(p,\chi_2)$. The proof is very similar to the Lipschitz part of Theorem \ref{ParacompoPrd}.

\noindent
\textbf{Step 4: The algebra part.}

It remains to prove that the symbols $q_{m-k}$, as defined, indeed take the 
form indicated in formula (\ref{q(x,xi)}). Theoretically, this 
can be inferred from the explicit formulas (\ref{gamma_m-k}) 
and (\ref{wp(m-k)explicit}), but there is an algebraic 
justification that avoids the (obviously!) cumbersome computations. 
The general idea is straightforward: the symbolic calculus of paradifferential 
operators preserves all the algebraic structures of \emph{pseudo-differential operators of the usual type (1,0)}.

In fact, the classical proof of Theorem \ref{Change(1,0)} (see e.g. Theorem 2.1.2 of \cite{Hormander1971}) makes use of asymptotic expansion of oscillatory integrals to decide the lower order terms in the polyhomogeneous asymptotic expansion. On the other hand, we can recast the proof by considering the diffeomorphism as the time 1 map of a flow. Let us still define the homotopy $\Theta$ and deformation vector field $X$ as in Definition \ref{Deformation}:
$$
\Theta(\tau,x)=\tau e^{-(1-\tau)\langle D_x\rangle}\theta(x),
\quad
X(\tau,x)=-\big(I_n+\partial_x\Theta(\tau,x)\big)^{-1}\partial_\tau\Theta(\tau,x),
$$
Let us then set $p(x,\xi)\in\mathscr{S}^m_{1,0}$ to be a \emph{smooth} symbol and homogeneous of degree $m$ in $\xi$, and still consider the solution $u=(\Op(p)f)^{\sharp;\theta}$ of the \emph{classical} transport equation
$$
\partial_\tau u+X\cdot\nabla_xu=0,\quad u\Big|_{\tau=0}=\Op(p)f.
$$
Then we still have $u\Big|_{\tau=1}=(\Op(p)f)\circ\chi$. 

On the other hand, if we define $\wp_m$ as in (\ref{wp(m)}) and $\wp_{m-k}$ as in (\ref{wp(m-k)}) \emph{for any $k>0$}, we obtain the following ``classical" counterpart of (\ref{Tqmw'}):
$$
(\partial_\tau+X\cdot\nabla_x)\Op(\wp)f^{\sharp;\theta}
\simeq\Op\mathscr{S}^{-\infty}f^{\sharp;\theta},
\quad
\Op(\wp)f^{\sharp;\theta}\Big|_{\tau=0}=\Op(p)f.
$$
with $\wp=\sum_{k=0}^\infty\wp_{m-k}$. Subtracting the equation for $u$ and $\Op(\wp)f^{\sharp;\theta}$, evaluating at $\tau=1$, we obtain 
$$
(\Op(p)f)\circ\chi=\Op\big(\wp\big|_{\tau=1}\big)(f\circ\chi)
\mod\Op\mathscr{S}^{-\infty}f.
$$
Noticing that $\wp$ is necessarily polyhomogeneous of degree $m$ in $\xi$ by the transport equation (\ref{wp(m-k)}), we find $\wp\big|_{\tau=1}$ must take the form indicated in (\ref{Smoothq(x,xi)}). Therefore we have proved the solution of (\ref{wp(m-k)}) is of the desired form \emph{for smooth symbols}. It is then a simple matter of approximation to pass to rough symbols, as (\ref{wp(m-k)}) is merely a classical linear transport equation.

The proof of Theorem \ref{4PCConj} is complete.

\subsection{Addendum: Paradifferential Egorov Analysis}
It is sometimes necessary to develop a symbolic calculus of conjugation by flow maps beyond change of variables. For usual pseudo-differential operators, the classical \emph{Egorov theorem} asserts that the symbol of a pesudo-differential operator propagates along the bi-characteristic flow of the phase function. 

It turns out that this classical Egorov analysis has a paradifferential generalization. 

We fix some $r>1$, some symbol $a(x,\xi)\in\Gamma^{\delta}_r$, where $0\leq \delta<1$, and suppose that the principal symbol of $a$ is \emph{real}. That is, $a-\bar a\in\Gamma^{\delta-1}_r$. By Theorem \ref{2theo:sc2}, the paradifferential operator $T_a$ is almost self-adjoint: the adjoint $(T_a)^*$ has symbol $a^{\mathfrak{c};r}:=\bar{a}+\cdots$, so
$$
\big\|T_a-(T_a)^*;\mathcal{L}(H^s;H^{s+(1-\delta)})\big\|
\lesssim_s\mathcal{M}_r^\delta(a),
\quad \forall s\in\xR.
$$
Consequently, standard operator semi-group theory ensures that $iT_a$ generates a strongly continuous operator group $\Phi(\tau):=\exp(i\tau T_a)$ on any $H^s$. 

Now suppose $p\in\Gamma^m_r$. We immediately see that the linear operator $P(\tau):=\Phi(\tau)T_p\Phi(\tau)^{-1}$ solves the Heisenberg equation
\begin{equation}\label{Hei}
\partial_\tau P=i[T_a,P]=:i\cdot\mathbf{Ad}_{T_a}P,
\quad
P(0)=T_p.
\end{equation}
In order to study the operator $P(1)$, we can just formulate the Lie series expansion for (\ref{Hei}). In fact, reiterating the Duhamel formula for this differential equation, we find
\begin{equation}\label{LieExp}
\begin{aligned}
\Phi(1)T_p\Phi(1)^{-1}
&=T_p+\sum_{k=1}^N\frac{i^k}{k!}\mathbf{Ad}_{T_a}^kT_p\\
&\quad+\frac{i^{N+1}}{N!}\int_0^1(1-\tau)^{N+1}
\Phi(\tau)\cdot\mathbf{Ad}_{T_a}^{N+1}T_p\cdot\Phi(\tau)^{-1}\dtau.
\end{aligned}
\end{equation}
By the symbolic calculus formula, we find for example the principal symbol of the sum in (\ref{LieExp}) is $\{a,p\}$, the Poisson bracket, which is of order $m+\delta-1<m$. Every time the adjoint $\mathbf{Ad}_{T_a}$ is applied, the order will be decreased by $1-\delta$. Thus the Lie expansion just given provides a delicate description of the behaviour of $T_p$ under the conjugation of $\Phi(\tau)$. An important advantage is that the symbolic calculus for paradifferential operators provides quantitative, \emph{tame} control on the remainder of (\ref{LieExp}).

We notice that the symbol of the right-hand-side of (\ref{LieExp}) can be formally written as
$$
p+\{a,p\}+\frac{1}{2!}\{a,\{a,p\}\}+\cdots,
$$
which can be thought as ``propagation of $p(x,\xi)$ along the bi-characteristic flow of $a(x,\xi)$". To a certain extent, the symbolic calculus of paracomposition operators can be viewed as the ``$\delta=1$ limit" of the above Egorov analysis. In fact, as illustrated in Theorem \ref{4PCConj}, the principal symbol of a paradifferential operator $T_p$ conjugated by a paracomposition $\chi^\star$ is exactly the propagation of $p(x,\xi)$ along the ``time-dependent bi-characteristic flow" generated by the deformation vector field corresponding to $\chi$, as mentioned in Remark \ref{HamiltonxiX}.

\newpage
\part{Paradifferential Reducibility}\label{Part2}

In Section \ref{sec:5}-\ref{sec:6}, we discuss linear reducibility results for matrix differential operator and nearly parallel vector field on torus. The informal theorems \ref{MatRedInformal} and \ref{VectRedInformal} are stated as Theorem \ref{MatRed} and \ref{6VectRed}. Being crucial intermediate results needed for the paradifferential reducibility of (\ref{1NLPHyp}), they can also be considered as toolboxes per se. 

Following the same general approach, the proof for matrix operator is technically simpler, so we address it first in Section \ref{sec:5}. The reducibility for vector field will be discussed in Section \ref{sec:6}.

\section{Paradifferential Approach to Reducibility of Matrix Operator}\label{sec:5}
Throughout Section \ref{sec:5}-\ref{sec:6}, we denote by $\Dio[\gamma,\tau]\subset\xR^n$ the set of Diophantine frequency vectors with parameters $\gamma,\tau$; that is, 
\begin{equation}\label{diopc}
\Dio[\gamma,\tau]=\left\{\omega\in\xR^n:\,|\omega\cdot\xi|\geq\frac{\gamma}{|\xi|^\tau},
\quad \forall\xi\in\xZ^n\setminus\{0\}\right\}.
\end{equation}
As a preliminary remark, we shall see, given some parameters $\gamma,\tau$, one can define for \emph{any} $\omega\in\xR^n$, a Fourier multiplier $\bm{L}_{\gamma,\tau}(\omega)$ which coincides with $(\omega\cdot\partial)^{-1}$ when $\omega\in\Dio[\gamma,\tau]$.

Hereafter, given an expression $F$ depending on some variable $X$, we denote by $\Delta_{12}F$ the difference of the quantity under concern when evaluated at $X_1$ and $X_2$. For instance, $\Delta_{12}\omega=\omega_1-\omega_2$, and $\Delta_{12}\bm{L}_{\gamma,\tau}(\omega)=\bm{L}_{\gamma,\tau}(\omega_1)-\bm{L}_{\gamma,\tau}(\omega_2)$. 

\begin{lemma}\label{Dio_Ext}
Fix $\gamma,\tau>0$. There is a Fourier multiplier
$$
\bm{L}_{\gamma,\tau}(\omega)u:=\sum_{\xi\in\xZ^n}L_{\gamma,\tau}^\xi(\omega)\hat u(\xi)e^{i\xi\cdot x},
$$
depending on $\omega\in\xR^n$, with the following properties:
\begin{itemize}
    \item For every $\omega\in\xR^n$, there holds
    $$
    \|\bm{L}_{\gamma,\tau}(\omega)u\|_{H^s}\leq \gamma^{-1}\|u\|_{H^{s+\tau}},
    $$
    and $\bm{L}_{\gamma,\tau}(\omega)$ reverts parity of functions.
    \item The multiplier is Lipschitz in $\omega$ with loss of derivatives: 
    $$
    \big\|\big(\Delta_{12}\bm{L}_{\gamma,\tau}(\omega)\big)u\big\|_{H^s}\leq \gamma^{-2}|\Delta_{12}\omega|\cdot\|u\|_{H^{s+2\tau+1}}.
    $$
    \item When $\omega\in\Dio[\gamma,\tau]$, there holds
    $$
    \bm{L}_{\gamma,\tau}(\omega)=(\omega\cdot\partial)^{-1}.
    $$
\end{itemize}
\end{lemma}
\begin{proof}
Let us set $L_0=0$ and consider an arbitrary decomposition of the grid set $\xZ^n\setminus\{0\}$ into the union of disjoint subsets $H_1\cup H_2$, such that $H_1$ is the reflection of $H_2$ with respect to 0. For $\xi\in H_1$, we consider the function $L^\xi\colon \Dio[\gamma,\tau]\to i\xR$ defined by
$$
L^\xi(\omega):=\frac{1}{i\omega\cdot\xi},\quad \omega\in\Dio[\gamma,\tau].
$$
Then by definition of Diophantine frequency, we have $|L^\xi(\omega)|\leq \gamma^{-1}|\xi|^\tau$. Furthermore, 
$$
|\Delta_{12}L^\xi(\omega)|
=\left|\frac{\Delta_{12}\omega\cdot\xi}{(\omega_1\cdot\xi)(\omega_2\cdot\xi)}\right|
\leq \gamma^{-2}|\xi|^{2\tau+1}|\Delta_{12}\omega|,
\quad \omega_1,\omega_2\in\Dio[\gamma,\tau].
$$

Now we recall a classical extension theorem for Lipschitz functions -- see for example, Theorem 1.2 of Chapter 2, \cite{Simon2017}. If $f\colon A\to \xR$ is a Lipschitz function defined on a non-empty subset $A$ of a metric space $(X,d)$, then there is a Lipschitz extension $F\colon X\to \xR$ such that $F\arrowvert_A=f$, having the same supremum norm and Lipschtiz constant. The extension is indeed given explicitly as
$$
F(x)=\inf_{a\in A}\big(f(a)+(\mathrm{Lip}f)\cdot d(a,x)\big)\vee\sup_A|f|.
$$

We thus extend the function $L^\xi$ to a Lipschitz function $L^\xi_{\gamma,\tau}$ on $\xR^n$ without changing the supremum norm and the Lipschitz constant. This defines $L^\xi_{\gamma,\tau}\colon\xR^n\to i
\xR$ for $\xi\in H_1$. Let us then define $L^\xi_{\gamma,\tau}(\omega)=-L^{-\xi}_{\gamma,\tau}(\omega)$ for $\xi\in H_2$. Since the mapping $(\xi,\omega)\mapsto L^\xi_{\gamma,\tau}(\omega)$ before extension is odd in $\xi$, we find that $L^\xi_{\gamma,\tau}(\omega)$ is indeed an extension of $1/(i\omega\cdot\xi)$ for $\xi\in H_2$, and the extended function is still odd in $\xi$. Therefore, the resulting Fourier multiplier reverts parity, and satisfies all the operator bounds.
\end{proof}

\subsection{The Reducibility Theorem}
The theorem for reducibility of the matrix linear differential operator $\D{\omega}-A$ is stated as follows:

\begin{theorem}\label{MatRed}
Fix a natural number $n$. Given $\gamma\in(0,1]$, $\tau>n-1$, $\delta>0$, let $s_0=2\tau+1+n/2+\delta$. There is a constant $\varepsilon_0=\varepsilon_0(s_0)$ and a Lipschitz mapping
$$
U=U(\omega,A),\quad
\xR^n\times \big(H_\odd^{s_0+2\tau+1}\otimes\mathbf{M}_{N}(\xC)\big)
\to H_\even^{s_0}\otimes\mathbf{M}_{N}(\xC),
$$
with the following properties. 

\begin{enumerate}

    \item\label{Mat1} If $A\in H^{s+\tau}$ with $s\geq s_0$, then in fact $U\in H^s$:
    $$
    \|U\|_{H^s}\leq C_{s}\gamma^{-1}\|A\|_{H^{s+\tau}}.
    $$

    \item\label{Mat2} The Lipschitz dependence of $U(\omega,A)$ is as follows: if $\omega_1,\omega_2\in\xR^n$, $A_1,A_2\in H^{s+2\tau+1}$ with $s\geq s_0$ and $\|A_{1,2}\|_{H^{s_0+2\tau+1}}\leq\varepsilon_0\gamma^2$, then with $\Delta_{12}U:=U(\omega_1,A_1)-U(\omega_2,A_2)$,
    $$
    \begin{aligned}
    \|\Delta_{12}U\|_{H^{s}}
    &\leq C_s\gamma^{-1}\|\Delta_{12}A\|_{H^{s+\tau}}\\
    &\quad +C_s\gamma^{-2}\big(\|A_1\|_{H^{s+2\tau+1}}+\|A_2\|_{H^{s+2\tau+1}}\big)\big(\|\Delta_{12}A\|_{H^{s_0+\tau}}+|\Delta_{12}\omega|\big).
    \end{aligned}
    $$

    \item If $\omega\in\Dio[\gamma,\tau]$ and $\|A\|_{H^{s_0+2\tau+1}}<\varepsilon_0\gamma^2$, then $U(\omega,A)\in H^{s_0}$ solves the matrix homological equation
    $$
    \D{\omega} U-A\cdot(I_N+U)=0.
    $$
    Consequently, the matrix differential operator $\D{\omega} -A$ conjugates to $\D{\omega}$:
    $$
    (I_N+U)^{-1}\cdot(\D{\omega}-A)\circ(I_N+U)=\D{\omega}.
    $$
\end{enumerate} 
\end{theorem}

\subsection{Existence and Additional Regularity}\label{ProofMatRed1}
We first discuss existence of $U$ and its regularity.

\noindent
\textbf{Step 1: Paralinearization and Neumann series argument.}

Given $A\in H^{s_0+\tau}_\odd$, We define a mapping
$$
\mathscr{E}(A,U):=\D{\omega} U-A\cdot(I_N+U)
$$
for odd $A$ and even $U$, and search for a zero $U\in H^{s_0}_\even$ of it. 
Notice that, if $\|U\|_{H^{s_0}}$ is small enough, then it follows from the Sobolev embedding $H^{s_0}\subset L^\infty$ 
that it can be assumed without loss of generality that $|U|_{L^\infty}\ll 1$, 
making $I_N+U$ always an invertible matrix, and $T_{I_N+U}^{\gau}$ always an invertible operator. Consequently, from the definition we solve
\begin{equation}\label{subA}
A=\big[\D{\omega} U-\mathscr{E}(A,U)\big](I_N+U)^{-1}.
\end{equation}
This simple identity actually has deeper root in the algebraic structure of conjugacy problems, as was first observed by Zehnder (see Section 5 of\cite{Zehnder1975}). Further explanation can be found in Section 2 of \cite{AS2023}.

The only nonlinearity appearing in $\mathscr{E}(A,U)$ is the product $AU$, so we can directly cast paraproduct decomposition and obtain
$$
\mathscr{E}(A,U)=\D{\omega} U-A-T^{\gau}_AU-T_{U}^\droi A-R_\PM(A,U),
$$
where recall that $T^{\gau}$ and $T^{\droi}$ are defined by~\eqref{LRPM} and $R_\PM$ is the remainder in paraproduct decomposition as in Definition \ref{R_PM}. Recall that the remainder satisfies $R_\PM(A,U)\in H^{2s_0+\tau-n/2}$, as follows from Theorem \ref{PMReg}.

Replacing $A$ by its expression given by~\eqref{subA} into the term $T_A^{\gau}U$, we obtain the paralinearization formula for $\mathscr{E}(A,U)$:
\begin{equation}\label{MatParaLinE}
\mathscr{E}(A,U)
=\D{\omega} U-T^{\gau}_{\D{\omega} U\cdot(I_N+U)^{-1}}U-T_{I_N+U}^\droi A-R_\PM(A,U)
+T^\gau_{\mathscr{E}(A,U)(I_N+U)^{-1}}U.
\end{equation}
Now, the proof is based on two distinct observations. 
The first and most important one is the following claim: 
one can rewrite 
$$
\D{\omega} U-T^{\gau}_{\D{\omega} U\cdot(I_N+U)^{-1}}U
$$
into the form 
\begin{equation}\label{np1}
T^{\gau}_{I_N+U}\D{\omega} T^{\gau}_{(I_N+U)^{-1}}U-\R(U)
\end{equation}
for some admissible remainder term $\R(U)$. 
To see this, we use the Leibniz rule 
\begin{align*}
\D{\omega} T_U^{\gau}V&=T_{\D{\omega} U}^{\gau}V+T_U^{\gau}\D{\omega} V,\\
\D{\omega}(I_N+U)^{-1}&=-(I_N+U)^{-1}\D{\omega} U\cdot(I_N+U)^{-1},
\end{align*}
so that one gets \eqref{np1}, where the the paradifferential remainder is produced in view of the symbolic calculus theorem \ref{2theo:sc} (Remark \ref{2MatSymbol}), being
\begin{equation}\label{MatR(U)}
\begin{aligned}
\R(U)&=\left(T^{\gau}_{I_N+U}T^{\gau}_{(I_N+U)^{-1}}-I_n\right)\D{\omega} U\\
&\quad
-\left(T^{\gau}_{I_N+U}T^{\gau}_{(I_N+U)^{-1}\D{\omega} U\cdot(I_N+U)^{-1}}-T^{\gau}_{\D{\omega} U\cdot(I_N+U)^{-1}}\right)U
\in H^{2s_0-n/2-1}.
\end{aligned}
\end{equation}
We thus end up with
\begin{equation}\label{MatParaLinE'}
\begin{aligned}
\mathscr{E}(A,U)
&=T^{\gau}_{I_N+U}\D{\omega} T^{\gau}_{(I_N+U)^{-1}}U
-T_{I_N+U}^\droi A
-R_\PM(A,U)-\R(U)\\
&\quad+T^\gau_{\mathscr{E}(A,U)(I_N+U)^{-1}}U.
\end{aligned}
\end{equation}

The second key idea is to seek for $U\in H^{s_0}_\even$ that solves the \emph{parahomological equation}\footnote{We say that it is a \emph{parahomological equation} because this is a quadratic perturbation of the linear homological equation $\D{\omega} U=A$ by paradifferential operators.}:
\begin{equation}\label{MatParaHom}
T^{\gau}_{I_N+U}\D{\omega} T^{\gau}_{(I_N+U)^{-1}}U
-T_{I_N+U}^\droi A-R_\PM(A,U)-\R(U)=0,
\end{equation}
which cancels all terms in the right-hand-side of (\ref{MatParaLinE}), except the one linear in $\mathscr{E}(A,U)$. 

The trick here is that solving (\ref{MatParaHom}) will in fact yield a true solution of the exact equation. More precisely, we claim that there exists some constant $\rho_0$ depending only on $s_0$, such that if $U\in H^{s_0}_\even$ solves (\ref{MatParaHom}) with $\|U\|_{H^{s_0}}\leq \rho_0$, then \emph{necessarily} $\mathscr{E}(A,U)=0$, i.e. $U$ is the solution of the reducibility problem. 
To prove this claim, notice that, if $U$ solves the parahomological equation (\ref{MatParaHom}), then 
$$
\mathscr{E}(A,U)
=T^\gau_{\mathscr{E}(A,U)(I_N+U)^{-1}}U.
$$
By Theorem \ref{2theo:sc}, if $\|U\|_{H^{s_0}}\ll_{s_0} 1$, then using Sobolev embedding $H^{s_0}\subset L^\infty$,
$$
|\mathscr{E}(A,U)|_{L^\infty}
\leq C_{s_0}|\mathscr{E}(A,U)(I_N+U)^{-1}|_{L^\infty}\|U\|_{H^{s_0}}
\leq \frac{1}{2}|\mathscr{E}(A,U)|_{L^\infty}.
$$
This forces $\mathscr{E}(A,U)=0$, which proves that it suffices to find a sufficiently small solution to the parahomological equation (\ref{MatParaHom}).

\noindent
\textbf{Step 2: Solve parahomological equation.}

Let us define a mapping $\mathscr{G}$ for $(\omega,A,U)\in\xR^n\times H^{s_0+2\tau+1}\times H^{s_0}$:
\begin{equation}\label{MatParaHom'}
\mathscr{G}(\omega,A;U)
:=(T^{\gau}_{(I_N+U)^{-1}})^{-1}\bm{L}_{\gamma,\tau}(\omega)(T^{\gau}_{I_N+U})^{-1}\left(T_{I_N+U}^\droi A+R_\PM(A,U)+\R(U)\right),
\end{equation}
where the Fourier multiplier $\bm{L}_{\gamma,\tau}(\omega)$ is as in Lemma \ref{Dio_Ext}. 
We claim that if $\omega\in\Dio[\gamma,\tau]$, then the parahomological equation (\ref{MatParaHom}) is implied by the fixed point equation
\begin{equation}\label{Mat_Fix}
U=\mathscr{G}(\omega,A;U).
\end{equation}
In fact, recalling that smoothing operators preserve parity, we find for even $U$ and odd $A$, the matrix-valued function
$$
(T^{\gau}_{I_N+U})^{-1}\left(T_{I_N+U}^\droi A+R_\PM(A,U)+\R(U)\right)
$$
in (\ref{MatParaHom'}) must still be odd, hence has zero average. By Lemma \ref{Dio_Ext}, the outcome of applying the Fourier multiplier $\bm{L}_{\gamma,\tau}(\omega)$ is even. Thus $\mathscr{G}$ is a well-defined mapping taking an even function $U$ to an even function. When $\omega\in\Dio[\gamma,\tau]$, we know that $\bm{L}_{\gamma,\tau}(\omega)=\D{\omega}^{-1}$ by Lemma \ref{Dio_Ext}. So if $U$ is a fixed point of $\mathscr{G}(\omega,A;U)$ and $\omega\in\Dio[\gamma,\tau]$, then $U$ solves the parahomological equation (\ref{MatParaHom}).

It thus suffices to study the fixed point equation (\ref{Mat_Fix}); the mapping $U(\omega,A)$ will be the fixed point of $\mathscr{G}(\omega,A;U)$.

Let us then show that provided $\|A\|_{H^{s_0+\tau}}$ is suitably small, the mapping $\mathscr{G}$ in (\ref{MatParaHom'}) shall define a contraction on some closed ball $\bar{\mathcal{B}}^{s_0}_\rho(0)\subset H^{s_0}_\even$ with suitably small radius $\rho$. In fact, we can first choose $\rho$ so small that 
\begin{equation}\label{MatUinfty}
|U|_{L^\infty}\leq C_{s_0}\|U\|_{H^{s_0}}\leq C_{s_0}\rho\ll1,
\end{equation}
making $I_N+U$ always an invertible matrix. By Theorem \ref{2theo:sc0}, we know that $U\to T^{\gau}_{I_N+U}$ is $C^\infty$ from $\bar{\mathcal{B}}^{s_0}_\rho(0)$ to $\mathcal{L}(H^s,H^s)$ for any $s\in\xR$, together with a similar result for $T^{\droi}_{I_N+U}$. Furthermore, since $R_\PM(A,U)$ is bilinear in $A,U$, by Theorem \ref{PMReg}, we obtain
$$
\begin{aligned}
\|R_\PM(A,U)\|_{H^{2s_0+\tau-n/2}}
&\leq C_{s_0}\|A\|_{H^{s_0+\tau}}\|U\|_{H^{s_0}}
\leq C_{s_0}\|A\|_{H^{s_0+\tau}}\rho,\\
\|R_\PM(A,U_1-U_2)\|_{H^{2s_0+\tau-n/2}}
&\leq C_{s_0}\|A\|_{H^{s_0+\tau}}\|U_1-U_2\|_{H^{s_0}}.
\end{aligned}
$$
On the other hand, the expression (\ref{MatR(U)}) of $\R(U)$ together with the symbolic calculus theorem \ref{2theo:sc} (Remark \ref{2MatSymbol}) show that it is smoothing and quadratic in $U\in H^{s_0}$, so
$$
\begin{aligned}
\|\R(U)\|_{H^{2s_0-n/2-1}}
&\leq C_{s_0}\|U\|_{H^{s_0}}^2
\leq C_{s_0}\rho^2\\
\|\R(U_1)-\R(U_2)\|_{H^{2s_0-n/2-1}}
&\leq C_{s_0}\rho\|U_1-U_2\|_{H^{s_0}}.
\end{aligned}
$$
To summarize, if $U\in \bar{\mathcal{B}}^{s_0}_\rho(0)$, the sum $R_\PM(A,U)+\R(U)$ in the bracket of (\ref{MatParaHom'}) gain $s_0-n/2-1$ 
derivatives. 

Lemma \ref{Dio_Ext} shows that $\bm{L}_{\gamma,\tau}(\omega)$ loses $\tau$ derivatives, 
still leaving $2s_0-n/2-1-\tau>s_0+\delta$ derivatives. Therefore, the mapping $\mathscr{G}$ in (\ref{MatParaHom'}) satisfies
\begin{equation}\label{MatG}
\begin{aligned}
\|\mathscr{G}(\omega,A;U)\|_{H^{s_0}}
&\leq C_{s_0}\gamma^{-1}\big((1+\rho)\|A\|_{H^{s_0+\tau}}+\rho^2\big),\\
\|\mathscr{G}(\omega,A;U_1)-\mathscr{G}(\omega,A;U_2)\|_{H^{s_0}}
&\leq C_{s_0}\gamma^{-1}\big(\|A\|_{H^{s_0+\tau}}+\rho\big)\|U_1-U_2\|_{H^{s_0}}.
\end{aligned}
\end{equation}
Consequently, if $\|A\|_{H^{s_0+\tau}}\leq \varepsilon_0\gamma^2$, with $\varepsilon_0$ depending on $s_0$ suitably small, then one can find a real number $\rho=C_{s_0}\gamma^{-1}\|A\|_{H^{s_0+\tau}}$, such that $\mathscr{G}:\bar{\mathcal{B}}_\rho^{s_0}(0)\to \bar{\mathcal{B}}_\rho^{s_0}(0)$ is a contraction.

In conclusion, we proved that if $A$ is odd and satisfies $\|A\|_{H^{s_0+\tau}}\leq \varepsilon_0\gamma^2$, then the fixed point equation (\ref{Mat_Fix}) has a unique fixed point $U=U(\omega,A)\in H^{s_0}_\even$, satisfying 
\begin{equation}\label{MatU(s_0)}
\|U\|_{H^{s_0}}\leq C_{s_0}\gamma^{-1}\|A\|_{H^{s_0+\tau}}.
\end{equation}
When $\omega\in\Dio[\gamma,\tau]$, $U(\omega,A)$ is a solution of the parahomological equation (\ref{MatParaHom}).

\noindent
\textbf{Step 3: Additional regularity.}

Now suppose in addition that $A\in H^{s+\tau}$ with $s> s_0=2\tau+n/2+1+\delta$. The previous step gives a solution $U=U(\omega,A)\in H^{s_0}$ of the fixed point equation (\ref{Mat_Fix}). Let us then compute, again by the paraproduct remainder estimate in Theorem \ref{PMReg} and symbolic calculus theorem \ref{2theo:sc} (Remark \ref{2MatSymbol}),
$$
\begin{aligned}
T_{I_N+U}^\droi A&+R_\PM(A,U)+\R(U) \\
&\in H^{s+\tau}+H^{s+\tau+s_0-n/2}+H^{2s_0-n/2-1}
\subset H^{\min(s+\tau,s_0+\tau+\delta)}.
\end{aligned}
$$
Since $\D{\omega}^{-1}$ only loses $\tau$ derivatives, by the expression (\ref{MatParaHom'}) again, we find
$$
U=\mathscr{G}(\omega,A;U)\in H^{\min(s,s_0+\delta)}.
$$
Namely $U$ is actually of better regularity. By the standard inductive ``bootstrap argument", we conclude $U\in H^{s}$.

Once this additional regularity is ensured, the tameness of $R_\PM(A,U)$ and $\R(U)$ then gives the desired estimate of $\|U\|_{H^{s}}$. In fact, by (\ref{MatUinfty}) and Theorem \ref{PMReg}, we have
$$
\|R_\PM(A,U)\|_{H^{s+\tau}}
\leq C_{s}\|A\|_{H^{s+\tau}}|U|_{L^\infty}
\leq C_{s}\|A\|_{H^{s+\tau}}.
$$
By Remark \ref{2MatSymbol} and interpolation in Sobolev spaces, we have 
\begin{equation}\label{MatTame}
\begin{aligned}
\|\R(U)\|_{H^{s+\tau}}
&\leq C_{s}\|U\|_{H^{s_0}}\|U\|_{H^{s-\delta}}\\
&\overset{(\ref{MatU(s_0)})}{\leq} 
C_{s}\gamma\varepsilon_0\|U\|_{H^{s-\delta}}\\
&\leq \gamma\big(K\|U\|_{H^{s}}+C_{s,\delta,K}\gamma^{-1}\|A\|_{H^{s_0+\tau}}\big),
\end{aligned}
\end{equation}
with $K=K(s)\ll1$ to be determined. Substituting into the definition of $\mathscr{G}(\omega,A;U)$, we obtain
\begin{equation}\label{MatUs}
\begin{aligned}
\|U\|_{H^s}
&=\|\mathscr{G}(\omega,A;U)\|_{H^s}\\
&\leq C_{s}\gamma^{-1}\|A\|_{H^{s+\tau}}
+C_sK\|U\|_{H^{s}}+C_{s,\delta,K}\gamma^{-1}\|A\|_{H^{s_0+\tau}}.
\end{aligned}
\end{equation}
Thus we can choose $K=K(s)$ to make $C_sK<1/2$, obtaining the desired estimate 
$$
\|U\|_{H^s}\leq C_{s}\gamma^{-1}\|A\|_{H^{s+\tau}}.
$$

\subsection{Lipschitz Dependence}\label{ProofMatRed2}
We then prove that the solution $U(\omega,A)$ of (\ref{Mat_Fix}) is Lipschitz in $(\omega,A)$, whether $\omega$ is Diophantine or not. 

\noindent
\textbf{Step 4: Lipschitz dependence in low norm.}

Let us consider first $A_1,A_2\in H^{s_0+2\tau+1}$. Let us show that, if in addition $\|A_{1,2}\|_{H^{s_0+2\tau+1}}\leq \varepsilon_0\gamma^2$, then $\mathscr{G}(\omega,A;U)$ shall have Lipschitz dependence on $(\omega,A)$ when $\|U\|_{H^{s_0}}\leq C_{s_0}\gamma\varepsilon_0$. 

In fact, we write
\begin{equation}\label{MatDiffGLow}
\begin{aligned}
\Delta_{12}\big(\mathscr{G}(\omega,A,U)\big)
&=\big(\mathscr{G}(\omega_1,A_1,U)-\mathscr{G}(\omega_1,A_2,U)\big)
+\big(\mathscr{G}(\omega_1,A_2,U)-\mathscr{G}(\omega_2,A_2,U)\big).
\end{aligned}
\end{equation}
If $\|U\|_{H^{s_0}}\leq C_{s_0}\gamma\varepsilon_0$, then
\begin{equation}\label{MatDiffGA}
\big\|\mathscr{G}(\omega_1,A_1,U)-\mathscr{G}(\omega_1,A_2,U)\big\|_{H^{s_0}}
\leq C_{s_0}\gamma^{-1}\|\Delta_{12}A\|_{H^{s_0+\tau}}.
\end{equation}
As for the second difference in (\ref{MatDiffGLow}), we have \begin{equation}\label{MatDiffG}
\begin{aligned}
\mathscr{G}&(\omega_1,A_2;U)-\mathscr{G}(\omega_2,A_2;U)\\
&=(T^{\gau}_{(I_N+U_2)^{-1}})^{-1}
\big(\Delta_{12}\bm{L}_{\gamma,\tau}(\omega)\big)
(T^{\gau}_{I_N+U})^{-1}\left(T_{I_N+U}^\droi A_2+R_\PM(A_2,U)+\R(U)\right).
\end{aligned}
\end{equation}
By Lemma \ref{Dio_Ext}, the Fourier multiplier $\Delta_{12}\bm{L}_{\gamma,\tau}(\omega)$ loses $2\tau+1$ derivatives. Therefore, (\ref{MatDiffG}) is estimated as (using Theorem \ref{PMReg} and \ref{2theo:sc0})
\begin{equation}\label{MatDiffGs0}
\begin{aligned}
\|\mathscr{G}&(\omega_1,A_2;U)-\mathscr{G}(\omega_2,A_2;U)\|_{H^{s_0}}\\
&\leq C_{s_0}\gamma^{-2}|\Delta_{12}\omega|\left\|T_{I_N+U}^\droi A_2+R_\PM(A_2,U)+\R(U)\right\|_{H^{s_0+2\tau+1}}\\
&\leq C_{s_0}\gamma^{-2}|\Delta_{12}\omega|\big(\|A_2\|_{H^{s_0+2\tau+1}}
+\|A_2\|_{H^{s_0+2\tau+1}}|U|_{L^\infty}
+\|U\|_{H^{s_0}}^2\big)\\
&\leq C_{s_0}\gamma^{-2}|\Delta_{12}\omega|\|A_2\|_{H^{s_0+2\tau+1}}.
\end{aligned}
\end{equation}
Combining (\ref{MatDiffGA}) and (\ref{MatDiffGs0}), we find that the mapping $U\to\mathscr{G}(\omega,A;U)$ from $\bar{\mathcal{B}}_{\rho}^{s_0}$ to itself has Lipschitz dependence on $(\omega,A)\in \xR^n\times\bar{\mathcal{B}}_{\varepsilon_0\gamma^2}^{s_0+2\tau+1}$:
$$
\begin{aligned}
\big\|\mathscr{G}(\omega_1,A_1,U)-\mathscr{G}(\omega_2,A_2,U)\big\|_{H^{s_0}}
&\leq C_{s_0}\gamma^{-2}|\Delta_{12}\omega|\|A_2\|_{H^{s_0+2\tau+1}}
+C_{s_0}\gamma^{-1}\|\Delta_{12}A\|_{H^{s_0+\tau}}.
\end{aligned}
$$
By Theorem \ref{Depend} concerning Lipschitz dependence of fixed point, we conclude that $U(\omega,A)$ has the desired Lipschitz continuity with $s=s_0$.

\noindent
\textbf{Step 5: Lipschitz dependence in high norm.}

The proof of Lipschitz dependence for $s>s_0$ is a combination of Step 3 and 4. Namely, with $U_i=U(\omega_i,A_i)$, $i=1,2$, we need to estimate
\begin{equation}\label{MatDiffGHigh}
\begin{aligned}
U_1-U_2
&=\big(\mathscr{G}(\omega_1,A_1,U_1)-\mathscr{G}(\omega_1,A_2,U_1)\big)
+\big(\mathscr{G}(\omega_1,A_2,U_1)-\mathscr{G}(\omega_2,A_2,U_1)\big)\\
&\quad+\big(\mathscr{G}(\omega_2,A_2,U_1)-\mathscr{G}(\omega_2,A_2,U_2)\big).
\end{aligned}
\end{equation}

Just as in (\ref{MatDiffGA}) and (\ref{MatDiffGs0}), if $\|U_{1,2}\|_{H^{s_0}}\leq C_0\gamma^{-1}\varepsilon_0$, the first two differences in (\ref{MatDiffGHigh}) have $H^s$ norm bounded by
$$
C_s\gamma^{-1}\|A_1-A_2\|_{H^{s+\tau}}
+C_s\gamma^{-2}|\omega_1-\omega_2|\|A_2\|_{H^{s+2\tau+1}}.
$$
As for the third difference, we may simply apply the regularity gain due to $\R_\CM$ to bound its $H^s$ norm as (using the already proved Item \ref{Mat2})
$$
\begin{aligned}
C_s\gamma^{-1}\big(&\|A_2\|_{H^{s+\tau}}\|U_1-U_2\|_{H^{s_0}}+\|\R(U_1)-\R(U_2)\|_{H^{s+\tau}}\big)\\
&\leq C_s\gamma^{-2}\|A_2\|_{H^{s+\tau}}\big(\|A_1-A_2\|_{H^{s_0+\tau}}
+|\omega_1-\omega_2|\big)\\
&\quad+C_s\gamma^{-1}\|U_1-U_2\|_{H^{s_0}}\big(\|U_1\|_{H^s}+\|U_2\|_{H^s}\big)\\
&\leq C_s\gamma^{-2}\big(\|A_1\|_{H^{s+\tau}}+\|A_2\|_{H^{s+\tau}}\big)\big(\|A_1-A_2\|_{H^{s_0+\tau}}
+|\omega_1-\omega_2|\big)
\end{aligned}
$$

To summarize, we estimate (\ref{MatDiffGHigh}) as
$$
\begin{aligned}
\|U_1-U_2\|_{H^{s}}
&\leq C_s\gamma^{-1}\|A_1-A_2\|_{H^{s+\tau}}\\
&\quad C_s\gamma^{-2}\big(\|A_1\|_{H^{s+2\tau+1}}+\|A_2\|_{H^{s+2\tau+1}}\big)\big(\|A_1-A_2\|_{H^{s_0+\tau}}
+|\omega_1-\omega_2|\big).
\end{aligned}
$$
This is just Item \ref{Mat2}, and hence completes the proof of Theorem \ref{MatRed}.

\section{Paradifferential Approach to Reducibility of Vector Field}\label{sec:6}
\subsection{The Reducibility Theorem}
Throughout this section, we shall continue to follow the notation introduced at the beginning of Part \ref{Part2}. The theorem for the reducibility of a vector field $\omega+X$ on $\xT^n$ close to a parallel one is stated as follows:

\begin{theorem}\label{6VectRed}
Fix a natural number $n$. Given $\gamma\in(0,1]$, $\tau>n-1$, $\delta>0$, let $s_1=2\tau+2+n+\delta$. Denote by $\bar{\mathcal{B}}^{s}_\rho$ the closed ball in $H^{s}(\xT^n,\xR^n)$ centered at 0 with radius $\rho$. One can find a positive number $\varepsilon_1=\varepsilon_1(s_1)\ll1$ and Lipschitz mappings
$$
h:\xR^n\times \bar{\mathcal{B}}^{s_1+2\tau+1}_{\varepsilon_1\gamma^2}
\to\xR^n,
\quad
\eta:\xR^n\times \bar{\mathcal{B}}^{s_1+2\tau+1}_{\varepsilon_1\gamma^2}\to H^{s_1}(\xT^n,\xT^n),
$$
such that for $\omega_1,\omega_2\in\xR^n$ and $X_1,X_2\in H^{s_1+2\tau+1}$,
\begin{equation}\label{VectLip}
\begin{aligned}
|\Delta_{12}h(\omega,X)|
&\leq C_{s_1}\varepsilon_1|\Delta_{12}\omega|
+C_{s_1}\gamma^{-1}\|\Delta_{12}X\|_{H^{s_1+2\tau+1}},\\
\|\Delta_{12}\eta(\omega,X)\|_{H^{s_1}}
&\leq C_{s_1}\varepsilon_1|\Delta_{12}\omega|
+C_{s_1}\gamma^{-1}\|\Delta_{12}X\|_{H^{s_1+2\tau+1}},\\
\end{aligned}
\end{equation}
with the following properties:

\begin{enumerate}    
    \item\label{Vect1}
    If $X\in H^{s+2\tau+1}$ with $s\geq s_1$, then in fact $\eta(\omega,X)\in H^{s+\tau+1}(\xT^n,\xT^n)$, and it is a diffeomorphism of $\xT^n$ close to the identity, satisfying
    $$
    \|\eta(\omega,X)-\Id\|_{H^{s+\tau+1}}\leq C_{s}\gamma^{-1}\|X\|_{H^{s+2\tau+1}},
    $$

    \item\label{Vect2}
    If $\omega_1,\omega_2\in\xR^n$ and $X_1,X_2\in H^{s+2\tau+1}$ with $s\geq s_1$, then there holds a tame Lipschitz estimate: with $\Delta_{12}\eta=\eta(\omega_1,X_1)-\eta(\omega_2,X_2)$
    $$
    \begin{aligned}
    \|\Delta_{12}\eta(\omega,X)\|_{H^{s}}
    &\leq C_{s}\gamma^{-1}\|\Delta_{12}X\|_{H^{s+\tau}}\\
    &\quad+C_{s}\gamma^{-2}\big(\|X_1\|_{H^{s+2\tau+1}}+\|X_2\|_{H^{s+2\tau+1}}\big)\big(|\Delta_{12}\omega|
    +\|\Delta_{12}X\|_{H^{s_1+2\tau+1}}\big).
    \end{aligned}
    $$

    \item\label{Vect3}
    If $\omega,X$ satisfy the Diophantine condition
    $$
    \big|(\omega+h(\omega,X))\cdot\xi\big|\geq\frac{\gamma}{|\xi|^\tau},
    \quad \forall\xi\in\xZ^n\setminus\{0\},
    $$
    then the flow of the vector field $\omega+X$ on $\xT^n$ conjugates to the flow of the parallel vector field $\omega+h(\omega,X)$ on $\xT^n$ via $\eta$.

\end{enumerate}
\end{theorem}

\begin{corollary}\label{6VectRed'}
Suppose in addition to all the assumptions of Theorem \ref{6VectRed}, the perturbation $X$ is an even function. Then the diffeomorphism $\eta=\eta(\omega,X)$ is an odd function on $\xT^n$.
\end{corollary}

\subsection{A Modified Problem}
Just as in classical KAM literature, for example \cite{Moser19662}, we first solve a modified problem instead of directly addressing Theorem \ref{6VectRed}:
\begin{problem}\label{VectMod}
Given $\bar\omega\in\Dio[\gamma,\tau]$, find a shift $\lambda=\lambda(\bar\omega,X)\in\xR^n$ such that the flow of $\bar\omega+X-\lambda$ conjugates to that of $\bar\omega$. 
\end{problem}

\noindent
\textbf{Step 1: Paralinearization for the modified problem.}

Let $\phi_t$ and $\zeta_t$ be the flow of $\bar\omega$ and $\bar\omega+X-\lambda$ respectively. Problem \ref{VectMod} is equivalent to finding a diffeomorphism $\eta=\Id+\theta$ on $\xT^n$, with $\theta:\xT^n\to\xR^n$ of small magnitude, such that 
$$
\forall t\in\mathbb{R},\quad \eta\circ\phi_t=\zeta_t\circ\eta.
$$
Differentiating with respect to $t$, this is equivalent to finding $\eta$ and $\lambda$ such that
\begin{equation}\label{VectConj0}
\D{\bar\omega}\eta
=\bar\omega+X\circ\eta-\lambda.
\end{equation}
This leads to solving $\theta$ and $\lambda$ from the partial differential equation
\begin{equation}\label{VectConj}
\D{\bar\omega} \theta
=X\circ(\Id+\theta)-\lambda.
\end{equation}
So we define a nonlinear mapping $\mathscr{F}$ acting on $(X,\theta)\in H^{s_1+\tau}\times H^{s_1}$ by
$$
\mathscr{F}(X,\theta)=\D{\bar\omega} \theta-X\circ(\Id+\theta),
$$
and look for $\theta\in H^{s_1}$ and a constant $\lambda\in\xR^n$ such that $\mathscr{F}(X,\theta)+\lambda=0$. 

Notice that the refined paralinearization theorem \ref{ParaLinRefined} gives\footnote{Strictly speaking, we should use the notation $T^{\gau}_A$ or $\Op^\PM_\gau(A)$ since the paradifferential operators are matrix-valued. However, we can easily justify from the context that the operators are all acting \emph{from the left}, so we abbreviate the superscript ``left" to simplify notation.}
\begin{equation}\label{VectParaLinF}
\mathscr{F}(X,\theta)=\D{\bar\omega} \theta-T_{X'\circ(\Id+\theta)}\theta-\eta^\star X-\R_{\PLR}(X,\theta).
\end{equation}
As in the previous section, the crux of the proof is to express the factor 
$X'\circ(\Id+\theta)$ in terms of $\mathscr{F}(X,\theta)$. This identity is rooted in the proof by 
Zehnder of his variant of the Nash-Moser scheme. 
For this problem, this amounts to differentiate $\mathscr{F}(X,\theta)$, to get
$$
\partial\big(\mathscr{F}(X,\theta)\big)
=\D{\bar\omega}\partial \theta
-X'\circ(\mathrm{Id}+\theta)\cdot(I_n+\partial \theta),
$$
from which we solve
$$
X'\circ(\mathrm{Id}+\theta)=
\big[(\bar\omega\cdot\partial)\partial \theta\big]\cdot(I_n+\partial \theta)^{-1}-\partial\big(\mathscr{F}(X,\theta)\big)\cdot(I_n+\partial \theta)^{-1}.
$$
Substituting the above expression of $X'\circ(\Id+\theta)$ in the paralinearization formula (\ref{VectParaLinF}), we conclude that
\begin{equation}\label{VectParaLinF'}
\begin{aligned}
\mathscr{F}(X,\theta)
&=\D{\bar\omega} \theta-
\Op^\PM\big[\D{\bar\omega}\partial \theta\cdot(I_n+\partial \theta)^{-1}\big]\theta-\eta^\star X-\R_\PLR(X,\theta)\\
&\quad
+\Op^\PM\left[\partial\big(\mathscr{F}(X,\theta)\big)\cdot(I_n+\partial \theta)^{-1}\right]\theta\\
&=T_{(I_n+\partial \theta)}\D{\bar\omega} T_{(I_n+\partial \theta)^{-1}}\theta
-\eta^\star X-\R_\PLR(X,\theta)-\R_1(\theta)\\
&\quad
+\Op^\PM\left[\partial\big(\mathscr{F}(X,\theta)\big)\cdot(I_n+\partial \theta)^{-1}\right]\theta.
\end{aligned}
\end{equation}
The last equality follows from the Leibniz rule $\partial(T_A\theta)=T_{\partial A}\theta+T_{A}\partial \theta$ together with composition of paraproducts. The remainder $\R_1(\theta)$ is caused by composition of paraproducts due to Theorem \ref{2theo:sc}: with 
$$
\begin{aligned}
A&=I_n+\partial \theta\in H^{s_1-1}\subset C^{s_1-n/2-1},\\
B&=\D{\bar\omega}\partial \theta\cdot(I_n+\partial \theta)^{-1}\in H^{s_1-2}\subset C^{s_1-n/2-2},
\end{aligned}
$$
we have
\begin{equation}\label{VectR1(theta)}
\R_1(\theta)=(T_{A}T_{A^{-1}}-1)\D{\bar\omega} \theta-(T_AT_{A^{-1}B}-T_{B})\theta
\in H^{2s_1-n/2-2}.
\end{equation}

We now introduce an equation, still called the \emph{parahomological equation}, which collects all terms but the last one in the right-hand-side of (\ref{VectParaLinF'}). It is an equation for the unknown $(\theta,\lambda)\in H^{s_1}\times\xR^n$:
\begin{equation}\label{VectParahomo}
T_{(I_n+\partial \theta)}\D{\bar\omega} T_{(I_n+\partial \theta)^{-1}}\theta
=\eta^\star X+\R_\PLR(X,\theta)+\R_1(\theta)-\lambda.
\end{equation}

Just as in Subsection \ref{ProofMatRed1}, we assert there exists some constant $\rho_0$ depending only on $s_1$, such that if $\theta$ solves (\ref{VectParahomo}) with $\|\theta\|_{H^{s_1}}\leq\rho_0$, then \emph{necessarily} $\mathscr{F}(X,\theta)+\lambda=0$. In fact, if $\theta$ solves (\ref{VectParahomo}), then the paralinearization formula (\ref{VectParaLinF'}) becomes
$$
\mathscr{F}(X,\theta)+\lambda
=\Op^\PM\left[\partial\big(\mathscr{F}(X,\theta)+\lambda\big)\cdot(I_n+\partial \theta)^{-1}\right]\theta.
$$
Again by Theorem \ref{2theo:sc}, if $\|\theta\|_{H^{s_1}}\ll_{s_1}1$, then using Sobolev embedding $H^{s_1}\subset C^1$,
$$
\begin{aligned}
\big|\mathscr{F}(X,\theta)+\lambda\big|_{C^1}
&\leq C_{s_1}\big|\mathscr{F}(X,\theta)+\lambda\big|_{C^1}\|\theta\|_{H^{s_1}}\\
&\leq \frac{1}{2}\big|\mathscr{F}(X,\theta)+\lambda\big|_{C^1}.
\end{aligned}
$$
This forces $\mathscr{F}(X,\theta)+\lambda=0$, thus giving a solution $(\theta,\lambda)$ of (\ref{VectConj}).

It suffices to find a sufficiently small solution to the parahomological equation (\ref{VectParahomo}).

\noindent
\textbf{Step 2: Banach fixed point for the modified problem.}

The next step is to solve the parahomological equation (\ref{VectParahomo}). We claim that, when $\bar\omega\in\Dio[\gamma,\tau]$,  (\ref{VectParahomo}) is implied by the fixed point form equation
\begin{equation}\label{VectParahomoFix}
\begin{aligned}
\theta&=\mathscr{G}(\bar\omega,X;\theta)\\
&:=T_{(I_n+\partial \theta)^{-1}}^{-1}\bm{L}_{\gamma,\tau}(\bar\omega)T_{(I_n+\partial \theta)}^{-1}\big(\eta^\star X+\R_\PLR(X,\theta)+\R_1(\theta)-\lambda(\theta)\big)
\end{aligned}
\end{equation}
where the Fourier multiplier $\bm{L}_{\gamma,\tau}(\bar\omega)$ is as in Lemma \ref{Dio_Ext}, and the constant $\lambda(\theta)\in\xR^n$ is given by 
\begin{equation}\label{6VectShift}
\lambda(\theta)=\Avg\left[T_{(I_n+\partial \theta)}^{-1}\big(\eta^\star X+\R_\PLR(X,\theta)+\R_1(\theta)\big)\right].
\end{equation}

In fact, in order that (\ref{VectParahomo}) has a solution, the average of 
$$
T_{(I_n+\partial \theta)}^{-1}\big(\eta^\star X+\R_\PLR(X,\theta)+\R_1(\theta)-\lambda\big)
=\D{\bar\omega} \theta
$$
has to be zero. Recalling the definition of paraproduct (\ref{T_au}), we find the action of $T_{I_n+\partial \theta}$ on a constant vector is nothing but the identity, so the unique expression of $\lambda$ must be (\ref{6VectShift}). With this $\lambda$ fixed, when $\bar\omega\in\Dio[\gamma,\tau]$, we recall from Lemma \ref{Dio_Ext} that $\bm{L}_{\gamma,\tau}(\bar\omega)=\D{\omega}^{-1}$, so the solution of (\ref{VectParahomoFix}) satisfies the parahomological equation (\ref{VectParahomo}), and hence gives a solution of the modified conjugacy problem \ref{VectMod}.

It thus suffices to solve the fixed point equation (\ref{VectParahomoFix}).

Let us then verify that, under the assumptions of Theorem \ref{6VectRed}, equation (\ref{VectParahomoFix}) 
may be solved by the standard Banach fixed-point theorem. We may suppose that 
$|\theta|_{C^3}$ is small, while $\theta\in\bar{\mathcal{B}}^{s_1}_\rho\subset H^{s_1}$ where $\rho$ is to be determined. Just as in Subsection \ref{ProofMatRed1}, we obtain from Theorem \ref{2theo:sc0} the following estimates:
\begin{equation}\label{VectR(theta)s1}
\begin{aligned}
\|\R_1(\theta)\|_{H^{2s_1-n/2-2}}
&\leq C_{s_1}\rho^2,\\
\|\R_1(\theta_1)-\R_1(\theta_2)\|_{H^{2s_1-n/2-2}}
&\leq C_{s_1}\rho\|\theta_1-\theta_2\|_{H^{s_1}}.
\end{aligned}
\end{equation}
The continuity theorem \ref{PCNormCont} gives the following estimates: with $\eta_i=\Id+\theta_i$,
\begin{equation}\label{Vectetaf(s1)}
\begin{aligned}
\|\eta^\star X\|_{H^{s_1+\tau}}
&\leq C_{s_1}\|X\|_{H^{s_1+\tau}},\\
\|\eta_1^\star X-\eta_2^\star X\|_{H^{s_1+\tau}}
&\leq C_{s_1}|\theta_1-\theta_2|_{C^2}\|X\|_{H^{s_1+\tau+1}}.
\end{aligned}
\end{equation}
Theorem \ref{ParaLinRefined} implies that  
\begin{equation}\label{VectPLA(s1)}
\begin{aligned}
\|\R_\PLR(X,\theta)\|_{H^{2s_1-n-\delta}}
&\leq C_{s_1}\|X\|_{H^{s_1+\tau}}\rho,\\
\|\R_\PLR(X,\theta_1)-\R_\PL(X,\theta_2)\|_{H^{2s_1-n-\delta}}
&\leq C_{s_1}\|X\|_{H^{s_1+\tau}}\|\theta_1-\theta_2\|_{H^{s_1}}.
\end{aligned}
\end{equation}
Eventually, notice that 
$\lambda(\theta)$ is Lipschitz continuous in $\theta\in H^{s_1}$.

Recall from Lemma \ref{Dio_Ext} that $\bm{L}_{\gamma,\tau}(\bar\omega)$ loses $\tau$ derivatives. Noticing $s_1=2\tau+n+2+\delta$, we have
$$
\begin{aligned}
2s_1-\frac{n}{2}-2-\tau> s_1+\delta
\quad\text{and}\quad
2s_1-n-\delta-\tau> s_1+1.
\end{aligned}
$$
Consequently, the loss of $\tau$ derivatives in (\ref{VectParahomoFix}) is balanced by the gain due to $\eta^\star X$, $\R_\PLR(X,\theta)$ and $\R_1(\theta)$, yielding
$$
\begin{aligned}
\|\mathscr{G}(\bar\omega,X;\theta)\|_{H^{s_1}}
&\leq C_{s_1}\gamma^{-1}\big((1+\rho)\|X\|_{H^{s_1+\tau}}+\rho^2\big)\\
\|\mathscr{G}(\bar\omega,X;\theta_1)-\mathscr{G}(\bar\omega,X;\theta_2)\|_{H^{s_1}}
&\leq C_{s_1}\gamma^{-1}\big(\|X\|_{H^{s_1+\tau+1}}+\rho\big)\|\theta_1-\theta_2\|_{H^{s_1}}.
\end{aligned}
$$
Therefore, if $\|X\|_{H^{s_1+\tau+1}}\leq\varepsilon_1\gamma^2$ with $\varepsilon_1$ depending on $s_1$ suitably small, then one can find a real number $\rho=C_{s_1}\gamma^{-1}\|X\|_{H^{s_1+\tau+1}}$, such that $\mathscr{G}(\bar\omega,X;\theta)$ is a contraction for $\theta\in\bar{\mathcal{B}}^{s_1}_\rho$. We thus obtain the unique solution $\theta=\Phi(\bar\omega,X)$ of (\ref{VectParahomoFix}), satisfying
\begin{equation}\label{Vecttheta(s_1)}
\|\Phi(\bar\omega,X)\|_{H^{s_1}}\leq C_{s_1}\gamma^{-1}\|X\|_{H^{s_1+\tau+1}}.
\end{equation}
\noindent
\textbf{Step 3: Lipschitz estimates for (\ref{VectParahomoFix}).}

Let us now determine Lipschitz dependence on parameters. Just as in proving Theorem \ref{MatRed}, this follows from the general Lipschitz dependence on parameters for Banach fixed point theorem.

We already know that for $(\bar\omega,X)\in\xR^n\times\bar{\mathcal{B}}_{\varepsilon_1\gamma^2}^{s_1+2\tau+1}$, the mapping $\mathscr{G}(\bar\omega,X;\cdot)$ has a unique fixed point $\Phi(\bar\omega,X)\in\bar{\mathcal{B}}^{s_1}_{\rho}$, with $\rho=C_{s_1}\gamma\varepsilon_1$. We now estimate how $\mathscr{G}(\bar\omega,X;\theta)$ changes when $(\bar\omega,X)\in\xR^n\times\bar{\mathcal{B}}^{s_1+2\tau+1}_{\varepsilon_1\gamma^2}$ change for $\theta\in \bar{\mathcal{B}}^{s_1}_{\rho}$. We compute
\begin{equation}\label{VectDiffG}
\begin{aligned}
\Delta_{12}\mathscr{G}(\bar\omega,X;\theta)
&=\big(\mathscr{G}(\bar\omega_1,X_1;\theta)-\mathscr{G}(\bar\omega_1,X_2;\theta)\big)\\
&\quad+T_{(I_n+\partial \theta)^{-1}}^{-1}\left(\Delta_{12}\bm{L}_{\gamma,\tau}(\bar\omega)\right)
T_{(I_n+\partial \theta)}^{-1}\big(\eta^\star X_2+\R_\PLR(X_2,\theta)+\R_1(\theta)-\lambda(\theta)\big).
\end{aligned}
\end{equation}

For the first difference in (\ref{VectDiffG}), we argue as in (\ref{MatDiffGA}), together with the estimates already obtained in the last step, to find
\begin{equation}\label{VectDiff(X)}
\|\mathscr{G}(\bar\omega_1,X_1;\theta)-\mathscr{G}(\bar\omega_1,X_2;\theta)\|_{H^{s_1}}
\leq C_{s_1}\gamma^{-1}\|X_1-X_2\|_{H^{s_1+\tau+1}}
\end{equation}

As for the second difference in (\ref{VectDiffG}), we know from Lemma \ref{Dio_Ext} that the Fourier multiplier $\Delta_{12}\bm{L}_{\gamma,\tau}(\bar\omega)$ loses $2\tau+1$ derivatives. Therefore, for $\theta\in\bar{\mathcal{B}}^{s_1}_\rho$, the second difference in (\ref{VectDiffG}) is estimated as
\begin{equation}\label{DiffGomega}
\begin{aligned}
\|\Delta_{12}\mathscr{G}&(\bar\omega_1,X_2;\theta)-\mathscr{G}(\bar\omega_2,X_2;\theta)\|_{H^{s_1}}\\
&\leq C_{s_1}\gamma^{-2}|\Delta_{12}\bar\omega|
\left\|\eta^\star X_2+\R_\PLR(X_2,\theta)+\R_1(\theta)-\lambda(\theta)\right\|_{H^{s_1+2\tau+1}}\\
&\leq C_{s_1}\gamma^{-2}|\Delta_{12}\bar\omega|
\left((1+\rho)\|X_2\|_{H^{s_1+2\tau+1}}+\rho^2\right)\\
&\leq C_{s_1}\varepsilon_1|\Delta_{12}\bar\omega|.
\end{aligned}
\end{equation}
Here we used the estimates established in the previous step again. Combining (\ref{VectDiff(X)}) and (\ref{DiffGomega}), using again Theorem \ref{Depend} concerning Lipschitz dependence of fixed point, we obtain
\begin{equation}\label{VectLipFix(low)}
\|\Delta_{12}\Phi(\bar\omega,X)\|_{H^{s_1}}
\leq C_{s_1}\gamma^{-1}\|\Delta_{12}X\|_{H^{s_1+\tau+1}}
+C_{s_1}\varepsilon_1|\Delta_{12}\bar\omega|.
\end{equation}

We then substitute these back to the expression (\ref{6VectShift}) of $\lambda(\theta)$. Note that the definition of paracomposition $\eta^\star$ and $\R_\PLR(X,\theta)$, $\R_1(\theta)$ (see (\ref{VectR1(theta)})) only requires $\theta\in C^2$. Thus, by averaging on $\xT^n$, we obtain the Lipschitz continuity of the shift $\lambda$ with respect to $(\bar\omega,X)$: denoting the shift as $\lambda(\bar\omega,X)$, we have
\begin{equation}\label{VectLipShift}
|\Delta_{12}\lambda(\bar\omega_,X)|
\leq C_{s_1}\varepsilon_1|\Delta_{12}\bar\omega|
+C_{s_1}\gamma^{-1}\|\Delta_{12}X\|_{H^{s_1+2\tau+1}}.
\end{equation}

\subsection{Back to the Original Problem}
Having solved the modified problem \ref{VectMod}, we can then pass to the original problem and prove Theorem \ref{6VectRed}.

\noindent
\textbf{Step 4: Determining $h$ and $\eta$.}

Let us summarize the result of Step 3 as follows. For some small $\varepsilon_1\ll1$ depending on $s_1$, there is a shift $\lambda(\bar\omega,X)$ and a perturbation $\theta=\Phi(\bar\omega,X)$, both Lipschitz in $(\bar\omega,X)\in\xR^n\times\bar{\mathcal{B}}^{s_1+2\tau+1}_{\varepsilon_1\gamma^2}$, such that if $\bar\omega\in\Dio[\gamma,\tau]$, then $\lambda(\bar\omega,X)$ and $\theta=\Phi(\bar\omega,X)$ give the unique solution of Problem \ref{VectMod}.

Notice (\ref{VectLipShift}) implies that the Lipschitz constant of $\lambda(\bar\omega,X)$ with respect to $\bar\omega$ is small since $\varepsilon_1\ll1$, so the classical Lipschitz inversion theorem implies that the mapping $\bar\omega\to\bar\omega-\lambda(\bar\omega,X)$ is in fact a bi-Lipschitz homeomorphism on $\mathbb{R}^n$. We shall denote its inverse as $\omega+h(\omega,X)$. As a consequence of the classical Lipschitz inversion theorem, $h(\omega,X)$ satisfies (\ref{VectLip}). We then define
$$
\eta(\omega,X):=\Id+\Phi\big(\omega+h(\omega,X),X\big),
\quad\mathbb{R}^n\times\bar{\mathcal{B}}^{s_1+2\tau+1}_{\varepsilon_1\gamma^2}\to H^{s_1}.
$$
Then by (\ref{VectLipFix(low)})(\ref{VectLipShift}), we find $\eta$ satisfies (\ref{VectLip}). 

The solution of Problem \ref{VectMod} exactly shows that when $\bar\omega=\omega+h(\omega,X)\in\Dio[\gamma,\tau]$, the flow of $\omega+X=\bar\omega+X-\lambda(\bar\omega,X)$ conjugates to the flow of $\omega+h(\omega,X)$, via the diffeomorphism $\eta(\omega,X)=\Id+\Phi(\bar\omega,X)$. This justifies Item \ref{Vect3} of Theorem \ref{6VectRed}.

To obtain the additional regularity $\eta\in H^{s_1+\tau}$ in Item \ref{Vect1} of Theorem \ref{6VectRed}, we simply employ (\ref{VectR(theta)s1})(\ref{Vectetaf(s1)})(\ref{VectPLA(s1)}) to find
\begin{equation}\label{Vectu(s_1+tau)}
\|\Phi(\bar\omega,X)\|_{H^{s_1+\tau+1}}\leq C_{s_1}\gamma^{-1}\|X\|_{H^{s_1+2\tau+1}}.
\end{equation}

Corollary \ref{6VectRed'} follows readily from the fixed point form modified problem. We simply notice that if $X$ is even and $\theta$ is odd, then $\mathscr{G}(\bar\omega,X;\theta)$ in (\ref{VectParahomoFix}) gives an odd function, because paraproduct and paralinearization do not affect parity. The fixed point arguments can thus be conducted within the space of odd functions.

\noindent
\textbf{Step 5: Additional regularity.}

For Item \ref{Vect2} of Theorem \ref{6VectRed}, we may argue similarly as in step 3 of proving Theorem \ref{MatRed}. Namely, if $s>s_1$ and $X\in H^{s+2\tau+1}$, then with $\bar\omega=\omega+h(\omega,X)$, $\theta=\Phi(\bar\omega,X)$, $\eta=\Id+\theta$, we may apply Theorem \ref{2theo:sc2}, \ref{PCNormCont}, \ref{ParaLinRefined} again to obtain
$$
\begin{aligned}
\eta^\star X+\R_\PLR(X,\theta)+\R_1(\theta)-\lambda(\bar\omega,X)
&\in H^{s+2\tau+1}+H^{2s_1-n-\delta}+H^{2s_1-n/2-2}+C^\infty\\
&\subset H^{s_1+2\tau+2}.
\end{aligned}
$$
Taking into account the equation (\ref{VectParahomoFix}) satisfied by $\theta$, we find $\eta\in H^{s+\tau+1}$. 

Having established this additional regularity, the tame estimate $\|\theta\|_{H^{s+\tau+1}}\lesssim_s\gamma^{-1}\|X\|_{H^{s+2\tau+1}}$ then follows easily. In fact, as long as $\eta$ stays near $\Id$ in $C^2$ topology, we have
\begin{equation}\label{VectetafHs}
\|\eta^\star X\|_{H^{s+2\tau+1}}
\leq C_s\|X\|_{H^{s+2\tau+1}}.
\end{equation}
We may also recast (\ref{MatTame}) for $\R_1(\theta)$ to obtain
\begin{equation}\label{VectR1uHs}
\begin{aligned}
\|\R_1(\theta)\|_{H^{s+2\tau+1}}
&\leq \gamma\big(K\|\theta\|_{H^{s+\tau+1}}+C_{s,\delta,K}\gamma^{-1}\|X\|_{H^{s_1+2\tau+1}}\big),\\
&\quad K=K(s)\ll1\text{ to be decided.}
\end{aligned}
\end{equation}

To estimate $\|\R_{\PLR}(X,\theta)\|_{H^{s+2\tau+1}}$, we use the refined paralinearization theorem \ref{ParaLinRefined} with $r=(s_1+s)/2-n/2=(s+2\tau+2+\delta)/2$, $r'=r-\delta/2$, and consider $X$ as in $H^{(s_1+s)/2+2\tau+1}$. Then we estimate
$$
\begin{aligned}
\|\R_{\PLR}(X,\theta)\|_{H^{s+2\tau+1}}
&\leq C_s\big\|X;H^{(s_1+s)/2+2\tau+1}\big\|\cdot|\theta|_{C^{(s_1+s)/2-n/2}}\\
&\leq C_s\big\|X;H^{(s_1+s)/2+2\tau+1}\big\|\cdot\|\theta\|_{H^{(s_1+s)/2}}.
\end{aligned}
$$
By interpolation inequalities in Sobolev spaces, we obtain
$$
\begin{aligned}
\|\R_{\PLR}(X,\theta)\|_{H^{s+2\tau+1}}
&\leq C_s\|X\|_{H^{s+2\tau+1}}^{1/2}\|X\|_{H^{s_1+2\tau+1}}^{1/2}
\left(\beta\|\theta\|_{H^s}+\beta^{-1}\|\theta\|_{H^{s_1}}\right).
\end{aligned}
$$
We then choose $\beta=C_s^{-1}K\gamma\|X\|_{H^{s+2\tau+1}}^{-1/2}\|X\|_{H^{s_1+2\tau+1}}^{-1/2}$, where $K=K(s)\ll1$ is to be determined; then 
\begin{equation}\label{VectRPLAHs}
\begin{aligned}
\|\R_{\PLR}(X,\theta)\|_{H^{s+2\tau+1}}
&\leq 
\gamma K\|\theta\|_{H^s}+C_{s,K}\gamma^{-1}\|X\|_{H^{s+\tau+1}}\|X\|_{H^{s_1+\tau+1}}\|\theta\|_{H^{s_1}}\\
&\leq \gamma K\|\theta\|_{H^{s}}+C_{s,K}\gamma^2\varepsilon_1^2\|X\|_{H^{s+\tau+1}}.
\end{aligned}
\end{equation}
Summing up (\ref{VectetafHs})(\ref{VectR1uHs})(\ref{VectRPLAHs}), we obtain the following estimate for $\theta=\mathscr{G}(\bar\omega,X;\theta)$:
\begin{equation}\label{VectHs}
\begin{aligned}
\|\theta\|_{H^{s+\tau+1}}
&\leq C_s\gamma^{-1}\left(\|\eta^\star X\|_{H^{s+2\tau+1}}+\|\R_1(\theta)\|_{H^{s+2\tau+1}}+\|\R_{\PLR}(X,\theta)\|_{H^{s+2\tau+1}}\right)\\
&\leq C_sK\|\theta\|_{H^{s}}+C_s\gamma^{-1}\|X\|_{H^{s+2\tau+1}}.
\end{aligned}
\end{equation}
If we choose $K$ to make $C_sK<1/2$, then this inequality exactly gives the desired estimate of Item \ref{Vect2}.

For Item \ref{Vect3}, we have
$$
\begin{aligned}
\eta(\omega_1,X_1)-\eta(\omega_2,X_2)
&=\Phi\big(\omega_1+h(\omega_1,X_1),X_1\big)-\Phi\big(\omega_2+h(\omega_2,X_2),X_2\big).
\end{aligned}
$$
We can recast the estimate (\ref{VectHs}), but this time for the difference of two fixed points, and the computations largely remain similar.

\section{Proof of Theorem \ref{1NLPHyp}}\label{sec:7}

The proof of Theorem \ref{1NLPHyp} is divided into two major steps. We first paralinearize (\ref{1NLPEQ}) into
$$
\big[(\omega+\varepsilon T_{Y})\cdot\nabla_x\big]u
+\varepsilon T_{A}u+\varepsilon\R(X,F;u)=\varepsilon f,
$$
which is equation (\ref{7ParalinEQ}) in the sequel (together with the meaning of $Y,A,\R$), and prove \emph{paradifferential reducibility}. That is, we conjugate this equation with an appropriate paracomposition $\chi^\star$ and paradifferential operator $T_{I_N+U}$, to transform it into 
$$
\big((\omega+h(\omega,\varepsilon Y))\cdot\nabla_x\big)y=\varepsilon T_{I_N+U}^{-1}\chi^\star f+\mathfrak{F}(\omega,y),
\quad\text{provided }\omega+h(\omega,\varepsilon Y)\text{ is Diophantine}.
$$
This is equation (\ref{7Paray}) in the sequel. It is a \emph{constant coefficient equation} in the new unknown $y$ modulo smoothing remainder $\mathfrak{F}(\omega,y)$: the shifted vector field $\omega+h(\omega,\varepsilon Y)$ is parallel. The new unknown $y=T_{I_N+U}^{-1}\chi^\star u$ is defined by Proposition \ref{7NewVar}. 

Next, we show that the constant coefficient equation is implied by the following fixed point form equation (with $\bm{L}_{\gamma,\tau}$ as in Lemma \ref{Dio_Ext}) when $\omega+h(\omega,\varepsilon Y)$ is Diophantine:
$$
y=\bm{L}_{\gamma,\tau}\big(\omega+h(\omega,\varepsilon Y)\big)\left[\varepsilon T_{I_N+U}^{-1}\chi^\star f+\mathfrak{F}(\omega,y)\right].
$$
Whether $\omega$ satisfies this Diophantine requirement or not, this equation can always be solved by \emph{standard Banach fixed point argument}, since the loss of regularity caused by $\bm{L}_{\gamma,\tau}\big(\omega+h(\omega,\varepsilon Y)\big)$ is balanced by the smoothing right-hand-side. This yields a solution $u=u(\omega)$. The set of $\omega$ making $\omega+h(\omega,\varepsilon Y)$ Diophantine is ``very ample" in the sense of Lebesgue measure. Therefore, there are ``plenty" of $\omega$ for which it gives a solution of the original problem.

To the authors' knowledge, this paradifferential reducibility method differs from most of the existing literature on ``KAM for PDEs", where only the \emph{linearized equation} is reduced to constant coefficient form. If we regard Sections \ref{sec:2}-\ref{sec:6} as black boxes of paradifferential calculus to be exploited, then the proof occupies very few pages.

We now enter into the details. To simplify the notation, we will not explicitly indicate the dependence of constants on $X$ and $F$ below. We continue to denote by $\bar{\mathcal{B}}^{s}_\rho$ the closed ball in $H^s$ centered at $0$ with radius $\rho$. We also assume \emph{a priori} that for some fixed $s > n/2 + 1$, whose value to be specified later, we have $u \in H^s_\even(\xT^n,\xR^N)$ with $\lA u\rA_{H^s}\leq 1$, so that $\varepsilon$ is the only parameter that needs to be suitably small. 

We begin by rewriting (\ref{1NLPEQ}) into a paradifferential form using Bony's theorem \ref{2ParaLin}:
\begin{equation}\label{7ParalinEQ}
\big[(\omega+\varepsilon T_{Y})\cdot\nabla_x\big]u
+\varepsilon T_{A}u+\varepsilon\R(X,F;u)=\varepsilon f
\end{equation}
where $Y(x):=X(x,u(x))\in H^s_\even(\xT^n;\xR^n)$, and
\begin{equation}\label{7A(x)}
\begin{aligned}
A(x)&:=F'_z\big(x,u(x)\big)
+\sum_{j=1}^n\big(\partial_ju(x)\big)\otimes X_{j;z}'\big(x,u(x)\big)
\in H^{s-1}_{\odd}\otimes\mathbf{M}_{N}(\xR).
\end{aligned}
\end{equation}
Furthermore, the paralinearization remainder 
\begin{equation}\label{7R(X,F,u)}
\R(X,F;u)
=\R_\CM\big(X(x,u),\nabla_xu\big)
+\R_{\PL}(F,u)+T_{\nabla_x u}\cdot\R_{\PL}(X,u)
\in H^{2s-n/2-1}_\odd,
\end{equation}
so that $\|\R(X,F;u)\|_{H^{2s-n/2-1}}\lesssim_s1$. The parity of these functions are all easy to check since paraproduct and paralinearization preserve parities. Since $X,F$ are smooth, the mappings $H^s\ni u\mapsto Y\in H^{s}$ and $H^s\ni u\mapsto A\in H^{s-1}$ are smooth. The remainder $\R_\CM$ is bilinear, while the remainder $\R_{\PL}(F,u)+T_{\nabla_x u}\cdot\R_{\PL}(X,u)\in H^{2s-n/2}$ is in fact $C^\infty$ in $u$ by Theorem~\ref{2ParaLin}. 

\subsection{Reducibility of Linear Paradifferential Operator}
Let us fix three real numbers:
$$
\tau>n-1,\quad \gamma\in(0,1],\quad \delta>0,
$$
and still write, as in Theorem \ref{MatRed} and \ref{6VectRed},
$$
s_1=2\tau+2+n+\delta,\quad
s_0=2\tau+1+n/2+\delta.
$$
We shall fix a real number 
$$
s\geq s_1+2\tau+2
$$
whose value will be explicitly determined later. The set $\Dio[\gamma,\tau]$ still denotes the set of Diophantine frequencies with parameter $\gamma,\tau$ as in (\ref{diopc}). 

Suppose now we have a vector field $Y\in H^s_\even(\xT^n,\xR^n)$ and a matrix-valued function $A\in H^{s-1}_\odd\otimes\mathbf{M}_N(\xR)$. Suppose $Y,A$ both have size $\simeq1$ in the function spaces $H^{s_1},H^{s_0}$ respectively, so $\varepsilon>0$ is the only parameter that should be suitably small. Let us look at a general symbol
\begin{equation}\label{7p(x,xi)}
p(x,\xi):=i\big(\omega+\varepsilon Y(x)\big)\cdot\xi
+\varepsilon A(x)
\in\Gamma^1_{s-n/2}+\Gamma^0_{s-n/2-1},
\end{equation}
corresponding to the general linear paradifferential operator $T_p=\big(\omega+\varepsilon T_{Y}\big)\cdot\nabla_x+\varepsilon T_A$. This obviously corresponds to the paralinearized equation (\ref{7ParalinEQ}) if $Y,A$ are related to $u$ in the manner indicated therein. 

Our aim is to reduce the general linear paradifferential operator $T_p$ into a constant coefficient form modulo a smoothing remainder. To do so, we will rely on Theorem \ref{6VectRed} and Corollary \ref{6VectRed'}. 

Suppose $\varepsilon\leq\varepsilon_1\gamma^2$. Theorem \ref{6VectRed} immediately gives a shift $h(\omega,\varepsilon Y)\in\xR^n$ and an odd diffeomorphism $\eta=\eta(\omega,\varepsilon Y)\in H^{s-\tau}_\odd$, such that if $\omega+h(\omega,\varepsilon Y)\in\Dio[\gamma,\tau]$, then $\eta(\omega,\varepsilon Y)$ conjugates the vector field $\omega+\varepsilon Y$ to the parallel one $\omega+h(\omega,\varepsilon Y)$, namely (see (\ref{VectConj0}))
\begin{equation}\label{7Eqneta}
\big((\omega+h(\omega,\varepsilon Y))\cdot\partial\big)(\eta-\Id)=\varepsilon Y\circ\eta.
\end{equation}
Let us define the diffeomorphism 
\begin{equation}\label{7chi}
\chi:=\eta(\omega,\varepsilon Y)^{-1}\in H^{s-\tau}_\odd.
\end{equation}
By (\ref{7Eqneta}), if $\omega+h(\omega,\varepsilon Y)\in\Dio[\gamma,\tau]$, then $\chi$ satisfies 
$$
p\left( \chi(x),\chi'(x)^{-\T}\xi\right)=i(\omega+h(\omega,\varepsilon Y))\cdot\xi
+\varepsilon A\circ\chi,
$$
where recall that $-\T$ stands for taking the inverse of transpose. 

We then apply the parachange of variable formula in Theorem \ref{4PCConj} to the diffeomorphism $\chi$ and symbol $p(x,\xi)$. Since $p(x,\xi)$ is a classical differential symbol, the conjugation formula (\ref{q(x,xi)}) reduces to the usual chain rule and contains only the leading term. We thus obtain
\begin{equation}\label{7chi_Tp}
\chi^\star T_p=(T_{q}+\varepsilon T_{A\circ\chi})\chi^\star+\R_\Con(p,\chi),
\end{equation}
where the symbol $q$ is given by
\begin{equation}\label{q(x,xi)_1}
q(x,\xi)=i\Big(\omega+\varepsilon Y\circ\chi(x)\Big)\cdot\big(\chi'(x)^{-\T}\xi\big)
\in\Gamma^1_{s-\tau-1-n/2}.
\end{equation}
By Proposition \ref{Findeta}, it shall equal to $i(\omega+h(\omega,\varepsilon Y))\cdot\xi$ if $\omega+h(\omega,\varepsilon Y)\in\Dio[\gamma,\tau]$. 

We next look for a matrix-valued function to eliminate the operator $\varepsilon T_{A\circ\chi}$. Since $\chi\in H^{s-\tau}_\odd$ and $A\in H^{s-1}_\odd$, we find
$A\circ\chi\in H^{s-\tau}_\odd$, and
$$
\|A\circ\chi\|_{H^{s-\tau}}\leq C_s,
\quad s\geq s_1+2\tau+2.
$$
We assume in addition 
\begin{equation}\label{7_s_size}
s\geq \max(s_1+(2\tau+2),s_0+(3\tau+2)).
\end{equation}
If we have further $\varepsilon\ll\varepsilon_0\gamma^2$, Theorem \ref{MatRed} immediately gives the existence of a matrix-valued function 
$$
U=U\big(\omega+h(\omega,\varepsilon Y),\varepsilon A\circ\chi\big)\in H^{s-2\tau}_\even,
$$
such that if $\omega+h(\omega,\varepsilon Y)\in\Dio[\gamma,\tau]$, then 
\begin{equation}\label{7UEqn}
\big((\omega+h(\omega,\varepsilon Y))\cdot\nabla_x\big)U
+\varepsilon(A\circ\chi)\cdot(I_N+U)=0.
\end{equation}

With this matrix-valued function $U$ given, we compute (with the aid of Theorem \ref{2theo:sc})
\begin{equation}\label{7UConj}
\begin{aligned}
(T_q+\varepsilon T_{A\circ\chi})T_{I_N+U}
&=T_{I_N+U}T_q+\varepsilon \R_\CM^{\gau}(A\circ\chi,U)\\
&\quad+[T_q,T_{I_N+U}]
+\varepsilon\Op^\PM\big((A\circ\chi)\cdot(I_N+U)\big).
\end{aligned}
\end{equation}
If $\omega+h(\omega,\varepsilon Y)\in\Dio[\gamma,\tau]$, then $T_q=(\omega+h(\omega,\varepsilon Y))\cdot\nabla_x$, so by (\ref{7UEqn}) the operator 
$$
\begin{aligned}
[T_q,T_{I_N+U}]
&+\varepsilon\Op^\PM\big((A\circ\chi)\cdot(I_N+U)\big)\\
&=\Op^\PM\left[\big((\omega+h(\omega,\varepsilon Y))\cdot\nabla_x\big)U
+\varepsilon(A\circ\chi)\cdot(I_N+U)\right]
\end{aligned}
$$
in (\ref{7UConj}) shall vanish. Summarizing the computations above, we conclude the following reducibility result for the linear paradifferential operator $T_p=\big(\omega+\varepsilon T_{Y}\big)\cdot\nabla_x+\varepsilon T_A$:

\begin{proposition}\label{7_Lin_Red}
Fix $s$ as in (\ref{7_s_size}). Suppose $Y\in H^s_\even(\xT^n;\xR^n)$ and $A\in H^{s-1}_\odd\otimes\mathbf{M}_N(\xR)$ are of norm $\leq1$ in the function spaces $H^{s_1},H^{s_0}$ respectively. Suppose $\varepsilon$ satisfies the smallness requirement as in Theorem \ref{MatRed} and \ref{6VectRed}, namely $\varepsilon\leq\min(\varepsilon_0,\varepsilon_1)\gamma^2$. If the shifted frequency $\omega+h(\omega,\varepsilon Y)$, given by Theorem \ref{MatRed}, is in the Diophantine set $\Dio[\gamma,\tau]$, then with 
\begin{equation}\label{7_chi_U}
\chi=\chi(\omega,\varepsilon Y)\in H^{s-\tau}_\odd,
\quad
U=U\big(\omega+h(\omega,\varepsilon Y),\varepsilon A\circ\chi\big)\in H^{s-2\tau}_\even,
\end{equation}
given by Theorem \ref{MatRed} and Theorem \ref{6VectRed}, we have
\begin{equation}\label{7_Lin_Red_Eq}
\begin{aligned}
\chi^\star T_p
&=T_{I_N+U}\left[\big(\omega+h(\omega,\varepsilon Y)\big)\cdot\nabla_x\right] T_{I_N+U}^{-1}\chi^\star\\
&\quad
+\varepsilon \R_\CM^{\gau}(A\circ\chi,U)T_{I_N+U}^{-1}\chi^\star
+\R_\Con(p,\chi)
\end{aligned}
\end{equation}
\end{proposition}

The paradifferential remainders $\R^\gau_\CM$, $\R_\Con$ in (\ref{7_Lin_Red_Eq}) are ``very smoothing": by Theorem \ref{2theo:sc} and Theorem \ref{4PCConj}, they gain nearly $s$ derivatives back given the regularity of $Y,\chi$ and $A$. Therefore, conjugation with $T_{I_N+U}^{-1}\chi^\star$ really brings $T_p$ to the simpler form $\big(\omega+h(\omega,\varepsilon Y)\big)\cdot\nabla_x$ (modulo smoothing remainder) under the Diophantine condition $\omega+h(\omega,\varepsilon Y)\in\Dio[\gamma,\tau]$.

\subsection{The New Unknown}
In the previous subsection, we did the computation for the paradifferential operator $T_p$ with \emph{general} $Y$ and $A$. Now we are at the place to let $Y$ and $A$ depend on $u\in H^s_\even$, with $s$ as in (\ref{7_s_size}). Namely, we put $Y(x):=X(x,u(x))$, and
$$
\begin{aligned}
A(x)=F'_z\big(x,u(x)\big)
+\sum_{j=1}^n\big(\partial_ju(x)\big)\otimes X_{j;z}'\big(x,u(x)\big)
\in H^{s-1}_{\odd}\otimes\mathbf{M}_{N}(\xR),
\end{aligned}
$$
as in (\ref{7A(x)}). In this case, the diffeomorphism $\chi$ and matrix $U$ in Proposition \ref{7_Lin_Red} will then depend on $u$, while the shift $h$ will be a $\xR^n$-valued function of $u$. From now on, we shall use subscripts 1,2 to denote quantities evaluated at $\omega_{1,2}\in\xR^n$ and $u_{1,2}\in H^s_\even$. 

We first study how $\chi$ and $U$ varies as $\omega,u$ varies. Since $X(x,z)$ is smooth in its arguments, we find that $Y(x)=X(x,u(x))$ has smooth dependence on $u\in H^s_\even$. Using the Lipschitz dependence part of Theorem \ref{6VectRed}, we conclude the following:

\begin{proposition}\label{Findeta}
Fix $u \in H^s_\even(\xT^n,\xR^N)$ with $\lA u\rA_{H^s}\leq 1$. Write $Y(x):=X(x,u(x))$, which is an even $H^s$ vector field on $\xT^n$. Suppose $\varepsilon$ satisfies the following smallness condition in low norm:
\begin{equation}\label{EtaSmall}
\forall u\in \bar{\mathcal{B}}^{s_1+2\tau+1}_{\varepsilon_1\gamma^2,\even},\quad \varepsilon\|X(x,u(x))\|_{H^{s_1+2\tau+1}}\leq\varepsilon_1\gamma^2,
\end{equation}
where $\varepsilon_1$ is as in Theorem \ref{6VectRed}. Then, for all $(\omega,u)\in \xR^n\times\bar{\mathcal{B}}^{s_1+2\tau+1}_{\varepsilon_1\gamma^2,\even}$, the odd diffeomorphism $\eta=\eta(\omega,\varepsilon Y)$ has regularity
\begin{equation}\label{Addieta}
\|\eta(\omega,\varepsilon Y)-\Id\|_{H^{s-\tau}}
\leq C_s\varepsilon\gamma^{-1},
\end{equation}
and satisfies Lipschitz estimates: for $s_1+2\tau+2\leq \alpha\leq s$,
\begin{equation}\label{Liph&eta}
\begin{aligned}
|\Delta_{12}h(\omega,\varepsilon Y)|
&\leq C_{s_1}\varepsilon_1|\Delta_{12}\omega|+C_{s_1}\varepsilon\gamma^{-1}\|\Delta_{12}u\|_{H^{s_1+2\tau+1}};\\
\|\Delta_{12}\eta(\omega,\varepsilon Y)\|_{H^{\alpha-(2\tau+1)}}
&\leq C_{s}\varepsilon\gamma^{-2}\big(|\Delta_{12}\omega|+\|\Delta_{12}u\|_{H^{\alpha}}\big).
\end{aligned}
\end{equation}
\end{proposition}

Therefore, by (\ref{Addieta}), $\chi=\eta^{-1}$ satisfies
\begin{equation}\label{Addichi}
\|\chi-\Id\|_{H^{s-\tau}}
\leq C_s\varepsilon\gamma^{-1}.
\end{equation}
By (\ref{Liph&eta}), $\chi$ has Lipschitz dependence on $(\omega,u)\in\xR^n\times\big(\bar{\mathcal{B}}^{s_1+2\tau+1}_{\varepsilon_1\gamma^2}\cap\bar{\mathcal{B}}^s_1\big)$ at the price of losing more derivatives: for $s_1+2\tau+2\leq \alpha\leq s$,
\begin{equation}\label{Lipchi}
\|\Delta_{12}\chi\|_{H^{\alpha-(2\tau+2)}}\leq C_{s}\varepsilon\gamma^{-2}\big(|\Delta_{12}\omega|+\|\Delta_{12}u\|_{H^{\alpha}}\big).
\end{equation}
Note that the loss is $2\tau+2$ instead of $2\tau+1$ because inverting a diffeomorphism loses one more derivative. By (\ref{Lipchi}), $A\circ\chi$ has Lipschitz dependence on $(\omega,u)$ at the price of losing more derivatives: for $s_1+2\tau+2\leq \alpha\leq s$,
\begin{equation}\label{AchiLip}
\|\Delta_{12}A\circ\chi\|_{H^{\alpha-(2\tau+2)}}
\leq C_s\big(\varepsilon\gamma^{-2}|\Delta_{12}\omega|+\|\Delta_{12}u\|_{H^{\alpha}}\big).
\end{equation}
Using the Lipschitz dependence part of Theorem \ref{MatRed}, we conclude the following:
\begin{proposition}\label{FindU}
Fix $u \in H^s_\even(\xT^n,\xR^N)$ with $\lA u\rA_{H^s}\leq 1$. Still write $Y(x)=X(x,u(x))$, an even $H^s$ vector field on $\xT^n$. Let $\chi=\chi(\omega,\varepsilon Y)=\eta(\omega,\varepsilon Y)^{-1}$, with $\eta$ as in Proposition \ref{Findeta}. In addition to (\ref{EtaSmall}), assume the following smallness in low norm:
\begin{equation}\label{SmallU}
\varepsilon\|A\circ\chi\|_{H^{s_0+\tau}}\leq\varepsilon_0\gamma^2
\quad \forall (\omega,u)\in\xR^n\times\bar{\mathcal{B}}_{\varepsilon_0\gamma^2,\even}^{s_1+2\tau+1},
\end{equation}
where $\varepsilon_0$ is as in Theorem \ref{MatRed}. Then the matrix-valued function $U=U(\omega,\varepsilon A\circ\chi)\in H^{s_0}_\even$, with $(\omega,u)\in\xR^n\times \bar{\mathcal{B}}_{\varepsilon_0\gamma^2,\even}^{s_1+2\tau+1}$, is of regularity $H^{s-2\tau}$:
\begin{equation}\label{7UHs}
\|U(\omega,\varepsilon A\circ\chi)\|_{H^{s-2\tau}}
\leq C_s\varepsilon\gamma^{-1},
\end{equation}
Furthermore, for $s_0+(3\tau+2)\leq \alpha\leq s$, there holds the Lipschitz estimate
\begin{equation}\label{7UHsLip}
\begin{aligned}
\|\Delta_{12}U\|_{H^{\alpha-(3\tau+2)}}
&\leq C_s\varepsilon\gamma^{-2}\big(|\Delta_{12}\omega|+\|\Delta_{12}u\|_{H^{\alpha}}\big).
\end{aligned}
\end{equation}
\end{proposition}

With the dependence of $\chi,U$ on $\omega,u$ clarified, we can now define a new unknown $y:=T_{I_N+U}^{-1}\chi^\star u$. We show that this is a local homeomorphic correspondence.

\begin{lemma}\label{7NewVar}
Suppose 
$$
\begin{aligned}
s\geq \max(s_1+2\tau+2,s_0+3\tau+2)+\delta
=:s_2+\delta.
\end{aligned}
$$
Let $\varepsilon$ satisfy the smallness assumptions of Proposition \ref{Findeta} and \ref{FindU}, and in addition $\varepsilon\ll_s\gamma^2$. Let $\chi=\eta^{-1}$ and $U$ be as in Proposition \ref{Findeta} and \ref{FindU}. Then the mapping 
$$
\Psi:(\omega,u)\to T_{I_N+U}^{-1}\chi^\star u,\quad \xR^n\times\big(\bar{\mathcal{B}}_{\varepsilon_0\gamma^2,\even}^{s_1+2\tau+1}
\cap\bar{\mathcal{B}}^s_{1}\big)
\to H^s_{\even}
$$
takes an even function $u$ to an even function. Furthermore, for any fixed $\omega\in\xR^n$, the mapping $\bar{\mathcal{B}}_{\varepsilon_0\gamma^2,\even}^{s_1+2\tau+1}
\cap\bar{\mathcal{B}}^s_{1}\ni u\to\Psi(\omega,u)$ is a homeomorphism with image covering $\bar{\mathcal{B}}^{s_1+2\tau+1}_{\varepsilon,\even}\cap\bar{\mathcal{B}}^{s}_{1/2}$, satisfying
\begin{equation}\label{LipNewVar}
\begin{aligned}
\|\Delta_{12}\Psi(\omega,u)\|_{H^{s-1}}
&\lesssim_{s}|\Delta_{12}\omega|+\|\Delta_{12}u\|_{H^s},\\
\|\Delta_{12}u\|_{H^{s-1}}
&\lesssim_{s}\|\Delta_{12}\Psi(\omega,u)\|_{H^s}+|\Delta_{12}\omega|.
\end{aligned}
\end{equation}
\end{lemma}
\begin{proof}
The parity of $\Psi(\omega,u)$ is easy to check since all the operations involved preserve parity. It is also easy to check $\Psi(\omega,u)\in H^s$ by the boundedness of paraproduct and paracomposition.

Given that $\varepsilon\ll_s\gamma^2$, the Lipschitz dependence of paracomposition gives the Lipschitz estimate
$$
\|\Delta_{12}\Psi(\omega,u)\|_{H^{s-1}}
\overset{(\ref{Lipchi})}{\lesssim_{s}}|\Delta_{12}\omega|+\|\Delta_{12}u\|_{H^s}.
$$
As for the reverse inequality, the Lipschitz dependence of paracomposition together with Proposition \ref{Findeta}-\ref{FindU} gives
\begin{equation}\label{y1-y2(s)}
\begin{aligned}
\|\Delta_{12}\Psi(\omega,u)\|_{H^{s}}
&\geq \|\Delta_{12}\Psi(\omega,u)\|_{H^{s-1}}\\
&\geq -C_{s}\big(|\Delta_{12}U|_{L^\infty}+|\Delta_{12}\chi|_{C^2}\big)\|u_2\|_{H^s}+
C_{s}\|\Delta_{12}u\|_{H^{s-1}}\\
&\overset{(\ref{Lipchi})(\ref{7UHsLip})}{\geq} C_{s}\big(1-C_{s}\varepsilon\gamma^{-2}\big)\|\Delta_{12}u\|_{H^{s-1}}
-C_s\varepsilon\gamma^{-2}|\Delta_{12}\omega|
\end{aligned}
\end{equation}
Provided $C_{s}\varepsilon\ll\gamma^{2}$, this implies the desired Lipschitz estimate. 

Let us then prove that for fixed $\omega\in\xR^n$, the mapping $\bar{\mathcal{B}}_{\varepsilon_0\gamma^2,\even}^{s_1+2\tau+1}
\cap\bar{\mathcal{B}}^s_{1}\ni u\to\Psi(\omega,u)$ is a homeomorphism with image covering the set $\bar{\mathcal{B}}^{s_1+2\tau+1}_{\varepsilon,\even}\cap\bar{\mathcal{B}}^{s}_{1/2}$. In fact, by Proposition \ref{PCNormCont}, the mapping $u\to\Psi(\omega,u)$ is continuous under the $H^s$ norm topology, and by (\ref{LipNewVar}) it is injective. So it suffices to show that it is surjective and the inverse is continuous. This amounts to solving the equation $\Psi(\omega,u)=y$ for $y\in\bar{\mathcal{B}}^{s_1+2\tau+1}_{\varepsilon,\even}\cap\bar{\mathcal{B}}^{s}_{1/2}$.

However, the equation $\Psi(\omega,u)=y$ can be written in a fixed point form
\begin{equation}\label{Schauder}
u=(\chi^\star)^{-1}T_{I_N+U}y.
\end{equation}
By (\ref{Addichi})(\ref{7UHs}), if $\varepsilon\ll_s\gamma^2$ suitably, the operator norms of $\chi^\star$ and $T_{I_N+U}$ could be estimated as $\leq 1+C_s\varepsilon\gamma^{-1}$ (and do not depend on $\omega$). Therefore, for $y\in\bar{\mathcal{B}}^{s_1+2\tau+1}_{\varepsilon,\even}\cap\bar{\mathcal{B}}^{s}_{1/2}$, we find that the right-hand side of (\ref{Schauder}) defines a continuous mapping from the ball $\bar{\mathcal{B}}_{\varepsilon_1\gamma^2,\even}^{s_1+2\tau+1}$ to itself, and the image is in $H^s$, hence has compact closure within $\bar{\mathcal{B}}_{\varepsilon_1\gamma^2,\even}^{s_1+2\tau+1}$. By the Schauder fixed point theorem, the equation (\ref{Schauder}) has a solution $u\in\bar{\mathcal{B}}_{\varepsilon_1\gamma^2,\even}^{s_1+2\tau+1}$. By (\ref{LipNewVar}) the solution is unique, and since $\|y\|_{H^s}\leq 1/2$, we conclude by norm estimates of $\chi^\star$ and $T_{I_N+U}$ that $\|u\|_{H^s}\leq1$. As for continuity of the solution $u$ in terms of $y$, we find by (\ref{LipNewVar}) that when $y$ varies slightly in $H^s$, then $u$ varies slightly in $H^{s-1}$. By the continuous dependence of $\chi$ and $U$ on $u$, we conclude that $u$ in fact varies slightly in $H^s$. This finishes the proof that $u\to\Psi(\omega,u)$ is a homeomorphism.
\end{proof}

\subsection{Paradifferential Reducibility}
We are now at the place to state the core result of Part \ref{Part2}, namely the \emph{paradifferential reducibility} of the equation (\ref{7ParalinEQ}). 

Let us suppose $u \in H^s_\even(\xT^n,\xR^N)$ with $\lA u\rA_{H^s}\leq 1$ solves (\ref{7ParalinEQ}). We rewrite this equation as
\begin{equation}\label{7ParalinEQ1}
T_pu+\varepsilon\R(X,F;u)=\varepsilon f,
\quad
p(x,\xi)=i\big(\omega+\varepsilon Y(x)\big)\cdot\xi+\varepsilon A(x),
\end{equation}
where $Y(x)=X(x,u(x))$ and $A$ is as in (\ref{7A(x)}). Applying $\chi^\star$, using Proposition \ref{7_Lin_Red}, with the new unknown $y=T_{I_N+U}^{-1}\chi^\star u$ given by Proposition \ref{7NewVar}, we obtain \emph{paradifferential reducibility}: suppose the shifted frequency $\omega+h(\omega,\varepsilon Y)\in\Dio[\gamma,\tau]$, then $(\omega,u)$ satisfies (\ref{7ParalinEQ1}) if and only if
\begin{equation}\label{7Paray}
\begin{aligned}
\left[\big(\omega+h(\omega,\varepsilon Y)\big)\cdot\nabla_x\right]y
&=\varepsilon T_{I_N+U}^{-1}\chi^\star f+\mathfrak{F}(\omega,y),
\end{aligned}
\end{equation}
with the remainder
\begin{equation}\label{F(w,y)}
\begin{aligned}
\mathfrak{F}(\omega,y)
&:=T_{I_N+U}^{-1}\big(-\varepsilon\chi^\star\R(X,F;u)-\R_\Con(p,\chi)u\big)
-\varepsilon T_{I_N+U}^{-1}\R_\CM^{\gau}(A\circ\chi,U)y.
\end{aligned}
\end{equation}
In the above equalities, the quantities $Y,A,p,\chi,U$ are all evaluated at $(\omega,u)$. Thus $\mathfrak{F}(\omega,y)$ is defined in terms of an implicit mapping of $y$. 

By Lemma \ref{Dio_Ext}, if $\omega+h(\omega,\varepsilon Y)\in\Dio[\gamma,\tau]$, then 
$\bm{L}_{\gamma,\tau}\big(\omega+h(\omega,\varepsilon Y)\big)=\big((\omega+h(\omega,\varepsilon Y))\cdot\nabla_x\big)^{-1}$. The right-hand-side of (\ref{7Paray}) is obviously odd. Therefore, we obtain the following:
\begin{proposition}\label{ParaRedu}
Let $X(x,z),F(x,z)$ be given as in Theorem \ref{1NLPHyp}. Fix some $\tau>n-1$, $\gamma\in(0,1]$, $\delta>0$, still write
$$
s_1=2\tau+2+n+\delta,\quad s_0=2\tau+1+n/2+\delta.
$$ 
Fix an index 
$$
\begin{aligned}
s&\geq \max(s_1+2\tau+2,s_0+3\tau+2)+\delta.
\end{aligned}
$$
Let the Fourier multiplier $\bm{L}_{\gamma,\tau}$ be as in Lemma \ref{Dio_Ext}. Suppose $\omega\in\xR^n$, $y\in\bar{\mathcal{B}}^{s_1+2\tau+1}_{\varepsilon,\even}\cap\bar{\mathcal{B}}^{s}_{1/2}$ solve
\begin{equation}\label{7ParalinEQ2}
y=\bm{L}_{\gamma,\tau}\big(\omega+h(\omega,\varepsilon Y)\big)\left[\varepsilon T_{I_N+U}^{-1}\chi^\star f+\mathfrak{F}(\omega,y)\right].
\end{equation}
Here $\varepsilon$ satisfies the smallness conditions in Proposition \ref{Findeta}-\ref{7NewVar}; the vector field $Y(x)=X(x,u(x))$; the shift $h$ is as in Theorem \ref{6VectRed}; $u$ and $y$ are related by $y=\Psi(\omega,u)$; $\chi=\eta^{-1},U$ as in Proposition \ref{Findeta}-\ref{FindU} are evaluated at $(\omega,u)$.

If $\omega+h(\omega,\varepsilon Y)\in\Dio[\gamma,\tau]$, then in fact $\omega,u$ solve the paradifferential equation (\ref{7ParalinEQ}).
\end{proposition}

It remains to justify that the remainder $\mathfrak{F}(\omega,y)$ in (\ref{F(w,y)}) is indeed very smoothing.

\begin{proposition}\label{Lip_F(w,y)}
Fix some $\tau>n-1$, $\gamma\in(0,1]$, $\delta>0$, and still write
$$
s_1=2\tau+2+n+\delta,\quad s_0=2\tau+1+n/2+\delta.
$$ 
Fix an index 
$$
\begin{aligned}
s&\geq \max(s_1+2\tau+2,s_0+3\tau+2)+\delta.
\end{aligned}
$$
Assume the smallness conditions for $\varepsilon$ in Proposition \ref{Findeta}-\ref{7NewVar}. The nonlinear mapping $\mathfrak{F}$ enjoys the following regularity gain: for $\omega_{1,2}\in\xR^n$ and  $y_{1,2}\in\bar{\mathcal{B}}^{s_1+2\tau+1}_{\varepsilon,\even}\cap\bar{\mathcal{B}}^{s}_{1/2}$, there holds
\begin{equation}\label{F(y)gain}
\begin{aligned}
\big\|\mathfrak{F}(\omega,y);H^{2s-(2\tau+2)-n/2-\delta}\big\|
&\leq C_s\varepsilon\gamma^{-1}\\
\big\|\Delta_{12}\mathfrak{F}(\omega,y);H^{2s-(3\tau+6)-n/2-\delta}\big\|
&\leq C_s\varepsilon\gamma^{-2}\big(\|\Delta_{12}y\|_{H^s}
+|\Delta_{12}\omega|\big).
\end{aligned}
\end{equation}
Here $\Delta_{12}$ stands for the difference of the quantity under concern when evaluated at $(\omega_i,u_i)$.
\end{proposition}
\begin{proof}
We first estimate the norm of $\mathfrak{F}(\omega,y)$. Recalling $\R(X,F;u)\in H^{2s-n/2-1}$ from the beginning of this section, we obtain (see Theorem \ref{2ParaLin})
$$
\|\varepsilon T_{I_N+U}^{-1}\chi^\star\R(X,F;u)\|_{H^{2s-n/2-1}}\leq C_s\varepsilon.
$$
As for the remainder $\R_\Con(p,\chi)$, by Theorem \ref{4PCConj} and the estimate (\ref{Addichi}), the remainder $\R_\Con(p,\chi)$ gains back $s-(\tau+3)-n/2-\delta$ derivatives: for all $\sigma\in\xR,\,s\geq s_1+2\tau+2$,
$$
\|\R_\Con(p,\chi)g;H^{\sigma+s-(\tau+3)-n/2-\delta}\|
\lesssim_{s,\sigma} \varepsilon\gamma^{-1}\|g\|_{H^\sigma}.
$$
Therefore,
$$
\big\|T_{I_N+U}^{-1}\R_\Con(p,\chi)u;H^{2s-(\tau+3)-n/2-\delta}\big\|
\leq C_s\varepsilon\gamma^{-1}.
$$
Finally, recall that $A\circ\chi\in H^{s-\tau}$, $U\in H^{s-2\tau}$, so
$$
\big\|\varepsilon T_{I_N+U}^{-1}\R_\CM^{\gau}(A\circ\chi,U)y;H^{2s-2\tau-n/2}\big\|
\overset{(\ref{7UHs})}{\leq} C_s\varepsilon^2\gamma^{-1}.
$$
Summing up, we find the regularity gain of $\mathfrak{F}(\omega,y)$ is at least $s-(2\tau+2)-n/2-\delta$, hence the first inequality of (\ref{F(y)gain}). 

As for the difference $\mathfrak{F}(\omega_1,y_1)-\mathfrak{F}(\omega_2,y_2)$ in (\ref{F(y)gain}), we should try to control the regularity gain of it in terms of $\|u_1-u_2\|_{H^{s-1}}$, since by (\ref{LipNewVar}), this is the best that we can control using $\|y_1-y_2\|_{H^s}$. But this is not a difficulty at all since the paradifferential remainders are nearly twice as regular as $u$. In fact, with $u_1,u_2\in H^{s-1}$, we have (see Theorem \ref{2ParaLin})
$$
\|\Delta_{12}\R(X,F;u)\|_{H^{2s-n/2-2}}\leq C_s\|\Delta_{12}u\|_{H^{s-1}}.
$$
Since $\Delta_{12}\chi^\star$ loses one derivative while $\Delta_{12}T_{I_N+U}$ does not, we obtain
$$
\|\varepsilon \Delta_{12}T_{I_N+U}\chi^\star\R(X,F;u)\|_{H^{2s-n/2-3}}
\overset{(\ref{Lipchi})(\ref{7UHsLip})}{\leq} C_s\varepsilon\|\Delta_{12}u\|_{H^{s-1}}
+C_s\varepsilon^2\gamma^{-2}|\Delta_{12}\omega|.
$$
Furthermore, by Theorem \ref{4PCConj} and (\ref{Lipchi}), $\R_\Con(p,\chi)$ is Lipschitz in $(\omega,u)\in\xR^n\times\bar{\mathcal{B}}^s_1$ at the price of losing some smoothing, again since the coefficients of $p(x,\xi)$ are much smoother than $\chi$: for all $\sigma\in\xR,\,s\geq \alpha\geq s_1+2\tau+2$,
$$
\begin{aligned}
\big\|\Delta_{12}\R_\Con(p,\chi)g;H^{\sigma+\alpha-(2\tau+5)-n/2-\delta}\big\|
\lesssim_{s,\sigma} \varepsilon\gamma^{-2}\big(|\Delta_{12}\omega|+\|\Delta_{12}u\|_{H^{\alpha}}\big)\|g\|_{H^\sigma}.
\end{aligned}
$$
Setting $g=u$, $\alpha=\sigma=s-1$ in the inequality above, we obtain
$$
\big\|\Delta_{12}T_{I_N+U}^{-1}\R_\Con(p,\chi)u;H^{2s-(2\tau+7)-n/2-\delta}\big\|
\overset{\text{Theorem }\ref{4PCConj}}{\leq} C_s\varepsilon\gamma^{-2}\big(|\Delta_{12}\omega|+\|\Delta_{12}u\|_{H^{s-1}}\big).
$$
Setting $\alpha=s-1$ in (\ref{AchiLip})(\ref{7UHsLip}), considering $y_i$ as in $H^{s-1}$, we obtain
$$
\big\|\varepsilon \Delta_{12}T_{I_N+U}^{-1}\R_\CM^{\gau}(A\circ\chi,U)y;H^{2s-(3\tau+3)-n/2}\big\|
\overset{(\ref{esti:quant2})}{\leq} C_s\varepsilon\gamma^{-1}\big(|\Delta_{12}\omega|+\|\Delta_{12}u\|_{H^{s-1}}\big).
$$
Summing these up, we obtain the second inequality in (\ref{F(y)gain}):
$$
\begin{aligned}
\big\|\Delta_{12}\mathfrak{F}(\omega,y);H^{2s-(3\tau+6)-n/2-\delta}\big\|
&\leq C_s\varepsilon\gamma^{-2}\big(|\Delta_{12}\omega|+\|\Delta_{12}u\|_{H^{s-1}}\big)\\
&\overset{(\ref{LipNewVar})}{\leq} C_s\varepsilon\gamma^{-2}\big(|\Delta_{12}\omega|+\|\Delta_{12}y\|_{H^{s}}\big).
\end{aligned}
$$
\end{proof}

\subsection{Banach Fixed Point Argument: Proof of Theorem \ref{1NLPEQ}.}\label{sec:7.3}

Proposition \ref{ParaRedu} suggests a method to solve the paradifferential equation (\ref{7ParalinEQ}): it suffices to solve the Banach fixed point equation (\ref{7ParalinEQ2}), and then pick out those $\omega$ making $\omega+h(\omega,\varepsilon Y)$ Diophantine.

Given $X(x,z),F(x,z)$ as in Theorem \ref{1NLPEQ}, corresponding to the equation (\ref{7ParalinEQ}), let us define a mapping
\begin{equation}\label{EQ_Fix}
\mathscr{P}(\omega,y)=\bm{L}_{\gamma,\tau}\big(\omega+h(\omega,\varepsilon Y)\big)\left[\varepsilon T_{I_N+U}^{-1}\chi^\star f+\mathfrak{F}(\omega,y)\right]
\quad\text{for }
\omega\in\xR^n,\,y\in\bar{\mathcal{B}}^{s_1+2\tau+1}_{\varepsilon,\even}\cap\bar{\mathcal{B}}^{s}_{1/2},
\end{equation}
where the Fourier multiplier $\bm{L}_{\gamma,\tau}$ is as in Lemma \ref{Dio_Ext}, and
\begin{itemize}
    \item The vector field $Y(x)=X(x,u(x))$;
    \item $\chi,h,U$, as in Proposition \ref{7_Lin_Red}, are all evaluated at $(\omega,u)$;
    \item $u$ and $y$ are related by $y=\Psi(\omega,u)$ as in Proposition \ref{7NewVar};
    \item $\mathfrak{F}(\omega,y)$ is as in (\ref{F(w,y)}) and estimated by Proposition \ref{Lip_F(w,y)}.
\end{itemize}
We assume all the smallness assumptions in Proposition \ref{Findeta}-\ref{7NewVar}, so that $\Psi(\omega,u)$ gives a homeomorphic correspondence between $u\in\bar{\mathcal{B}}_{\varepsilon_0\gamma^2,\even}^{s_1+2\tau+1}
\cap\bar{\mathcal{B}}^s_{1}$ and a range containing all $y\in\bar{\mathcal{B}}^{s_1+2\tau+1}_{\varepsilon,\even}\cap\bar{\mathcal{B}}^{s}_{1/2}$, as indicated by Proposition \ref{7NewVar}. 

\noindent
\textbf{Step 1: Contraction $y\to\mathscr{P}(\omega,y)$.}

We first fix the frequency $\omega\in\xR^n$ and show that $y\to\mathscr{P}(\omega,y)$ is a contraction with a uniform ratio. This amounts to combining all the Lipschitz estimates obtained above.

So let us estimate $\mathscr{P}(\omega,y_1)-\mathscr{P}(\omega,y_2)$. Let us use subscripts 1,2 to denote quantities evaluated at $y_1,y_2$ with fixed $\omega$ in the sequel. Let us split 
$\mathscr{P}(\omega,y_1)-\mathscr{P}(\omega,y_2)=\mathrm{I}+\mathrm{II}$, where
$$
\begin{aligned}
\mathrm{I}&=\left[\bm{L}_{\gamma,\tau}\big(\omega+h(\omega,\varepsilon Y_1)\big)-\bm{L}_{\gamma,\tau}\big(\omega+h(\omega,\varepsilon Y_2)\big)\right]
\big(\varepsilon T_{I_N+U_{2}}^{-1}\chi_{2}^\star f+\mathfrak{F}(\omega,y_{2})\big), \\
\mathrm{II}&=\bm{L}_{\gamma,\tau}\big(\omega+h(\omega,\varepsilon Y_2)\big)\big(\varepsilon T_{I_N+U_{1}}^{-1}\chi_{1}^\star f-\varepsilon T_{I_N+U_{2}}^{-1}\chi_{2}^\star f
+\mathfrak{F}(\omega,y_{1})-\mathfrak{F}(\omega,y_{2})\big).
\end{aligned}
$$
It follows from the assumption on the index $s$ in Proposition \ref{ParaRedu} that
\begin{equation}\label{Size(s)}
s-(3\tau+2)-n/2-\delta\geq2\tau+1,
\end{equation}
so that by Proposition \ref{Lip_F(w,y)}, the regularity gain of $\mathfrak{F}(\omega,y)$ will always be sufficient to balance the loss caused by small denominators.

To estimate I, we notice that $\bm{L}_{\gamma,\tau}\big(\omega+h(\omega,\varepsilon Y_1)\big)-\bm{L}_{\gamma,\tau}\big(\omega+h(\omega,\varepsilon Y_2)\big)$ is a Fourier multiplier losing $2\tau+1$ derivatives: by Lemma \ref{Dio_Ext},
$$
\begin{aligned}
\left\|\left[\bm{L}_{\gamma,\tau}\big(\omega+h(\omega,\varepsilon Y_1)\big)
-\bm{L}_{\gamma,\tau}\big(\omega+h(\omega,\varepsilon Y_2)\big)\right]g\right\|_{H^{\sigma}}
&\leq \gamma^{-2}|h(\omega,\varepsilon Y_1)-h(\omega,\varepsilon Y_2)|\|g\|_{H^{\sigma+2\tau+1}}\\
&\overset{(\ref{Liph&eta})}{\leq} C_{s_1,\sigma}\varepsilon\gamma^{-3}\|u_{1}-u_2\|_{H^{s_1+2\tau+1}}\|g\|_{H^{\sigma+2\tau+1}}\\
&\overset{\eqref{LipNewVar}}{\leq} C_{s_1,\sigma}\varepsilon\gamma^{-3}\|y_{1}-y_2\|_{H^{s}}\|g\|_{H^{\sigma+2\tau+1}}.
\end{aligned}
$$
Substituting this estimate into the expression of I, we obtain 
\begin{equation}\label{I}
\begin{aligned}
\|\,\mathrm{I}\,\|_{H^s}
&\leq C_s\varepsilon\gamma^{-3}\|y_{1}-y_2\|_{H^{s}}\left(\varepsilon\|f\|_{H^{s+2\tau+1}}+\big\|\mathfrak{F}(\omega,y_{2});H^{s+2\tau+1}\big\|\right)\\
&\overset{(\ref{Size(s)})}{\leq}
C_s\varepsilon\gamma^{-3}\|y_{1}-y_2\|_{H^{s}}\left(\varepsilon\|f\|_{H^{s+2\tau+1}}+\big\|\mathfrak{F}(\omega,y_{2});H^{2s-(2\tau+2)-n/2-\delta}\big\|\right)\\
&\overset{(\ref{F(y)gain})}{\leq} C_s\varepsilon^2\gamma^{-4}\big(1+\|f\|_{H^{s+2\tau+1}}\big)\|y_{1}-y_2\|_{H^{s}}.
\end{aligned}
\end{equation}

The estimate for II employs the second inequality of (\ref{F(y)gain}). Recalling that $\chi_{1}^\star-\chi^\star_2$ loses one derivative while $T_{U_{1}}-T_{U_2}$ does not (and the operator norms of these are all controlled by $\varepsilon\gamma^{-2}\|u_{1}-u_2\|_{H^{s-1}}\ll\|u_{1}-u_2\|_{H^{s-1}}$ by (\ref{Lipchi}) and (\ref{7UHsLip})), we obtain 
\begin{equation}\label{II}
\begin{aligned}
\|\mathrm{II}\|_{H^s}
&\leq C_s\gamma^{-1}\big(\varepsilon\|u_{1}-u_2\|_{H^{s-1}}\|f\|_{H^{s+\tau+1}}+\|\mathfrak{F}(\omega,y_{1})-\mathfrak{F}(\omega,y_{2})\|_{H^{s+\tau}}\big)\\
&\overset{(\ref{Size(s)})}{\leq} 
C_s\gamma^{-1}\Big(\varepsilon\|u_{1}-u_2\|_{H^{s-1}}\|f\|_{H^{s+\tau+1}}\\
&\quad+\big\|\mathfrak{F}(\omega,y_{1})-\mathfrak{F}(\omega,y_2);H^{2s-(3\tau+6)-n/2-\delta}\big\|\Big)\\
&\overset{(\ref{LipNewVar})(\ref{F(y)gain})}{\leq}
C_s\varepsilon\gamma^{-3}\big(1+\|f\|_{H^{s+\tau+1}}\big)\|y_{1}-y_2\|_{H^{s}}.
\end{aligned}
\end{equation}

Adding (\ref{I})(\ref{II}), we end up with (recalling $\varepsilon\gamma^{-2}\ll_s1$)
$$
\|\mathscr{P}(\omega,y_1)-\mathscr{P}(\omega,y_2)\|_{H^s}
\leq C_s\varepsilon\gamma^{-3}\big(1+\|f\|_{H^{s+2\tau+2}}\big)\|y_{1}-y_2\|_{H^{s}}.
$$
We can thus choose $\varepsilon\ll_{s,f}\gamma^3$ to make the constant $C_s\varepsilon\gamma^{-3}\big(1+\|f\|_{H^{s+2\tau+2}}\big)\leq1/2$, so that $y\to\mathscr{P}(\omega,y)$ is a contraction with ratio 1/2. The reasoning also shows that $y\to\mathscr{P}(\omega,y)$ takes the domain $\bar{\mathcal{B}}^{s_1+2\tau+1}_{\varepsilon,\even}\cap\bar{\mathcal{B}}^{s}_{1/2}$ into itself since $\mathscr{P}(\omega,0)=\bm{L}_{\gamma,\tau}(\omega)(\varepsilon f)$ has size $\simeq \varepsilon\gamma^{-1}\ll1$.

\noindent
\textbf{Step 2: Lipschitz dependence $\omega\to\mathscr{P}(\omega,y)$.}

This time we use subscripts 1,2 to denote quantities evaluated at $\omega_1,\omega_2$ with fixed $y$ in the sequel. We split $\mathscr{P}(\omega_1,y)-\mathscr{P}(\omega_2,y)=\mathrm{III}+\mathrm{IV}$, where
$$
\begin{aligned}
\mathrm{III}&=
\left[\bm{L}_{\gamma,\tau}\big(\omega_1+h(\omega_1,\varepsilon Y)\big)-\bm{L}_{\gamma,\tau}\big(\omega_2+h(\omega_2,\varepsilon Y)\big)\right]
\big(\varepsilon T_{I_N+U_1}^{-1}\chi_1^\star f+\mathfrak{F}(\omega_1,y_1)\big),\\
\mathrm{IV}&=\bm{L}_{\gamma,\tau}\big(\omega_2+h(\omega_2,\varepsilon Y)\big)
\big(\varepsilon T_{I_N+U_1}^{-1}\chi_1^\star f-\varepsilon T_{I_N+U_2}^{-1}\chi_2^\star f+\mathfrak{F}(\omega_1,y)-\mathfrak{F}(\omega_2,y)\big).
\end{aligned}
$$

Similarly as in estimating I, we use Lemma \ref{Dio_Ext} to compute
$$
\begin{aligned}
\big\|\big[\bm{L}_{\gamma,\tau}\big(\omega_1+h(\omega_1,\varepsilon Y)\big)
&-\bm{L}_{\gamma,\tau}\big(\omega_2+h(\omega_2,\varepsilon Y)\big)\big]g\big\|_{H^{\sigma}}\\
&\leq C_{s_1,\sigma}\gamma^{-2}\big(|\omega_1-\omega_2|
+|h(\omega_1,\varepsilon Y)-h(\omega_2,\varepsilon Y)|\big)\|g\|_{H^{\sigma+2\tau+1}}\\
&\overset{(\ref{Liph&eta})}{\leq} C_{s_1,\sigma}\gamma^{-2}|\omega_1-\omega_2|\cdot\|g\|_{H^{\sigma+2\tau+1}}.
\end{aligned}
$$
Just as in estimating (\ref{I}) in the previous step, using the first inequality of (\ref{F(y)gain}) again, we find
\begin{equation}\label{III}
\|\,\mathrm{III}\,\|_{H^s}\leq C_s\varepsilon\gamma^{-3}\big(1+\|f\|_{H^{s+2\tau+1}}\big)|\omega_1-\omega_2|.
\end{equation}

To estimate IV, we use the second inequality of (\ref{F(y)gain}) again. Just as in estimating term II (see (\ref{II})) in the previous step, we know that $\chi^\star_1-\chi^\star_2$ loses one derivative while $T_{U_1}-T_{U_2}$ does not. So using the second inequality in (\ref{F(y)gain}), we obtain
\begin{equation}\label{IV}
\|\mathrm{IV}\|_{H^s}\leq C_s\varepsilon\gamma^{-3}\big(1+\|f\|_{H^{s+\tau+1}}\big)
|\omega_1-\omega_2|.
\end{equation}

Adding (\ref{III})(\ref{IV}), we find that
$$
\|\mathscr{P}(\omega_1,y)-\mathscr{P}(\omega_2,y)\|_{H^{s}}
\leq C_s\varepsilon\gamma^{-3}\big(1+\|f\|_{H^{s+\tau+1}}\big)|\omega_1-\omega_2|.
$$
We can still choose $\varepsilon$ as in the last step to make  $C_s\varepsilon\gamma^{-3}\big(1+\|f\|_{H^{s+\tau+1}}\big)\leq1/2$, so that 
$$
\|\mathscr{P}(\omega_1,y)-\mathscr{P}(\omega_2,y)\|_{H^{s}}
\leq \frac{1}{2}|\omega_1-\omega_2|.
$$

In summary, we showed that the mapping $\mathscr{P}(\omega,y)$ is a contraction for $y\in\bar{\mathcal{B}}^{s_1+2\tau+1}_{\varepsilon,\even}\cap\bar{\mathcal{B}}^{s}_{1/2}$ with ratio $1/2$, and is Lipschitz in $\omega$ with constant $1/2$. By Theorem \ref{Depend}, we conclude that $\mathscr{P}(\omega,y)$ has a unique fixed point $y=y(\omega)\in\bar{\mathcal{B}}^{s_1+2\tau+1}_{\varepsilon,\even}\cap\bar{\mathcal{B}}^{s}_{1/2}$, which is 1-Lipschitz continuous in $\omega\in\xR^n$. Inverting the change-of-unknown $y=\Psi(\omega,u)$, we obtain a mapping $u=u(\omega)$ from $\xR^n$ to $H^s$. Using the second inequality of (\ref{LipNewVar}), we obtain
$$
\big\|u(\omega_1)-u(\omega_2)\big\|_{H^{s-1}}\leq C_s|\omega_1-\omega_2|.
$$

\noindent
\textbf{Step 3: Measure estimate; concluding the proof.}

With the mapping $u(\omega)$ defined as in the last paragraph of Step 2, We then make sure that there are ``plenty of" $\omega$ to make $\omega+h(\omega,\varepsilon Y)\in\Dio[\gamma,\tau]$, so that by Proposition \ref{ParaRedu}, the solution $y=y(\omega)$ of (\ref{7ParalinEQ2}) indeed yields a solution $u=u(\omega)$ of (\ref{7ParalinEQ}). With a little abuse of notation, we shall just write this shifted frequency as $\omega+h(\omega,u(\omega))$.

The frequency shift $h(\omega,u(\omega))$, as seen in Proposition \ref{Findeta}, relies only on the low $H^{s_1+2\tau+1}$ norm of $u(\omega)$, so the mapping $\xR^n\ni\omega\to h(\omega,u(\omega))$ is in fact Lipschitz. By Proposition \ref{Findeta} again and the last paragraph of Step 2, we know the Lipschitz constant is bounded by $C_{s_1}(\varepsilon_1+\varepsilon\gamma^{-1})$. Note that $\varepsilon\ll_s\gamma^{3}$, and since $s$ is a fixed index, it is of no harm to further shrink the magnitude of $\varepsilon_1$, yielding that the Lipschitz constant of $\omega\to h(\omega,u(\omega))$ is $<1/2$. Therefore, the set
\begin{equation}\label{O_Set}
\mathfrak{O}:=\big\{\omega\in\xR^n:\, \omega+h(\omega,u(\omega))\in\Dio[\gamma,\tau]\big\}
\end{equation}
is a Lipschitz distortion of $\Dio[\gamma,\tau]$. It is well-known that for $\tau>n-1$, the measure of $\bar B(0,R)\setminus\Dio[\gamma,\tau]$ is $\leq C_n\gamma R^n$ (see for example, Subsection 2.1 of Chapter III of \cite{AG}). Therefore, by classical results of Lebesgue measure under Lipschitz distortion, we conclude
$$
\Big|\bar B(0,R)\setminus\big\{\omega\in\xR^n:\, \omega+h(\omega,u(\omega))\in\Dio[\gamma,\tau]\big\}\big|
\leq C_n\gamma R^n.
$$

To summarize, in solving the fixed point equation (\ref{7ParalinEQ2}), we assumed $\varepsilon\ll\gamma^3$, in particular
$$
\varepsilon\gamma^{-3}\big(1+\|f\|_{H^{s+2\tau+2}}\big)\ll_s1.
$$
Note that the index $s$ is any real number greater than $\max(s_1+2\tau+2,s_0+3\tau+2)+\delta\simeq 5\tau$, and we did not explicitly indicate the dependence of these constants on $X,F$. Consequently, if we choose $\gamma=\varepsilon^{a}$ with $0<a<1/3$, the smallness conditions are all satisfied when $\varepsilon<\varepsilon_*(f,X,F)$. Thus, the set $\mathfrak{O}$ depends only on $\varepsilon$. The measure $\big|\bar{B}(0,R)\setminus\mathfrak{O}\big|\leq C_{n,a}R^n\varepsilon^a$ as $\varepsilon\to0$. For $\omega\in\mathfrak{O}$, the function $u=u(\omega)\in H^s$ really satisfies (\ref{7ParalinEQ}), hence (\ref{1NLPEQ}), ensured by Proposition \ref{ParaRedu}.

\appendix
\section{Banach Fixed Point Theorem}\label{App_Lip}
In this appendix, we state a refined version of the classical Banach fixed point theorem, concerning parameter dependence of the fixed point. The argument is elementary but does not seem to be recorded explicitly in any textbook to the authors' knowledge.
\begin{theorem}\label{Depend}
Let $(X,d)$ and $(\Omega,\rho)$ be complete metric spaces. Let $0<q<1$ and $L>0$ be fixed constants. Let $\mathscr{G}:X\times\Omega\to X$ be a mapping such that
$$
\begin{aligned}
d\big(\mathscr{G}(x,\mu),\mathscr{G}(y,\mu)\big)&\leq qd(x,y),\quad x,y\in X,\,\mu\in\Omega\\
d\big(\mathscr{G}(x,\lambda),\mathscr{G}(x,\mu)\big)&\leq L\rho(\lambda,\mu),\quad x\in X,\,\lambda,\mu\in\Omega.
\end{aligned}
$$
Then for every $\mu\in\Omega$, the mapping $x\to\mathscr{G}(x,\mu)$ has a unique fixed point $f(\mu)\in X$, which is Lipschitz continuous in $\mu$:
$$
d\big(f(\lambda),f(\mu)\big)\leq \frac{L}{1-q}\rho(\lambda,\mu).
$$
\end{theorem}
\begin{proof}
The existence and uniqueness of $f(\mu)$ is the content of the standard Banach fixed point theorem. To obtain Lipschitz continuity, we simply compute, with $x=f(\lambda)$, $y=f(\mu)$, 
$$
\begin{aligned}
d(x,y)
&= d\big(\mathscr{G}(x,\lambda),\mathscr{G}(y,\mu)\big)\\
&\leq d\big(\mathscr{G}(x,\lambda),\mathscr{G}(y,\lambda)\big)
+d\big(\mathscr{G}(y,\lambda),\mathscr{G}(y,\mu)\big)\\
&\leq qd(x,y)+L\rho(\lambda,\mu).
\end{aligned}
$$
We thus solve $d(x,y)$ from this inequality and obtain the desired inequality.
\end{proof}

\end{spacing}

\bibliographystyle{abbrv}
\bibliography{References}

\end{document}